\documentclass[11pt]{amsart}
\usepackage{amsmath}
\usepackage{amssymb}
\usepackage{amsthm}
\usepackage{amscd}
\usepackage{epsfig}
\usepackage{epic}
\usepackage{eepic}
\usepackage{graphics}
\usepackage{graphicx}
\usepackage{psfrag}

\usepackage{url}
\usepackage{cite}

\usepackage[margin=1in]{geometry}

\newif\ifdebug                                                      %
\debugfalse

\newcommand{\printname}[1]
   {\smash{\makebox[0pt]{\hspace{-1.0in}\raisebox{8pt}{\tiny #1}}}}
\newcommand{\Label}[1] {\ifdebug {\label{#1}\printname{#1}}
                        \else    {\label{#1}} \fi}

\numberwithin{equation}{section}
\numberwithin{figure}{section}

\newtheorem{Theorem}[equation]{Theorem}
\newtheorem{Lemma}[equation]{Lemma}
\newtheorem{Proposition}[equation]{Proposition}

\newtheorem{Question}[equation]{Question}
\theoremstyle{remark}
\newtheorem{Remark}[equation]{Remark}
\newtheorem{Example}[equation]{Example}
\theoremstyle{corollary}
\newtheorem{Corollary}[equation]{Corollary}
\theoremstyle{definition}
\newtheorem{Definition}[equation]{Definition}
\newtheorem*{Lemma*}{Lemma}

\def \new {\text{new}}

\input xy
\xyoption{all}
\UseComputerModernTips

\def \g {{\mathfrak g}}
\def \t {{\mathfrak t}}

\def \R {{\mathbb R}}
\def \C {{\mathbb C}}
\def \CP {{\mathbb C}{\mathbb P}}
\newcommand{\Z}{{\mathbb Z}}
\newcommand{\N}{{\mathbb N}} 
\def \del {{\partial}}
\def \half {\frac{1}{2}}
\def \Inv {^{-1}}
\def \ol {\overline}
\def \ssminus {\smallsetminus}
\def \eps {\epsilon}
\def \veps {\varepsilon}

\def \DH {Duistermaat--Heckman\ }
\def \tDH {\text{DH}}

\newcommand{\ddr}{\frac{\partial}{\partial r}}
\newcommand{\deldelt}{\frac{\partial}{\partial t}}

\def \calK {\mathcal K}

\def \calI {\mathcal I}

\def \bfv {{v^\sharp}}
\def \tM {\widetilde{M}}
\def \tomega {{\widetilde{\omega}}}
\def \tPhi {\widetilde{\Phi}}

\def \rhobar {\ol{\rho}}

\def \Cinf {C^{\infty}}
\def \teta {\widetilde{\eta}}
\def \tA {\widetilde{A}}

\def \std {\text{std}}

\DeclareMathOperator \Vect {Vect}
\DeclareMathOperator \Diff {Diff}
\DeclareMathOperator \supp {supp}
\DeclareMathOperator \Ad {Ad}
\DeclareMathOperator \germ{germ}
\DeclareMathOperator \Lie {Lie}
\DeclareMathOperator \image {image}
\DeclareMathOperator \Crit {Crit}

\DeclareMathOperator \relint{rel-int}
\DeclareMathOperator \LDO{LDO}

\DeclareMathOperator \AGL {AGL}

\newcommand{\hs}{{\hspace{3mm}}}
\newcommand{\hsm}{{\hspace{1mm}}}

\begin{document}

\title[Localization through cobordisms]
{Localization for equivariant cohomology with varying polarization}

\author{Megumi Harada}
\address{Department of Mathematics and Statistics, McMaster University,
Hamilton, Ontario L8S 4K1  Canada }
\email{Megumi.Harada@math.mcmaster.ca}
\urladdr{\url{http://www.math.mcmaster.ca/Megumi.Harada/}}
\thanks{Both authors are partially supported by NSERC Discovery Grants. 
The first author is additionally partially supported by
an NSERC University Faculty Award and an Ontario Ministry of Research
and Innovation Early Researcher Award.}

\author{Yael Karshon}
\address{Department of Mathematics, University of Toronto,
Toronto Ontario M5S 2E4 Canada }
\email{karshon@math.toronto.edu}
\urladdr{\url{http://www.math.toronto.edu/~karshon}} 


\maketitle

\begin{center}
\textit{Dedicated to the memory of Hans Duistermaat} 
\end{center} 


\begin{abstract} 
The main contribution of this paper is a generalization of 
several previous localization theories 
in equivariant symplectic geometry, including the
classical Atiyah-Bott/Berline-Vergne localization theorem, as well as
many cases of
the localization via the norm-square of the momentum map as initiated
and developed by Witten, Paradan, and Woodward. Our version unifies
and generalizes these theories by using noncompact cobordisms as in
previous work of Guillemin, Ginzburg, and Karshon, and by introducing
a more flexible notion of `polarization' than in previous theories. 
Our localization
formulas are also valid for closed $2$-forms $\omega$ that may be
degenerate. As a corollary, we
are able to answer a question posed some time ago by Shlomo Sternberg
concerning the classical Brianchon-Gram polytope decomposition,.
We illustrate our theory using concrete examples motivated by our
answer to Sternberg's question. 
\end{abstract}

\setcounter{tocdepth}{1} 
\tableofcontents

\section{Introduction}\label{sec:introduction}

The main results of this manuscript, Theorems~\ref{theorem:GLS-cobordism}
and~\ref{theorem:twisted-GLS-cobordism}, 
are part of a long chain of localization results in topology. 
Here, by a \emph{localization result} we mean a formula that expresses a
global topological or geometric quantity on a manifold $M$ as a sum of
local contributions near a subset of $M$ such as the fixed point set of a
torus action or of a diffeomorphism, the zero set of a vector field,
or the critical set of a function. 
This idea has a long history; early results of this general nature
are, for instance, the Poincar\'e-Hopf 
index theorem,  the Lefschetz fixed point theorem,
and the early works \cite{Bott1,Bott2} of Bott.
 Other results that relate to our current work
 are \cite{DuiHec82, DuiHec83, BerVer:1982, BerVer:1983, AtiBot:1984, GLS, 
 JefKir:1995, 
           Witten:1992, Par00, Woo05, GGK}. 
To place our results in the appropriate context, below we give a very brief sketch of this circle of ideas, 
focusing on aspects that directly relate to our results.

We begin by discussing the classical Duistermaat-Heckman 
exact stationary phase formula, which is probably the first such 
localization result in modern symplectic geometry. 
We will go back and forth between two different perspectives
that are frequently encountered in the literature 
and which are related by a Fourier transform. 
Let $(M,\omega)$ be a compact symplectic manifold, equipped with an
action of a torus $G$ with associated momentum map 
$\Phi \colon M \to \g^*$. 
In their original papers~\cite{DuiHec82,DuiHec83}, Duistermaat and Heckman
consider the oscillatory integral
\begin{equation} \label{DH integral}
 \int_M \frac{\omega^n}{n!} e^{i\langle \Phi, X\rangle} 
\end{equation}
over $M$ of the function $e^{i\langle \Phi,X\rangle}$ times the Liouville
measure, where $\langle \Phi, X\rangle$ denotes the component of the
momentum map for $X \in \g$. They then prove an \emph{exact stationary
phase formula} which expresses this integral as the sum of
local contributions near the fixed points of the $G$-action. 

Recall that \emph{Liouville measure} on $M$ is obtained by integration
of $\omega^n/n!$, the symplectic volume form.
The \emph{Duistermaat--Heckman} measure on $\g^*$,
which we denote $\tDH_{(M,\omega,\Phi)}$, is the push-forward
of Liouville measure via the momentum map $\Phi \colon M \to \g^*$.
The integral~\eqref{DH integral} is essentially the Fourier transform, denoted
$\widehat{\tDH}_{(M,\omega,\Phi)}$, of the \DH measure.  This can
be seen from the following computation:   for $X \in \g$,
\begin{equation}\label{eq:Fourier transform}
  \begin{split}
    \widehat{\tDH}_{(M,\omega,\Phi)}(-X) & = \int_{\xi \in \g^*}
    \tDH_{(M,\omega,\Phi)} e^{ i \langle \xi, X \rangle} \\
 & = \int_M \frac{\omega^n}{n!} e^{ i \langle \Phi, X \rangle},
 \end{split}
\end{equation}
where the first equality follows from the definition of Fourier transform and 
the second equality follows from the definition of the \DH measure.

The integral~\eqref{DH integral} can also be
interpreted as a push-forward in equivariant cohomology. With this
interpretation, the exact stationary phase formula becomes a special
case of a localization formula in equivariant cohomology. This
important formula was observed by Berline and Vergne and by Atiyah and
Bott, and hence it is often 
referred to as the ``ABBV formula". For details
see~\cite{BerVer:1982, BerVer:1983, AtiBot:1984}.

It is possible to relax some of the assumptions on the manifold $M$
and the $2$-form $\omega$ and still have versions of the exact stationary
phase formula, as we now describe. 
One such relaxation is to allow the closed $2$-form $\omega$
to be degenerate.  The manifold $M$ must then be equipped with an orientation,
so that we can still integrate the symplectic volume form. 
If $M$ is compact 
then the \DH measure remains well-defined as a signed measure.
The exact stationary phase formula continues to hold in this generality;
indeed, it still follows from the ABBV formula.

If the manifold $M$ is not compact, but the momentum map $\Phi$ is
proper, then the \DH measure can be defined as a distribution, even if
$\omega$ is degenerate. In this paper we mainly work with this
\emph{\DH distribution} (see\ Definition~\ref{DH distribution}).
In this situation, there might not exist a
localization formula to the fixed point set; there might not even
be any fixed points. (This can be seen, for instance, from the example
of the action of a torus $G$ on its cotangent bundle $T^*G \cong G
\times \g^*$ with the momentum map being the projection to $\g^*$.)
However if we additionally assume that $\Phi$ has a component that is
proper and bounded from below, then there does exist a localization
formula which expresses the \DH distribution as a sum of contributions
given in terms of infinitesimal data along the components of the
$G$-fixed point set.
Guillemin, Lerman, and Sternberg derived a formula of this form 
for the DH measure when $M$ is compact and $G$ is a torus acting 
with isolated fixed points \cite{GLS}.  
Their proof uses an (inverse) Fourier transform applied
to the left and right hand sides of the original DH formula 
for~\eqref{DH integral}.  
The case of non-isolated fixed points is worked
out by Guillemin and Cannas da Silva in \cite{GuiCdS96};
the case that a component of the momentum map is proper and bounded
from below is analyzed by Prato and Wu~\cite{PraWu94}.

With these historical remarks in hand, we can describe the motivations
and main contributions of the present manuscript. 
One of our major motivations was to further 
develop the point of view (initiated 
in~\cite{GGK1,K} and developed in~\cite{GGK}) that it is
possible to derive the above-mentioned localization formulae
\cite{GLS, GuiCdS96} \emph{directly}, without passing through Fourier
transforms. The main technique for doing so is 
an appropriate notion of noncompact cobordism, under
the hypothesis that a component of the momentum map is proper and
bounded below. Indeed, using cobordisms, one can prove a Guillemin-Lerman-Sternberg-type formula 
in the more general
situation of non-isolated fixed points and non-compact $M$ 
(\cite[Sec.~11]{K}, \cite[Chap.~4, Sec.~6]{GGK}).

Here we take a moment to sketch some of the history of 
cobordisms in symplectic geometry. 
The idea originates from the work of 
Guillemin and Sternberg, who observe in \cite{GS82-preprint} that two
compact Hamiltonian symplectic manifolds have the same
Duistermaat-Heckman measure if they are Hamiltonianly cobordant
through a compact manifold $W$. This is a simple consequence of
Stokes' theorem. 
Shaun Martin used cobordisms in the context of equivariant localization
in \cite{Mar00}. 
Guillemin, Ginzburg, and the second author, in \cite{GGK,K},
allow for {\em non-compact} Hamiltonian manifolds and use noncompact 
cobordisms to derive localization results.  
In this setting, to deal with the lack of compactness, 
one restricts attention to momentum maps that are proper.  Moreover,
one often
makes the stricter requirement that the momentum map is ``polarized",
i.e., it has a component that is proper and bounded from
below. This important technical condition on momentum maps 
first appeared in the work of Prato and Wu \cite{PraWu94}.

Our work generalizes and extends the cobordism ideas in \cite{GGK} 
by introducing the notions of a \textbf{taming map} and a
\textbf{$v$-polarized momentum map}, as we now explain. The basic idea
is that 
we allow the component $v$ of
the momentum map $\Phi$ -- with respect to which the function $\left<
  \Phi , v \right>$ must be proper and bounded from below
(i.e., ``polarized'') -- to \emph{vary} along the manifold. 
More precisely, suppose
$M$ is a $G$-manifold. We fix a \textit{taming map} $v \colon M \to
\g$, so named because we use it to control situations when the
manifold is not compact. Then we require the function obtained by
pairing the $\g^*$-valued momentum map with this taming map to be
proper and bounded from below, giving us the notion of a
\textit{$v$-polarized momentum map}. In the special case that $v
\equiv \eta \in \g$ is constant, we recover the \emph{$\eta$-polarized}
condition of~\cite{GGK}. Our notion of a taming map is motivated by
work of Maxim Braverman \cite{Bra02}, although our assumptions on the map
$v$ slightly differ from his.

We now describe in more detail the main results of this paper. 
Let $G$ be a compact Lie group,
$M$ an even-dimensional oriented $G$-manifold, $\omega$ a
$G$-invariant closed $2$-form (not necessarily symplectic), and $\Phi:
M \to \g^*$ a $G$-equivariant function such that Hamilton's equation 
\[
d\Phi^X = \iota(X^\sharp)\omega
\]
holds for all $X \in \g$. (We call such a triple a \textbf{Hamiltonian
  $G$-manifold}, although -- contrary to the standard use of the term
in the literature -- we do not require $(M, \omega)$ to be
symplectic.) 
A taming map $v$ determines a vector field $v^\sharp$ on $M$ via the
infinitesimal action of the Lie algebra $\g$ on $M$. The localizing
set in our theory is the zero set $Z := \{v^\sharp = 0\} \subseteq M$
of this vector field. 
The main results of our manuscript,
Theorems~\ref{theorem:GLS-cobordism}
and~\ref{theorem:twisted-GLS-cobordism}, express the 
Duistermaat-Heckman distribution of the Hamiltonian $G$-manifold 
$(M,\Phi,\omega)$ (respectively twisted
Duistermaat-Heckman distribution) in terms of local data on
arbitrarily small neighborhoods of $Z$. Namely, 
let \(Z = \bigcup_{i \in \mathcal{I}} Z_i\) 
be the decomposition of $Z$ into its connected components. 
Then our Theorem~\ref{theorem:GLS-cobordism} takes the form
\begin{equation}\label{eq:main-intro}
\tDH_{(M, \Phi, \omega)} = 
\sum_i \tDH_{\germ_{Z_i}(M,\omega,\Phi)}^v
\end{equation}
and, in the twisted case, our Theorem~\ref{theorem:twisted-GLS-cobordism}
takes the form
\begin{equation} \label{eq: main twisted intro}
\tDH_{(M,\omega,\Phi)}(A) = \sum_i
\tDH_{\germ_{Z_i}(M,\omega,\Phi)}^v(A \vert_{Z_i}) 
\end{equation}
where $A$ is an equivariant cohomology class on $M$. The $i$th summand
on the right hand side is defined to be the Duistermaat-Heckman
distribution (respectively twisted Duistermaat-Heckman distribution)
of a $v$-\textbf{polarized completion} relative to $Z_i$ 
(made precise in Definition~\ref{definition:v-polarized-completion}) 
of the restriction of $(M,\omega,\Phi)$ to an arbitrarily small 
$G$-invariant neighbourhood of the connected component $Z_i$ of $Z$. 
The notation $\tDH_{\germ_{Z_i}(M,\omega,\Phi)}^v$ 
(respectively $\tDH_{\germ_{Z_i}(M,\omega,\Phi)}^v(A \vert_{Z_i})$), 
which makes no mention of the choice of a neighbourhood 
or of a $v$-polarized completion, is justified since we prove 
in Section~\ref{sec:DH} (respectively Section~\ref{sec:twistedDH}) 
that under appropriate hypotheses this \DH distribution 
(respectively twisted Duistermaat-Heckman distribution)
is in fact \emph{independent of these choices}.  In this sense the right hand sides
of~\eqref{eq:main-intro} and~\eqref{eq: main twisted intro} 
depend only on data that is localized near $Z$, and  
thus it is valid to view our results as a ``localization to $Z$''.  
We emphasize that the specific form of the right hand side 
varies according to choice of $v$ (and hence $Z$), so actually 
each of~\eqref{eq:main-intro} and~\eqref{eq: main twisted intro}
is a {\em family} of formulas; we explicitly demonstrate this 
using a simple example in Section~\ref{sec:sphere}.

Next we relate our main results to another related circle of ideas, namely
that of the so-called nonabelian localization and localization with
respect to the norm-square of the momentum map. We begin with a brief
historical account. 
The `nonabelian localization' theory was initiated by Witten, who
considers integrals of the form 
 \begin{equation} \label{Witten integral}
\int_{\g} e^{-\varepsilon \|X\|^2} \int_M \eta(X) e^{\omega + i \Phi} dX
 \end{equation}
 where $\eta$ is an equivariant differential form, $dX$ is a volume
 form on $\g$, the $\varepsilon$ is a real parameter, and $\Phi$ again
 denotes the momentum map for a Hamiltonian $G$-action
 (see \cite[p.311]{Witten:1992} and \cite[Section 2]{Liu:1995}). 
 The group $G$ may be nonabelian.
 The integral~\eqref{Witten integral}
 can be interpreted as the evaluation on the Gaussian 
 $e^{-\varepsilon \|X\|^2}$ of a twisted DH distribution
 $\widehat{\tDH}_{(M,\omega,\Phi)}(\eta)$ on $\g$.
 (Mild assumptions on $M$ guarantee that this distribution is a temperate 
 (a.k.a tempered) distribution,
 so its evaluation on the Gaussian is well defined.)
 Witten then gives a formula similar in spirit to the abelian ABBV
 localization formula mentioned above; he expresses 
 \eqref{Witten integral}
 as a sum of local contributions from the components of the critical
 set of $\|\Phi\|^2$ , where the dominant contribution as
 $\varepsilon \to \infty$ is from the absolute minimum
 $\Phi^{-1}(0)$. (In fact Witten's theory is slightly more general: he begins with a localization formula
which depends on a choice of closed invariant 1-form $\lambda$ on $M$
and leads to a sum of local contributions from the components
of the set 
\begin{equation}\label{eq:Witten localizing set}
\left\{ x \ | \ \langle \lambda , \xi^\sharp \rangle = 0 \right\}
\end{equation}
for appropriate vector fields $\xi^\sharp$. 
He then specializes to the case $\lambda = d \| \Phi \|^2$,
for which the set~\eqref{eq:Witten localizing set} is $\Crit \| \Phi
\|^2$.)

Witten's results may be considered an extension of the abelian
Duistermaat-Heckman theory, in that it firstly introduces more 
general integrands (which correspond to the twisted Duistermaat-Heckman
distributions of Section~\ref{sec:twistedDH}), and secondly it 
produces formulas which localize to $\Crit(\|\Phi\|^2)$, the components
of the critical set of the norm-square of $\Phi$, instead of the fixed
points of the action.  Some years later, Jeffrey and Kirwan derived
 similar formulas by working through the maximal torus $T$ of a
 compact nonabelian Lie group $G$; they also explain the relation of
 their formula to that of Witten's in \cite{JefKir:1995}.  Kefeng Liu
 gives simplified proofs of some of these results in
 \cite{Liu:1995}. More recently,
 Paradan \cite{Par00} and Woodward \cite{Woo05} develop a localization
 theory for the norm-square $\|\Phi\|^2$ of the momentum map which
 incorporates the Witten integrals above, deriving formulas for (the
 Fourier transforms of) the twisted Duistermaat--Heckman distributions
 as sums of local contributions associated to components of the
 critical set of $\|\Phi\|^2$.

Our main theorems also generalize 
(the inverse Fourier transforms of) the nonabelian localization
formulas to $\Crit(\|\Phi\|^2)$, 
under the assumption that the connected components of $\Crit(\|\Phi\|^2)$
are smooth.  (We expect this assumption to not be necessary;
see Remark~\ref{technical hypothesis}.)
Namely, we
express the (twisted) \DH distribution as a sum of local contributions. 
Our formulas rely on a choice of taming map $v \colon M \to \g$,
which can be obtained from a choice of a real valued function
$\rho \colon \g^* \to \R$
by~\eqref{eq:v-drho}.   The localizing set is then the critical set
of the composition $\rho \circ \Phi$. 
When $G$ is a torus and $\rho$ is a linear functional, 
as described in our Example~\ref{example:Prato-Wu},
our Theorem~\ref{theorem:GLS-cobordism} 
recovers the so-called GLS formula \cite{GLS}.
When $\rho$ is the norm square function $\|\cdot\|^2 \colon \g^* \to \R$, 
as described in our in Example~\ref{example:norm-square}, 
we recover the localizing set $\Crit(\|\Phi\|^2)$ of Witten, Paradan,
 and Woodward \cite{Witten:1992, Par00, Woo05}.

Our results also answer a question asked some time ago by Shlomo
Sternberg concerning the Brianchon-Gram polytope decomposition. 
It is known that the Atiyah-Bott-Berline-Verge localization
theorem in equivariant cohomology \cite{BerVer:1982, BerVer:1983,
AtiBot:1984}, when applied to the exponent of the equivariant
symplectic form of a compact symplectic toric manifold, yields
the measure-theoretic version of the Lawrence-Varchenko polytope
decomposition \cite{Law91, Var87}, 
when applied to the corresponding momentum polytope. Sternberg had
asked whether there is a similar `localization-theoretic'
interpretation of the classical Brianchon-Gram polytope
decomposition. 
We can answer Sternberg's question in the affirmative: 
a special case of Theorem~\ref{theorem:GLS-cobordism}, applied to
the exponent of the equivariant symplectic form on a compact
symplectic toric manifold, yields (the measure-theoretic
version of) the Brianchon-Gram polytope
decomposition for the momentum polytope. 
However, this application requires the
fact 
that, for any simple polytope, there exists a smooth function
with a unique critical point on the relative interior of every face of
that polytope 
(and satisfies some additional technical conditions). 
If the polytope satisfies a technical assumption (recorded in \eqref{assumption}),
then the norm-square function has this property.
Surprisingly, proving that for \emph{any} simple polytope there exists
such a function turned out to be not entirely trivial, and 
our proof (by brute-force differential topology on $\R^n$) 
occupies Appendix~\ref{sec:appendix}.

In future work, it would be interesting to investigate whether
our results generalize to $G$-spaces that are not manifolds. 
(Such a generalization may correspond 
to polytope decompositions for {\em non-simple} polytopes,
such as Haase's generalization of the Lawrence-Varchenko 
decomposition to non-simple polytopes \cite{Haa05}, 
or the Brianchon-Gram decomposition applied to non-simple polytopes.)
It would also be interesting to see if a combination of our results
with Braverman's work in \cite{Bra02} would yield new index formulas.

We now give an outline of the contents of this paper.
In Section~\ref{sec:taming}, we define
the taming map and the corresponding localizing set and
we describe some of our motivating examples.
In Section~\ref{sec:polarized-completions}, we
define $v$-polarized completions and prove that they exist. 
We note that polarized completions are used both in the formulation 
of our localization formulas~\eqref{eq:main-intro} 
and~\eqref{eq: main twisted intro} and in their proof.
We prove the untwisted version of our localization formula,
Theorem~\ref{theorem:GLS-cobordism}, in Section~\ref{sec:DH} and the
twisted version, Theorem~\ref{theorem:twisted-GLS-cobordism}, in
Section~\ref{sec:twistedDH}. In Section~\ref{sec:BrianchonGram} we
discuss the Brianchon-Gram polytope decomposition and answer
Sternberg's question in the affirmative. In 
Section~\ref{sec:sphere}, we use the standard $S^1$-action on $S^2$ in
order to illustrate in detail how our equation~\eqref{eq:main-intro}
can yield different localization formulas by choosing different taming
maps. We prove a technical lemma required in
Section~\ref{sec:BrianchonGram} in Appendix~\ref{sec:appendix}. 

\medskip

\noindent {\bf Acknowledgments.} 
It is our pleasure to thank Shlomo Sternberg and Jonathan Weitsman
for inspiring questions and conversations. 
We thank both Victor Guillemin and Shlomo Sternberg for
their interest and encouragement in this project. 
We thank Brendan McLellan and Lisa Jeffrey for useful
discussions, 
Chris Woodward for explaining his results to us,
and Mich\`ele Vergne for helping to clarify Witten's work. 
We also thank the two anonymous referees, who read our manuscript
very carefully and made many helpful suggestions for improving
exposition.

\section{Taming maps and Hamiltonian spaces}
\Label{sec:taming}

The ``varying polarization" of our localization formulas
is controlled by a so-called taming map.  In this section
we introduce the notion of a taming map, its associated
vector field, and the corresponding notion of a polarization.
For a symplectic manifold with Hamiltonian $G$ action
and momentum map $\Phi \colon M \to \g^*$, we explain that 
a choice of an invariant smooth function $\rho \colon \g^* \to \R$ 
gives rise to a taming map whose localizing set 
is the critical set of the composition $\rho \circ \Phi$.

\subsection{Taming maps and $\mathbf{G}$-manifolds}
\Label{subsec:taming}

Throughout this manuscript, we use $G$ to denote  
a compact Lie group,
$\g$ its Lie algebra, $\g^*$ the dual space of $\g$, and 
$\left< \cdot , \cdot \right> \colon \g^* \times \g \to \R$
the natural pairing between a vector space and its dual.
We equip $\g$ with an $\Ad$-invariant inner product
and $\g^*$ with the induced inner product.
We also fix an $\Ad_G$-invariant inner product on the Lie algebra $\g$, 
and we denote the resulting isomorphism $\g \to \g^*$
by $\xi \mapsto \widehat{\xi}$.

Let $N$ be a manifold, possibly with boundary.\footnote
{In this manuscript, manifolds are $C^\infty$-smooth, Hausdorff, and
   second-countable.} 
A smooth $G$-action on $N$ is a homomorphism $G \to \Diff(N)$
that is smooth in the diffeological sense, i.e., 
the map $(g,x) \mapsto g \cdot x$ is smooth as a map 
from $G \times N$ to $N$.
A manifold $N$ equipped with a $G$-action is called a \textbf{$G$-manifold}.
If $N$ has boundary, then its boundary $\del N$ is a manifold,
any diffeomorphism of $N$ restricts to a diffeomorphism of $\del N$, 
and a smooth $G$-action on $N$ restricts to a smooth $G$-action on $\del N$.

The following simple notion will be crucial in what follows. 

\begin{Definition}\Label{def:taming-map} 
Let $N$ be a $G$-manifold, possibly with boundary.
A \textbf{taming map} on $N$ is a smooth $G$-equivariant function 
\(v \colon N \to \g \).
\end{Definition}

We now associate to a taming map $v \colon N \to \g$
a vector field $v^\sharp$ on $N$.
Recall that a Lie algebra element $X \in \g$ gives rise to a vector field
$X^\sharp \in \Vect(N)$ by
$$ X^\sharp |_x := \left. \frac{d}{dt} \right|_{t=0} (\exp(tX) \cdot x) $$
for all $x \in N$. The resulting map
\begin{equation}\Label{eq:g-VectN}
\g \to \Vect(N) \quad , \quad X \mapsto X^\sharp \quad ,
\end{equation}
is $G$-equivariant with respect to the adjoint action of $G$ on $\g$
and the $G$-action on the space of vector fields $\Vect(N)$ that is induced 
from the $G$-action on~$N$.
Also, if $N$ has boundary, then the restriction
of the vector field $X^\sharp$ to the boundary $\del N$ 
is a vector field on~$\del N$.
Using this notion, we associate to a taming map $v \colon N \to \g$
a vector field $v^{\sharp}$ on $N$ by
\begin{equation}\Label{eq:bfv}
\bfv|_x := v(x)^\sharp |_x \in T_x N 
\end{equation}
for all $x \in N$. 

The zero set of this vector field $\bfv$ will serve 
as the \emph{localizing set} for our
theory, in the sense that our localization formulas give expressions for 
global invariants (namely, Duistermaat-Heckman distributions) 
in terms of data near the localizing sets: 

\begin{Definition}\Label{def:localizing set} 
The \textbf{localizing set} associated to the taming map $v \colon N \to \g$
is defined by 
$$ Z := \{x \in N \ | \  \bfv|_x = 0\} .$$
\end{Definition}

If $N$ is a $G$-manifold with boundary, $v \colon N \to \g$ is a
taming map,
and $x$ is in the boundary of $N$,
then $v^{\sharp}|_x$ is tangent to the boundary
and $(v^{\sharp})|_{\del N} = (v|_{\del N})^{\sharp}$.

The $G$-equivariance of the taming map
implies that the
associated vector field $\bfv$ and localizing sets also behave well
with respect to the $G$-action:

\begin{Lemma} \Label{lemma:v-equivariant}
Let $N$ be a $G$-manifold, possibly with boundary, 
and let $v \colon N \to \g$ be a taming map. 
Then the associated vector field $v^\sharp$ is $G$-equivariant,
and the localizing set $Z = \{ v^\sharp = 0 \}$ is $G$-invariant.
\end{Lemma}

\begin{proof}
For all \(x \in N\),
\begin{align*}
 g_* (\bfv|_x) &= \left( \Ad_g v(x) \right)^\sharp_{g \cdot x}
      & \text{by $G$-equivariance of~\eqref{eq:g-VectN} } \\
   &= \left( v(g \cdot x) \right)^\sharp_{g \cdot x}
      & \text{by $G$-equivariance of $v \colon N \to \g$ } \\
   &=  v^\sharp|_{g \cdot x} & \text{by~\eqref{eq:bfv}} ,
\end{align*}
where $g_* \colon TN \to TN$ denotes the differential of the
diffeomorphism $g \colon N \to N$.
Hence the vector field $\bfv$ is $G$-invariant, as desired. 
The $G$-invariance of $Z$ follows from the equivariance of $v^\sharp$.
\end{proof}

\begin{Remark} \Label{Rk:braverman} 
We are not the first to use the term \emph{taming map}. In \cite{Bra02}
Braverman considers 
a complete Riemannian manifold $N$, equipped with an
  action of a compact Lie group $G$ by isometries. In this setting he
  calls a function \(v \colon N \to \g\) 
a \emph{taming map} if the zero set of the induced vector field is
  compact.  In \cite[Def.~3.2]{Bra02} he requires a cobordism of such
  structures to be consistent in a suitable sense with a choice of
  tubular neighbourhood of the boundary.  Our definitions are slightly
  different from Braverman's in that we do not equip $N$
  with a Riemannian metric, and, more significantly, 
  we do not require the localizing set to be compact.
\end{Remark}

\begin{Remark}
The localizing set $Z$ is not necessarily smooth.
See Remark~\ref{Z not manifold}.
\end{Remark}

As in the book~\cite{GGK}, 
we work with possibly noncompact $G$-manifolds $N$ 
equipped with maps $\Phi \colon N \to \g^*$.
We deal with the non-compactness of $N$
by requiring a properness condition:

\begin{Definition}  \Label{def:vpolarized}
Let $N$ be a $G$-manifold, possibly with boundary, 
and let $v \colon N \to \g$ be a taming map. 
We say that a continuous function
$\Phi \colon N \to \g^*$ is  {\it \textbf{v}}\textbf{-polarized} 
if the function 
$$ \Phi^v := \left< \Phi, v \right> \colon N \to \R $$
is proper and bounded from below.
\end{Definition} 

Definition~\ref{def:vpolarized} is well-suited for our purposes 
for two reasons.  First, being $v$-polarized frequently
implies that the original map $\Phi$ is proper (see\
Lemma~\ref{lemma:polarized-proper} below). Second,
being $v$-polarized is preserved under patchings by a
partition of unity or averaging with respect to compact group actions
(see\ Section~\ref{sec:polarized-completions}).

We make the following purely topological observations:

\begin{Lemma} \Label{lemma:polarized-proper}
Let $N$ be a $G$-manifold, possibly with boundary.
Let $v \colon N \to \g$ 
be a taming map 
and let $\Phi \colon N \to \g^*$ be a continuous function.
\begin{enumerate}
\item
If $N$ is compact, then $\Phi$ is $v$-polarized.
\item
Suppose that $\Phi \colon N \to \g^*$ is $v$-polarized.
Let $Y$ be a subset of $N$. 
Then the restriction to $Y$ of $\Phi$ is $v$-polarized 
if and only if $Y$ is closed in $N$.

\item
Suppose that $v$ is bounded.  Then
\[
\Phi \textup{ is $v$-polarized} \Rightarrow \Phi \textup{ is proper.}
\]
\end{enumerate}
\end{Lemma}

\begin{proof}
Part (1) follows from the fact that
every continuous function on a compact set is bounded and proper.

For Part (2) recall that for a proper map $\psi \colon N \to \R$
on a Hausdorff space $N$ and a subset $Y$ of $N$,
the restriction $\psi|_Y \colon Y \to \R$ is proper
if and only if $Y$ is closed.
Indeed, if $Y$ is closed then $\psi\Inv([a,b]) \cap Y$
is compact for any interval $[a,b]$. This implies $\psi \vert_Y$ is
proper. Now suppose that $\Psi|_Y$ is proper,
and let $x$ be an accumulation point of $Y$.
Let $a$ and $b$ be such that $a < \psi(x) < b$.  
Then $x$ is also an accumulation point of $\Psi\Inv([a,b]) \cap Y$.
Because $\psi \vert_Y$ is proper, $\Psi\Inv([a,b]) \cap Y$ is closed.
So $x$ is in $Y$.
Because $x$ was arbitrary, this shows that $Y$ is closed. Claim (2)
follows.

We now prove part (3).
Choose an inner product on $\g$, and consider the induced inner product
on $\g^*$. Then for any \(x \in N,\) 
\begin{align}
 \left| \left<  \Phi(x) , v(x) \right> \right|
 & \leq \| \Phi(x) \| \cdot \| v(x) \|
 & \text{by the Cauchy-Schwarz inequality} \nonumber \\
    & \leq c \| \Phi(x) \| & \Label{CS} 
\end{align}
where $c := \sup\limits_{x \in N} \| v(x) \| < \infty$
exists because $v$ is bounded by assumption.
Now let $K$ be a compact subset of $\g^*$ .  By \eqref{CS},
$\Phi(x) \in K$ implies that 
$\left| \left< \Phi(x) , v(x) \right> \right| \leq c r$,
where $r = \sup\limits_{\alpha \in K} \| \alpha \|$.
So the $\Phi$-preimage of $K$ is contained in the 
$\left< \Phi , v \right>$-preimage of the interval
$[-cr,cr]$, which is compact because $\left< \Phi , v \right>$
is proper by assumption.  Being a closed subset of a compact set,
$\Phi\Inv(K)$ is also compact, as required. 
\end{proof}

\begin{Remark} \Label{proper not polarized}
The converse of part (3) of Lemma~\ref{lemma:polarized-proper}
is generally false:  a proper map to $\g^*$ need not be $v$-polarized,
even if $v$ is bounded. 
For example, the identity map on $N=\g^*$ is proper but is not $v$-polarized
if $v \colon N \to \g$ is constant.
\end{Remark}

\begin{Remark} \Label{rk:equivalent}
In Sections~\ref{sec:DH} and~\ref{sec:twistedDH}
we derive localization formulas that depend on a choice
of taming map.  However, 
the role played by this choice is quite loose in the sense that many choices of $v$ give the same localization
formulas.
Specifically, we may define two taming maps $v_1$ and $v_2$ 
on a $G$-manifold $N$ to be equivalent
if there exists a $G$-invariant positive function $f \colon N \to \R_{>0}$
such that both $f$ and $1/f$ are bounded
and such that $v_2 = f v_1$.
If $v_1$ and $v_2$ are equivalent taming maps then 
\begin{itemize}
\item[--]
they have the same localizing set;
(see\  Definition~\ref{def:localizing set});
\item[--]
a function $\Phi \colon N \to \g^*$ is $v_1$-polarized
if and only if it is $v_2$-polarized. 
\end{itemize} 
Equivalent bounded taming maps give rise
to the same localization formulas
in Sections~\ref{sec:DH} and~\ref{sec:twistedDH}.
See Remarks~\ref{equivalent maps give same loc}
and~\ref{equivalent maps give same loc twisted}.
\end{Remark}

\begin{Remark} \Label{squish v}
We often require taming maps to be bounded,
because if $v \colon N \to \g$ is bounded
then every $v$-polarized map $N \to \g^*$ is proper 
(by Part (3) of Lemma~\ref{lemma:polarized-proper}).
On the other hand, sometimes it is more natural to begin with a
taming map $v \colon N \to \g$ that is unbounded
(see\ Example~\ref{example:norm-square}).
In this situation, we can replace $v$ with 
a bounded taming map $v^b \colon N \to \g^*$ by defining
$$ v^b(x) := h(\| v(x) \|) \, v(x) $$
where $\| \cdot \|$ is an $\Ad_G$-invariant norm on $\g$
and $h \colon \R_{\geq 0} \to \R_{\geq 0}$ is a smooth function
such that $h(r)=1$ for $r$ near $0$, \ $h(r) = 1/r$ for $r \geq 1$,
and such that the function $r \mapsto h(r) \cdot r$ is weakly monotone.
We can then derive a localization formula using the taming map $v^b$.
Different choices of the function $h$ result in bounded taming maps $v^b$ 
that are equivalent in the sense of Remark~\ref{rk:equivalent}.
Moreover, if $v$ was already bounded, then $v^b$ is also equivalent to $v$.
In this sense we can get a localization formula from \emph{any}
taming map $v \colon N \to \g$.
\end{Remark}

\begin{Remark}
Braverman works with a similar though not identical freedom
in \cite{Bra02}.
He needs his map $v \colon N \to \g$ to be sufficiently large in a
suitable sense. 
He achieves this by multiplying $v$ by a real valued function
that grows sufficiently fast, but his formulas are independent
of the choice of this function.  
\end{Remark}

\begin{Remark} 
The equivalence relation of Remarks~\ref{rk:equivalent}
and~\ref{squish v} is still finer than necessary for our purposes in
the sense that many inequivalent taming maps still give rise to the
same localization formula. For example, suppose that a torus $T$
acts on a compact symplectic manifold with a finite fixed point set. 
Let $\eta$ be an element of the Lie algebra of $T$
whose pairings with all the isotropy weights at all the fixed points
are nonzero.  When the taming map takes the constant value $\eta$,
our localization formula, Theorem~\ref{theorem:GLS-cobordism},
boils down to the Guillemin-Lerman-Sternberg formula \cite{GLS}. 
The right hand side of this formula depends only on the signs
of the pairings of $\eta$ with the isotropy weights.
\end{Remark}

\subsection{Taming maps on Hamiltonian $G$-manifolds}
\Label{subsec:rho-examples}

In this section we focus our attention on Hamiltonian
$G$-manifolds, and we discuss some motivating examples. 
As before, we denote by $\Phi^X$ the
$X$-component of a function $\Phi \colon N \to \g^*$, i.e., 
$\Phi^X(\cdot) := \langle \Phi(\cdot), X\rangle \colon N \to \R$ for
$X \in \g$. 

We note that in our definition of Hamiltonian $G$-manifold
the $2$-form is allowed to be degenerate. 

\begin{Definition}\Label{def:Hamiltonian-space}
Let $N$ be an oriented $G$-manifold, possibly with boundary. 
Let $\omega$ be a $G$-invariant closed $2$-form,
and let $\Phi \colon N \to \g^*$ be a $G$-equivariant function 
such that Hamilton's equation
\begin{equation}\Label{eq:Hamiltons}
 d \Phi^X  = \iota(X^\sharp) \omega
\end{equation}
holds for all \(X \in \g.\) 
Such a triple $(N,\omega,\Phi)$ is called a \textbf{Hamiltonian 
$\mathbf{G}$-manifold},
and the map $\Phi \colon N \to \g^*$ is called a \textbf{momentum map}.
\end{Definition}

In the symplectic geometry literature, the term ``Hamiltonian
$G$-manifold'' is usually reserved for $G$-actions on
\textit{symplectic} manifolds, i.e., the closed $2$-form $\omega$ is
additionally required to be nondegenerate. When the form $\omega$ is
allowed to be degenerate as in our
Definition~\ref{def:Hamiltonian-space}, some authors (e.g.\ Woodward
in \cite[Section 3.1]{Woo05}) call the structure a ``degenerate
Hamiltonian $G$-manifold''. As in the book \cite{GGK} (see, e.g.,
\cite[Chap.~2, \S 1.1]{GGK}), we deviate slightly from this
terminology, for two reasons. First, our localization formulas are
also valid for closed $2$-forms that are somewhere degenerate. Second,
because the derivation of our formulas uses cobordisms, we work with
both even- and odd-dimensional manifolds, and a closed $2$-form on an odd-dimensional manifold is everywhere degenerate.

\begin{Remark}
In Definition~\ref{def:Hamiltonian-space}, if $N$ is a manifold with boundary,
the restrictions of $\omega$ and $\Phi$ to the boundary also satisfy 
Hamilton's equations.
\end{Remark}

Since our definition of Hamiltonian $G$-spaces does not include
assumptions of properness of the momentum map nor nondegeneracy of
the 2-form, we use the following additional terminology:

\begin{Definition} \Label{proper nondegenerate}
We say that a Hamiltonian $G$ manifold $(N,\omega,\Phi)$ is \textbf{proper}
if the momentum map $\Phi \colon N \to \g^*$ is proper;
we say that it is \textbf{nondegenerate} if the closed $2$-form $\omega$ is
nondegenerate and, unless we say otherwise,
the orientation of $N$ is induced from $\omega$.
\end{Definition}

Given a Hamiltonian $G$-space $N$, the main technical idea of our constructions
in the next section is to 
associate to a momentum map $\Phi \colon N \to \g^*$
a taming map $v$ 
such that $\Phi$ is $v$-polarized, and
then to vary the $2$-form and momentum map 
(on $N$ as well as on an appropriate cobording manifold)
while maintaining the taming map that was built
from the original momentum map.
As a first step we now describe a way to obtain a taming map from a momentum map.

We can view the differential of a smooth function
$ \rho \colon \g^* \to \R $
as a function $ d\rho \colon \g^* \to \g$,  
since for any $\alpha \in \g^*$ the differential $d\rho|_\alpha$ at
$\alpha$ is an element of
$\operatorname{Hom}(T_\alpha \g^* , \R) \cong \g$. 
If $\rho$ is $G$-invariant, then $d\rho$ is $G$-equivariant.
Thus, given a $G$-invariant smooth function $\rho \colon \g^* \to \R$
and a momentum map $\Phi \colon N \to \g^*$,
we may compose $\Phi$ with the differential $d\rho$ to obtain
a taming map 
\begin{equation}\Label{eq:v-drho}
 v := d\rho \circ \Phi \colon N \to \g .
\end{equation}
Notice that if $G$ is abelian then 
the $G$ action on $\g^*$ is trivial so
every function $\rho \colon \g^* \to \R$ is $G$-invariant. 
We will need the following two technical lemmas:

\begin{Lemma} \Label{vsharp is Ham vf}
Let $(N,\Phi,\omega)$ be a Hamiltonian $G$-manifold and 
let $\rho \colon \g^* \to \R$ be a $G$-invariant smooth function. 
Let $v = d\rho \circ \Phi$ be the corresponding taming map as
in~\eqref{eq:v-drho} and let $v^\sharp$ be the associated 
vector field on $N$.
Then 
\begin{enumerate} 
\item $\bfv$ satisfies Hamilton's equation for the function
$\rho \circ \Phi \colon N \to \R$, and 
\item if $\omega$ is non-degenerate, 
the localizing set coincides with the critical set 
of the function $\rho \circ \Phi$:
\begin{equation} \Label{Z=Crit}
 Z  := \{ v^\sharp = 0 \} = \operatorname{Crit} (\rho \circ \Phi). 
\end{equation}
\end{enumerate}
\end{Lemma}

\begin{proof}
At each point $x \in N$,
\begin{align*}
 d(\rho \circ \Phi) |_x & = d\rho|_{\Phi(x)} (d\Phi|_x)
 \qquad \text{by the chain rule} \\ 
 & = \langle v(x) , d\Phi|_x \rangle
 \qquad \text{by the definition of $v$} \\
 & = \iota(v^\sharp) \omega|_x 
 \qquad \text{by the definition of $v^\sharp$ and Hamilton's equation
   for $\Phi$}. 
\end{align*}
The second assertion follows from the first assertion
by the non-degeneracy of $\omega$. 
\end{proof}

The next lemma describes the localizing set $Z$ in
terms of orbit type strata. We will need the following terminology.
Let $(N,\omega,\Phi)$ be a non-degenerate Hamiltonian $G$-manifold.
Given $x$ in $N$, let $G_x$ denote the stabilizer subgroup in $G$ of
$x$, let $\g_x$ denote the Lie algebra of $G_x$ and let $\g_x^0$ denote the
annihilator of $\g_x$ in $\g^*$. Since $\omega$ is non-degenerate,
Hamilton's equation~\eqref{eq:Hamiltons} implies that
\begin{equation} \Label{imagedPhi} 
\image d\Phi|_x = \g_x^0 . 
\end{equation} 
Now
suppose that $G$ is a torus and let $S$ denote the orbit type stratum
through a point $x \in N$. Then $S$ is the connected component of $x$
in the subset \[ \{x' \in N \mid G_{x'} = G_x \} \] consisting of
points with the same stabilizer as $x$. The image of $S$ under $\Phi$
is an open subset of the affine plane $\Phi(x) + \g_x^0 \subset \g^*$
\cite{GuiSte:convexity}. (For example, if $N$ is a toric variety,
$\Phi(S)$ is the relative interior of a face of the momentum map
polytope.) Hence 
\begin{equation} \Label{tangent} 
T_{\Phi(x)} \Phi(S) = \g_x^0. 
\end{equation}

\begin{Lemma}\Label{lemma:crit-rho}
Let $G$ be a torus, $(N,\omega,\Phi)$ a non-degenerate Hamiltonian
$G$-manifold, 
$\rho \colon \g^* \to \R$ a $G$-invariant smooth function, 
$v = d\rho \circ \Phi$ the corresponding taming map
as in~\eqref{eq:v-drho},
and $Z = \{ \bfv = 0 \}$ the corresponding localizing set. 
Let $x \in N$, and let $S$ be the orbit type stratum that contains $x$.
Then
$$ x \in Z \text{ \ if and only if \ } 
   \Phi(x) \text{ \ is a critical point for \ } \rho|_{\Phi(S)} .$$
\end{Lemma}

\begin{proof}
We have 
\begin{align*}
v^\sharp(x) = 0 
   & \text{ \ if and only if \ } d(\rho \circ \Phi)|_x = 0 
             \qquad \text{ by \eqref{Z=Crit} } \\
   & \text{ \ if and only if \ } (d\rho|_{\Phi(x)}) (\g_x^0) = 0 
             \qquad \text{ by \eqref{imagedPhi} } \\
   & \text{ \ if and only if \ } 
             \text{ $\Phi(x)$ is a critical point for } \rho|_{\Phi(S)} 
\qquad \text{ by~\eqref{tangent}.} 
\end{align*}
\end{proof}

 The main localization results for Hamiltonian $G$-manifolds
 that occur in the current literature involve two different localizing
 sets $Z$:
 the critical set for a component of the momentum map,
 and the critical set for the norm-square of the momentum map. We now
 show that these localizing sets (and accompanying properness
 conditions often used in the theory)
 are special cases of our general construction.

\begin{Example}\Label{example:Prato-Wu}
Let $G$ be a torus.  Fix a Lie algebra element $\eta \in \g$,
and consider the corresponding linear functional on the dual space,
$ \rho(\cdot) := \langle \cdot , \eta \rangle \colon \g^* \to \R$. 
Since $G$ is abelian, the coadjoint $G$-action on $\g^*$ is trivial, 
so $\rho$ is a $G$-invariant function. Moreover, 
the differential of $\rho$ is the function $\g^* \to \g$ 
with constant value $\eta$.
Then, for a Hamiltonian $G$-manifold $(N,\omega,\Phi)$,
the corresponding taming map is the function $v \colon N \to \g$
with constant value $\eta$, and 
the function $\Phi^v \colon N \to \R$ is just
the $\eta$-component of the momentum map $\Phi$.
For a generic choice of $\eta \in \g$, the zero set of the vector-field 
$v^\sharp = \eta^\sharp$ coincides with
the set of fixed points for the entire torus $G$, so
\[ Z = N^G. \]
Thus, we recover the classical localizing set of the
original Duistermaat-Heckman theorem. 
Moreover, for this choice of $v$, the momentum map $\Phi$ is $v$-polarized
exactly when its $\eta$-component $\langle \Phi, \eta \rangle$ 
is proper and bounded below. This important condition 
in the theory of momentum maps was
first introduced and analyzed by Prato and Wu in \cite{PraWu94}.
In the book \cite{GGK}, a function $\Phi$ that satisfies this condition with
respect to an element $\eta \in \g$ 
is said to be $\mathbf{\eta}$\textbf{-polarized}.
\end{Example}

 In the next example we allow $G$ to be non-abelian. 

\begin{Example}\Label{example:norm-square}
Let $G$ be a compact Lie group 
and let $(N,\Phi,\omega)$ be a Hamiltonian $G$-space. 
Consider the norm-square function on $\g^*$ 
\begin{equation}
\rho \colon \g^* \to \R \quad , \quad \rho(\xi) = \|\xi\|^2.
\end{equation}
Composing $\Phi$ with the differential of $\rho$, it is
straightforward to compute that
\begin{equation}\Label{eq:v-Phi}
v(x) := d\rho \circ \Phi(x) = 2 \widehat{\Phi}(x),
\end{equation}
where $\widehat{\Phi}(x)$ is the element of $\g$ that corresponds
to the element $\Phi(x)$ of $\g^*$ under the identification
$\g \cong \g^*$.
So $\Phi^v = \langle \Phi, 2\widehat{\Phi}\rangle = 2 \|\Phi\|^2$.
Thus, when $\omega$ is nondegenerate, our theory
recovers the localizing set 
of Witten \cite{Witten:1992}, Paradan \cite{Par00}, and Woodward \cite{Woo05}:
\[
Z := \{ v^\sharp = 0\} = \operatorname{Crit}(\|\Phi\|^2).
\]
Since the norm-square $\|\Phi\|^2$ is proper if and only if $\Phi$ is
proper, 
the momentum map $\Phi$ is $v$-polarized if and only if it is proper. 
(Contrast with Remark~\ref{proper not polarized}.)
\end{Example}

\begin{Remark} \Label{rk:Euler}
In each of the above examples, the $v$-component of $\Phi$ 
is in fact a multiple of the Hamiltonian function $\rho \circ \Phi$:  
indeed, in Example~\ref{example:Prato-Wu}, we have 
\(\Phi^v = \langle \Phi, \eta \rangle = \rho \circ \Phi,\) while in
Example~\ref{example:norm-square}, we have \(\Phi^v = \langle
\Phi, 2\widehat{\Phi}\rangle = 2 \|\Phi\|^2 = 2 (\rho \circ \Phi).\)
These are instances of the following more general statement: 
if $\rho \colon \g^* \to \R$ is homogeneous of degree $k$, then 
\begin{equation} \Label{Euler}
\Phi^v = k(\rho \circ \Phi).
\end{equation}
To see this, for given $x \in N$, setting $\alpha = \Phi(x)$,
\begin{multline*}
\Phi^v(x) = \langle \Phi(x),v(x)\rangle 
= \langle \Phi(x), d\rho|_{\Phi(x)}\rangle = (L_{\alpha}\rho)(\alpha)  
\stackrel{(\star)}{=} k\rho(\alpha) =
k(\rho \circ \Phi)(x).
\end{multline*}
The equality~$(\star)$ is Euler's formula,
which holds for any homogeneous function $\rho$
of degree $k$ on a vector space.
\end{Remark}

We close the section with some observations concerning the smoothness of the localizing set.

\begin{Remark} \Label{Z not manifold}
A localizing set that is associated with a constant taming map
as in Example~\ref{example:Prato-Wu}
is always smooth.
Indeed, for a torus $G$ and a constant taming map 
$v \equiv \eta \in \g$,
the localizing set $Z$ is the fixed point set of the closure
in $G$ of the one-parameter subgroup generated by $\eta$.
Because $Z$ is the fixed point set of a compact group action,
its connected components are smooth submanifolds.

On the other hand, if the taming map $v$ is associated to the norm-square
of a momentum map as in Example~\ref{example:norm-square},
then $Z = \Crit(\|\Phi\|^2)$ need not be smooth.  For example, consider 
$S^2 \times S^2$ equipped with the standard area form on each factor
and the diagonal circle action.
Denote by $N$ and $S$ the north and south poles of $S^2$.
By the local normal form theorem,
we can identify a neighbourhood of the point $(N,S)$ in $S^2 \times S^2$
with a neighbourhood of the origin in $\C^2$, 
where the circle group acts on $\C^2$ with the weights $1, -1$
and with the momentum map 
$\Phi(z,w) = -\frac{1}{2}\|z\|^2 + \frac{1}{2} \|w\|^2$.
The critical set of $\|\Phi\|^2$ on $\C^2$ is the zero level set
$\{ (z,w) \ | \ \|z\| = \|w\| \}$, which is a cone over $S^1 \times S^1$. 
Thus, $Z$ is not smooth at $(N,S)$.
\end{Remark}

\section{Polarized completions}
\label{sec:polarized-completions}

In this section we use the taming maps    
introduced in Section~\ref{sec:taming} in order to introduce and
develop the notion of 
\textbf{polarized completions}. This notion is the technical tool that
allows us to both state and prove our localization formulas
in Sections~\ref{sec:DH} and~\ref{sec:twistedDH}.

We begin with some motivation. Recall from Section~\ref{sec:taming}
that the property of being $v$-polarized is crucial for our theory
due to its link to the properness of $\Phi$. On the other hand, 
in the course of our analysis below, we will encounter 
Hamiltonian $G$-manifolds $(N,\omega,\Phi)$ and taming maps 
$v \colon N \to \g$ such that 
the restriction to $\Phi$ to a closed subset $Y$ of $N$ is $v$-polarized,
but $\Phi$ is not $v$-polarized on all of $N$.  
For example, this may happen if $Y$ is a closed subset (e.g.\ a
localizing set) 
of a $v$-polarized Hamiltonian $G$-manifold
and $N$ is a small open neighbourhood of $Y$.
In such situations we wish to find 
a closed $2$-form $\tomega$ and momentum map $\tPhi$ on $N$
that agree with $\omega$ and $\Phi$ on $Y$
and such that $\tPhi$ is $v$-polarized on $N$.
Since we do not require $Y$ to be a manifold
(see\ Remark~\ref{Z not manifold}), we must first make precise 
what we mean by the condition that differential forms ``agree on $Y$".  
We take the diffeological approach:

\begin{Definition} \label{agree on}
Let $N$ be a manifold and $Y$ a subset of $N$.
Let $\alpha_0$ and $\alpha_1$ be differential forms on $N$,
possibly of mixed degree and with coefficients in a vector space
other than $\R$ 
(such as in the case of equivariant differential forms, as
recalled in Section~\ref{sec:DH}).
We say that $\alpha_0$ and $\alpha_1$ \textbf{agree on} $\mathbf{Y}$
if for any positive integer $k$, any open subset $U$ of $\R^k$,
and any smooth map $p \colon U \to N$ whose image is contained in $Y$,
the pullbacks of $\alpha_0$ and $\alpha_1$ to $U$ coincide, i.e.,  
$p^* \alpha_0 = p^*\alpha_1$ as differential forms on~$U$.
\end{Definition}

\begin{Remark}\label{remark: when agree}
If $\alpha_0$ and $\alpha_1$ agree on a neighbourhood of $Y$ in $N$,
then they agree on $Y$.  If $Y$ is a submanifold of $N$, then
$\alpha_0$ and $\alpha_1$ agree on $Y$ exactly if their pullbacks
to $Y$ coincide.   In practice, these are the only two cases
that we need.
\end{Remark}

We can now define $v$-polarized completions:

\begin{Definition}\label{definition:v-polarized-completion}
Let $(N,\omega,\Phi)$ be a Hamiltonian $G$-manifold,
possibly with boundary, and let $v \colon N \to \g$ be a taming map.
Let $Y$ be a $G$-invariant closed subset of $N$.
Suppose that the restriction of $\Phi$ to $Y$ is $v$-polarized. 
A {\it\textbf{v}}\textbf{-polarized completion of}
$\mathbf{(N,\omega,\Phi)}$ \textbf{relative to} $\mathbf{Y}$
is a Hamiltonian $G$ manifold $(N,\tomega,\tPhi)$
with the same underlying manifold $N$,
such that $\tPhi$ is $v$-polarized and
such that $\tomega+\tPhi$ agrees with $\omega+\Phi$ on $Y$.
\end{Definition}

The following proposition
is the main result of this section. 
By Definition~\ref{agree on} and Remark~\ref{remark: when agree}, 
the proposition 
gives a $v$-polarized completion of $(N,\omega,\Phi)$ relative to $Y$.

\begin{Proposition} \label{proposition:completion}
Let $G$ be a compact Lie group, let $N$ be a $G$-manifold,
possibly with boundary, and let $v \colon N \to \g$ be a taming map. 
Let $Z = \{ v^\sharp = 0 \}$ be the corresponding localizing set.
Let $Y$ be a closed $G$-invariant subset of $N$ that contains $Z$.
Let $\omega$ be a $G$-invariant closed $2$-form on $N$
and $\Phi$ a corresponding momentum map. Suppose that the restriction
of $\Phi$ to $Y$ 
is $v$-polarized. 
Then there exists an invariant closed $2$-form $\tomega$ on $N$
and corresponding momentum map $\tPhi$
that coincide with $\omega$ and $\Phi$ 
on a $G$-invariant neighbourhood of $Y$
and such that $\tPhi$ is $v$-polarized on $N$.
\end{Proposition}

The remainder of this section is devoted to the proof of Proposition
\ref{proposition:completion}.  
We begin with four elementary lemmas about real valued functions.
The first lemma asserts that a convex combination of functions 
that are proper and bounded from below is still proper and bounded from below:

\begin{Lemma}\label{lemma:convex}
Let $N$ be a topological space. Let $f,g \colon N \to \R$ be
continuous functions that are proper and bounded from below.
Let $\rho_1,\rho_2 \colon N \to \R$ 
be continuous functions that satisfy
$\rho_1 \geq 0$, $\rho_2 \geq 0$, and $\rho_1 + \rho_2 \equiv 1$.
Then the function 
$$ \rho_1 f + \rho_2 g \ \colon \  N \to \R $$
is proper and bounded from below.
\end{Lemma}

\begin{proof}
Let $\psi$ denote the function $ \rho_1 f + \rho_2 g $.
The function $\psi$ is bounded from below by $\min ( \inf f , \inf g )$.  
For any $x \in N$ and $b \in \R$, 
if $\psi(x) \leq b$ then either $f(x) \leq b$ or $g(x) \leq b$.  So,
for any $a < b$, 
$$ \psi\Inv([a,b]) \subset f\Inv([\inf f,b]) \cup g\Inv([\inf g,b]) . $$
The union on the right hand side is compact because $f$ and $g$ are proper.
Being a closed subset of a compact set, $\psi\Inv([a,b])$ is compact.
Because the interval $[a,b]$ was arbitrary, this shows that $\psi$ is proper.
\end{proof}

The second lemma states that 
if a function is proper and bounded from below
then its average with respect to a compact group action
is also proper and bounded below: 

\begin{Lemma} \label{average-proper} Let $G$ be a compact Lie group and $N$ a
  topological space with a $G$-action.  Let $f \colon N \to \R$ be a
  continuous function that is proper and bounded from below.  Then its
  $G$-average $\ol{f} \colon N \to \R$, defined by
\[
\ol{f} (x) := \int_{g \in G} f(g \cdot x) dg
\]
where $dg$ denotes the Haar probability measure on $G$,
is also proper and bounded from below.
\end{Lemma}

\begin{proof}
Every lower bound for $f$ is also a lower bound for $\ol{f}$.
For any $x \in N$ and $b \in \R$,
if $\ol{f}(x) \leq b$, then there exists $g \in G$ 
such that $f(g \cdot x) \leq b$, 
and so $x \in g\Inv \cdot (f\Inv([\inf f,b]))$.  Hence, for any $a<b$,
\begin{equation} \label{in union}
 \ol{f}\Inv([a,b]) \subset \bigcup_{g \in G} g\Inv \cdot f\Inv([\inf f,b]).
\end{equation}
The right hand side of~\eqref{in union} is the image of the compact
set $G \times f\Inv([\inf f,b])$ under the continuous map 
$G \times N \to N$ \ , \ $(g,\alpha) \mapsto g\Inv \cdot \alpha$. 
Being a closed subset of a compact set, 
the left hand side is also compact, as desired.
\end{proof}

Next, we show that it is possible to expand slightly the set on which a
function is proper and bounded from below: 

\begin{Lemma}\label{lemma:proper-on-nd}
Let $N$ be a locally compact topological space and 
$f \colon N \to \R$ a continuous function. 
Let $Y \subset N$ be a closed subset, and 
suppose that the restriction $f|_Y \colon Y \to \R$
is proper and bounded from below.
Then there exists an open neighbourhood $U_Y$ of $Y$ in $N$
such that the restriction of $f$ to the closure of $U_Y$ in $N$,
$$ f|_{\ol{U}_Y} \colon \ol{U}_Y \to \R,$$
is also proper and bounded from below.
\end{Lemma}

\begin{proof}
By local compactness, for each point $y$ in $Y$ we may choose 
an open neighbourhood $U_y$ in $N$ whose closure in $N$ is compact
and such that $|f(u)-f(y)| < 1$ for all $u \in U_y$.

Let $\ell$ be any integer.  Because $f$ is proper on $Y$, 
the intersection $f\Inv([\ell,\ell+1]) \cap Y$
is compact,  so it is covered by finitely many of the sets $U_y$ 
for $y \in f\Inv([\ell,\ell+1]) \cap Y$.
Let $U_\ell$ denote the union of the elements of such a finite cover.
Then $f(U_\ell) \subset [\ell-1,\ell+2]$,
and the closure $\ol{U_\ell}$ is compact, by construction of the
sets $U_y$. 

Consider $U_Y := \bigcup \left\{ U_\ell \ | \ \ell \in \Z \right\}$.
Because the $U_\ell$ form a locally finite collection of subsets of $N$,
the closure of their union is the union of their closures: 
$\ol{U_Y} = \bigcup \left\{ \ol{U_\ell} \ | \ \ell \in \Z \right\}$.

Let $[a,b] \subset \R$ be any interval in $\R$. We wish to show
that $f^{-1}([a,b]) \cap \overline{U_Y}$ is compact. 
First, observe that $f^{-1}([a,b]) \cap \ol{U_\ell}$ is non-empty 
only if $[a,b]$ meets $[\ell-1,\ell+2]$,
which occurs for only finitely many integers $\ell$. 
Hence, the intersection $f^{-1}([a,b]) \cap \ol{U_Y}$
is contained in a finite union of the sets $\ol{U_\ell}$.
Being a closed subset of a finite union of compact sets, 
this intersection is compact.  
This shows that $f$ is proper on $\ol{U_Y}$, as desired. 

Finally, let $B$ be a lower bound for $f$ on $Y$;
then $B-1$ is a lower bound for $f$ on $\ol{U_Y}$, by construction 
of the sets $U_y$. Hence $f$ is also bounded below on $\ol{U_Y}$. 
This completes the proof. 
\end{proof}

The previous lemmas are quite general and apply to 
topological spaces that are not necessarily manifolds.
In preparation for proving
Proposition~\ref{proposition:completion},
we now return to the setting of manifolds
and prove a variant of Proposition~\ref{proposition:completion}
that applies to real-valued functions:

\begin{Lemma}\label{lemma:properFunction}
Let $N$ be a $G$-manifold, possibly with boundary. 
Let $Y$ be a $G$-invariant closed subset of $N$, 
and let $f \colon N \to \R$ be a smooth $G$-invariant function 
such that the restriction
$f|_Y \colon Y \to \R$ is proper and bounded from below.
Then there exists a smooth real valued $G$-invariant function on $N$ 
that is proper and bounded from below
and that coincides with $f$ on some $G$-invariant open neighbourhood 
of $Y$ in~$N$.
\end{Lemma}

\begin{proof}
By Lemma \ref{lemma:proper-on-nd}, there exists an open neighbourhood $U_Y$ 
of $Y$ such that the restriction of $f$ to $\ol{U_Y}$ is proper 
and bounded from below. 
Let $g \colon N \to \R$ be an arbitrary smooth function that is proper 
and bounded from below (see~\cite[p.53]{GP}).  
Let $\rho_1$ and $\rho_2$ be a smooth partition of unity
subordinate to the open covering \(\{U_Y, N \ssminus Y\}\) of $N$ and
let 
$$ \psi' := \rho_1 f + \rho_2 g \colon N \to \R .$$
The manifold $N$ decomposes as the union of the two closed subsets
$$ N = \ol{U}_Y \cup (N \ssminus U_Y).$$
On $N \ssminus U_Y$, the function $\psi'$ coincides with $g$.
Because $g$ is proper and bounded from below and $N \ssminus U_Y$
is closed in $N$, it follows that $\psi'|_{N \ssminus U_Y}$ is
proper and bounded from below.  On the set $\ol{U}_Y$, the function $\psi'$
is a convex combination of the functions $f$ and $g$, both of which are
proper and bounded from below on $\ol{U}_Y$.
By Lemma~\ref{lemma:convex}, it follows that $\psi'|_{\ol{U}_Y}$
is also proper and bounded from below and hence that $\psi'$ is proper
and bounded from below on all of $N$. 

Finally, since \(\supp \rho_2 \subset N \ssminus Y\)
and $N$ is Hausdorff, there exists some
open neighbourhood $U'$ of $Y$ such that 
\(\supp \rho_2 \cap U' = \emptyset,\) 
hence in particular \(\psi' \equiv f\) on $U'$.

We now define $\psi$ to be the $G$-average of $\psi'$. 
We claim that the function $\psi$ satisfies the conditions of the lemma.
By Lemma~\ref{average-proper}, $\psi$ is still proper 
and bounded from below on $N$.  Moreover, since $\psi' \equiv f$ on
$U'$ and $f$ is $G$-invariant, $\psi$ coincides with $f$ 
on the intersection
$$ \bigcap_{g \in G} g \cdot U'. $$
This intersection is clearly $G$-invariant.
It is a neighbourhood of $Y$
because its complement, 
being the image of the closed set $G \times (N \ssminus U')$
under the proper map $(g,x) \mapsto g \cdot x$ from $G \times N$ to $N$, 
is closed.
This concludes the proof. 
\end{proof}

We are ready to prove the main result of this section:

\begin{proof}[Proof of Proposition~\ref{proposition:completion}]
We first observe that the function $\Phi^v := \left< \Phi , v \right>
 \colon N \to \R$
is $G$-invariant, since both $\Phi$ and $v$ are $G$-equivariant by assumption.
Applying Lemma~\ref{lemma:properFunction} to $\Phi^v$ and $Y$, 
we conclude that there exists a $G$-invariant function $\psi \colon N \to \R$
that is proper, bounded from below, and
coincides with $\Phi^v$ on a $G$-invariant neighbourhood $U$ of $Y$.

Let $Z = \{v^\sharp = 0\}$ be the localizing set 
and let $g_N$ denote a choice of a $G$-invariant Riemannian metric on $N$. 
Consider the $G$-invariant $1$-form $\Theta$ on $N \ssminus Z$ defined by
$$ \Theta(\cdot) 
   = \frac{g_N( \cdot, \bfv)  }{ g_N(\bfv,\bfv)}.$$
Note that $\Theta$ has the property that $\Theta(v^\sharp) \equiv 1$
on $N \ssminus Z$. Such a $1$-form is sometimes called 
a \textbf{Bott projector} (see, e.g., \cite{carrell}), 
following Bott~\cite{Bott1}.
Since \(\psi  - \Phi^v\) is identically zero on an open set
that contains $Y$ and hence $Z$, the product
\((\psi - \Phi^v)\Theta\) defines a $G$-invariant $1$-form
on all of $N$ that vanishes on $Z$. 

We will now explicitly construct a $2$-form $\tomega$ and momentum map $\tPhi$
that satisfy the conditions of the proposition. 
Let
\begin{equation}\label{eq:def-tomega}
\tomega := \omega - d((\psi - \Phi^v)\Theta)
\end{equation}
and 
\begin{equation}\label{eq:def-tPhi}
\tPhi^X := \Phi^X + (\psi - \Phi^v)\Theta(X^{\sharp})
\qquad \text{ for } X \in \g.
\end{equation}

On the neighbourhood $U$ of $Y$ on which $\psi$ coincides with $\Phi^v$,
we have that $\tomega$ and $\tPhi$ coincide with $\omega$ and $\Phi$
respectively. 
Since $\omega$ is closed, the form 
$\tomega$ is also closed. Moreover, 
since both $\omega$ and $(\psi-\Phi^v)\Theta$ are $G$-invariant, it 
follows that $\tomega$ is $G$-invariant. The $G$-equivariance of
$\tPhi$ follows from the $G$-equivariance of $\Phi$, 
the $G$-invariance of $(\psi - \Phi^v)\Theta$,
and the $G$-equivariance of the map $X \mapsto X^\sharp$.
Hamilton's equation~\eqref{eq:Hamiltons} is satisfied by
  $\tomega$ and $\tPhi$, as can be checked as follows: 
\begin{eqnarray*}
 d(\tPhi^X) & = & d(\Phi^X + (\psi - \Phi^v) \Theta(X^{\sharp})) \\
 & = & d(\Phi^X) + d(\imath(X^{\sharp})((\psi - \Phi^v) \Theta)) \\
 & = & \imath(X^{\sharp})\omega - \imath(X^{\sharp})d((\psi - \Phi^v) \Theta) \\
 & = & \imath(X^{\sharp})\tomega,\\
\end{eqnarray*}
where the second to last equality uses the $G$-invariance
of the $1$-form $(\psi - \Phi^v) \Theta$ and the Cartan formula
for the Lie derivative,
$L_{X^\sharp} = d \iota (X^\sharp) + \iota(X^\sharp) d$.
On the set $N \ssminus Z$, by definition of the Bott projector,
\begin{align*}
\tPhi^v & = \Phi^v + (\psi - \Phi^v) \Theta(\bfv) \\
 & = \Phi^v + (\psi - \Phi^v) \\
 & = \psi .
\end{align*}
On the set $U$,
we have \(\psi - \Phi^v \equiv 0 \),
so $\tPhi^v = \Phi^v = \psi$.
Since \(N = U \cup (N \ssminus Z),\) 
we conclude that $\tPhi^v \equiv \psi$ on all of $N$. Since $\psi$ is
proper and bounded below, this implies $\tPhi$ is $v$-polarized on
$N$. 
The result follows. 
\end{proof}

\section{Localization formulas for the Duistermaat-Heckman distribution}
\Label{sec:DH}

The main result of this section, Theorem~\ref{theorem:GLS-cobordism}, 
is a localization formula that
expresses the \DH distribution of a Hamiltonian $G$-manifold
in terms of data near a localizing set
arising from an arbitrary taming map.
We begin by recalling the definition of the \DH distribution: 

\begin{Definition} \Label{DH distribution}
Let $(M,\omega,\Phi)$ be a $2n$-dimensional proper Hamiltonian $G$-manifold
(see\ Definition~\ref{proper nondegenerate}).
The \textbf{Duistermaat-Heckman distribution}, 
$\tDH_{(M,\omega,\Phi)} \colon \Cinf_c(\g^*) \to \R$, 
is the distribution on $\g^*$ that associates to any compactly supported 
test function $\varphi \in \Cinf_c(\g^*)$ on~$\g^*$ the real number
\[
\tDH_{(M,\omega,\Phi)} (\varphi) := 
      \int_M (\varphi \circ \Phi) \frac{\omega^n}{n!}.
\]
(The right hand side is well defined because $\varphi$ is 
compactly supported and $\Phi$ is proper.)
\end{Definition}

In the next lemma we observe that the \DH distribution associated 
to the boundary of an odd-dimensional 
Hamiltonian $G$-manifold must be identically zero. 
The lemma is 
an easy consequence of Stokes' theorem. The argument is the same
as that given in \cite[Section 2.3]{GGK}; 
we briefly recount the proof here for the reader's convenience
(compare also to the arguments in Section~\ref{sec:twistedDH}) and
since this idea is central to our cobordism arguments.

\begin{Lemma} \Label{lemma:Stokes}
Let $(W,\omega,\Phi)$ be a $(2n+1)$-dimensional 
proper Hamiltonian $G$-manifold with boundary $\del W$.
Let $\omega_{\del W}$ and $\Phi_{\del W}$ denote
the pullbacks of $\omega$ and $\Phi$ to the boundary.
Then the Duistermaat-Heckman distribution associated to 
the Hamiltonian $G$-manifold 
$(\del W, \omega_{\del W}, \Phi_{\del W} )$
is identically zero.
\end{Lemma}

\begin{proof}
Let $\varphi \in \Cinf_c(\g^*)$. 
We have
\begin{equation}\Label{eq:DHpairing}
\begin{split}
\tDH_{(\del W , \omega_{\del W} , \Phi_{\del W} )} (\varphi)
  & = \int_{\partial W} \varphi \circ \Phi \hsm \frac{\omega^n}{n!}
  \quad \text{ by definition of the $\tDH$ distribution }  \\
  & = \int_{W} 
      d \left(\varphi \circ \Phi \hsm \frac{\omega^n}{n!}\right) 
      \quad \text{ by Stokes' theorem } \\
  & = \int_{W} d \left(\varphi \circ \Phi\right) \wedge \frac{\omega^n}{n!}
      \quad \text{ since $\omega$ is closed} \\ 
  & = \int_{W} d\varphi \circ 
      \left(d\Phi \wedge \frac{\omega^n}{n!}\right) 
      \quad \text{ by the chain rule } 
     \\
\end{split}
\end{equation}
where $d\Phi$ is understood to be a $\g^*$-valued 1-form on $W$. 
For every \( X \in \g \), we have
\(\left<d\Phi, X\right> = d \Phi^X = \imath(X^{\sharp})\omega\),
so
\begin{equation}\Label{eq:closed}
\left< d\Phi \wedge \frac{\omega^n}{n!}, X\right>  \ = \ 
 \imath(X^{\sharp})\omega \wedge \frac{\omega^n}{n!} \ = \  
 \imath(X^{\sharp}) \frac{\omega^{n+1}}{(n+1)!} \ = \  0,
\end{equation}
because $\omega^{n+1}$ is a form of degree $2n+2$
on the $(2n+1)$-dimensional manifold $W$.  
Thus, the $\g^*$-valued $(2n+1)$-form \(d\Phi \wedge \frac{\omega^n}{n!}\) 
vanishes when paired with any \(X \in \g \). 
We conclude that \eqref{eq:DHpairing} vanishes 
for any test function $\varphi$, as required.
\end{proof}

In view of Lemma~\ref{lemma:Stokes} above, we recall the following
definition from~\cite[Chap.~2, Def.~2.20]{GGK}:

\begin{Definition} \Label{proper cob}
A \textbf{proper Hamiltonian cobordism}
between Hamiltonian $G$-manifolds
$(M_0,\omega_0,\Phi_0)$ and $(M_1,\omega_1,\Phi_1)$
is a proper Hamiltonian $G$-manifold with boundary
$(W,\tomega,\tPhi)$
and an orientation preserving diffeomorphism
$$ i \colon \ - M_0 \sqcup M_1 \ \to \ \del W $$
(where the negative sign denotes opposite orientation)
such that
$$ i^*(\tomega+\tPhi) \ = \ 
         (\omega_0+\Phi_0) \ \sqcup \ (\omega_1+\Phi_1).$$
\end{Definition}

The next proposition is essential for what follows. The idea of our
main theorem (Theorem~\ref{theorem:GLS-cobordism}) is to construct a
proper Hamiltonian cobordism between a given Hamiltonian $G$-manifold
$(M,\omega,\Phi)$ and another Hamiltonian $G$-manifold 
that is described only in terms of local data near $Z$.

\begin{Proposition} \Label{cob then same DH}
Let $(M_0,\omega_0,\Phi_0)$ and $(M_1,\omega_1,\Phi_1)$ be two even-dimensional
Hamiltonian $G$-manifolds. Suppose that 
there exists a proper Hamiltonian cobordism between them. 
Then 
$(M_0,\omega_0,\Phi_0)$ and $(M_1,\omega_1,\Phi_1)$ are proper
Hamiltonian $G$-manifolds and 
\[
\tDH_{(M_0, \omega_0, \Phi_0)} = \tDH_{(M_1, \omega_1, \Phi_1)}.
\]
\end{Proposition}

\begin{proof}
  The first assertion is immediate from the definition of a proper
  Hamiltonian cobordism, since each boundary component is a closed
  subset of the cobording manifold. The second assertion follows from
  Lemma~\ref{lemma:Stokes}. 
\end{proof}

\begin{Remark} \Label{compose}
We can compose cobordisms (see\ \cite{BJ}) 
after choosing ``trivializations" of
tubular neighborhoods of the boundary components.
Thus the existence of a proper Hamiltonian cobordism
is an equivalence relation on Hamiltonian $G$-manifolds
with proper momentum maps.
\end{Remark}

In the previous section, we define the notion of a polarized
  completion of a Hamiltonian $G$-manifold with respect to a closed
subset. This notion is used in the statement of the main theorem of
this section, since the right hand side of the formula~\eqref{main} 
is the \DH distribution $\tDH_{(U_Z, \omega_Z, \Phi_Z)}$ 
of a polarized completion $(U_Z, \omega_Z, \Phi_Z)$ of a
neighbourhood $U_Z$ of the localizing set $Z$
(see\ Def.~\ref{def:DHgerm}).  
Moreover, an implicit assertion in the statement of
Theorem~\ref{theorem:GLS-cobordism} is that this \DH distribution
$\tDH_{(U_Z, \omega_Z, \Phi_Z)}$ is in fact
independent of the choice of polarized completion in a sense that we make
precise below.
The justification of this last assertion will involve two main
ingredients.  First, we
need to place restrictions on the neighbourhood $U_Z$. 
Second, we use the invariance of equivariant cohomology
under equivariant homotopy and the Cartan model for equivariant
cohomology.  Definition~\ref{weak} explains the first of these two
ingredients: 

\begin{Definition} \Label{weak}
Let $N$ be a $G$-manifold and $Z$ a $G$-invariant closed subset.
A \textbf{smooth equivariant weak deformation retraction}
of $N$ to $Z$ is an equivariant smooth homotopy
$p_t \colon N \to N$, for $0 \leq t \leq 1$, 
such that 
\begin{itemize} 
\item $p_0$ is equal to the identity map on $N$,  
\item $p_1$ maps $N$ to $Z$, 
\item for all $t$, the map $p_t$ sends $Z$ to $Z$. 
\end{itemize} 
(A \emph{smooth homotopy} is a homotopy 
such that the map $[0,1] \times N \to N$ defined by $(t,x) \mapsto p_t(x)$
is smooth. The homotopy is \emph{equivariant} if this map 
$[0,1] \times N \to N$ is equivariant, where $G$ acts 
on the product $[0,1] \times N$ by
the given action on $N$ and trivially on the first factor.)
\end{Definition}

\begin{Remark}
  When $t=1$, the map $p_1$ may be viewed as a map from $N$ to $Z$. This
  map $p_1$ is, in particular, a homotopy inverse to the inclusion map 
  $i \colon Z \to N$ since, by assumption, the composition 
  $i \circ p_1 \colon N \to N$ is homotopic through the maps 
  $p_t \colon N \to N$ to the
  identity map on $N$, and the composition $p_1 \circ i \colon Z \to
  Z$ is homotopic through the maps $p_t \circ i \colon Z \to Z$ to the
  identity map on $Z$.  In particular, this implies that the
  restriction map $H^*_G(N) \to H^*_G(Z)$ is an isomorphism.  Here
  $H^*_G(Z)$ is understood to be the singular (not de Rham) Borel-equivariant
  cohomology of $Z$, since we do not assume that $Z$ is a manifold.
\end{Remark}

The next two remarks compare Definition~\ref{weak} to related notions
that appear in the literature. 

\begin{Remark}
 In Definition~\ref{weak}, the map $p_1 \colon N \to Z$ is not
  required to be a retraction, i.e., its
  restriction to $Z$ is not required to be the identity map on $Z$.
  Such a requirement is too stringent for our purposes.  The
  difference between the two notions may be seen in the following
  example. There is no smooth retraction from $\R^2$ to the union of
  the two coordinate axes, since if $p$ is a smooth map that fixes the
  two coordinate axes, then its differential at the origin must be the
  identity, so $p$ must be a diffeomorphism near the origin and it
  cannot be a retraction to the union of the axes. However, it is
  possible to construct a smooth weak deformation retraction in the sense 
  of Definition~\ref{weak}
  from $\R^2$ to the union of the coordinate axes.
\end{Remark}

\begin{Remark}
  Definition~\ref{weak} is the smooth equivariant analogue of a
  ``deformation retraction in the weak sense", in the sense of Hatcher
  in \cite[Chapter 0, Exercise 4]{hatcher}.  We also note that the phrase
  ``weak deformation retraction'' sometimes refers to
   a continuous map $p \colon N \to Z$ 
  that is a homotopy inverse to the inclusion map $Z \hookrightarrow N$,
  but in which the intermediate maps $N \to N$ in the homotopy 
  of $i \circ p$ to the identity are not required to carry $Z$ to itself.  
\end{Remark}

We will work with situations
in which $Z$ is an invariant closed subset of a $G$-manifold $M$
and $N$ is a $G$-invariant neighbourhood of $Z$ in $M$
that admits a smooth equivariant weak deformation retraction to $Z$.

We now briefly recall 
the Cartan model for Borel-equivariant cohomology (with $\R$ coefficients). 
Let $M$ be a $G$-manifold. Then 
an \textbf{equivariant differential form on $M$} 
is a $G$-equivariant polynomial function
from the Lie algebra $\g$ to the space $\Omega^*(M)$
of differential forms on $M$. Identifying polynomial $\R$-valued functions on
$\g$ with the symmetric algebra $S(\g^*)$, we may think of an
equivariant differential form $\alpha$ as an element of $\Omega^*(M)
\otimes S(\g^*)$. The $G$-equivariance condition ensures that $\alpha$
is an element of the $G$-invariants in  the tensor product,
where $G$ acts diagonally on each factor in the standard
fashion. Hence 
\[
\Omega_G^*(M) \cong \left( \Omega^*(M) \otimes_\R S(\g^*) \right)^G.
\]
Since both $\Omega^*(M)$ and $S(\g^*)$ are graded rings, we may equip 
$\Omega^*_G(M)$ with the grading 
\[
\Omega_G^k(M) := \bigoplus\limits_{i+2j=k} 
               \left( \Omega^i(M) \otimes_\R S^j(\g^*) \right)^G.
\]
The \textbf{equivariant differential} 
$d_G \colon \Omega_G^*(M) \to \Omega_G^{*+1}(M)$, 
the equivariant analogue of the ordinary exterior derivative operator 
on $\Omega^*(M)$, is defined by the formula
\[
(d_G\alpha)(X) = d(\alpha(X)) - \iota_{X^\sharp}(\alpha(X))
\]
where $X \in \g$ is a Lie algebra element and $X^\sharp \in \Vect(M)$
denotes its corresponding vector field on $M$.  
The equivariant differential satisfies 
$d_G \circ d_G = 0$, so we may define 
the \textbf{equivariant cohomology (with $\R$ coefficients)} 
$H_G^*(M;\R)$ as the cohomology
of the complex $\left( \Omega_G^*(M) , d_G \right)$. 
This is naturally isomorphic to the Borel-equivariant cohomology of the
$G$-space $M$ with $\R$ coefficients, as defined in terms of the Borel
construction. 

The next lemma, which is one of our main technical tools, states 
conditions under which a polarized completion is unique up to
cobordism.

\begin{Lemma} \Label{lemma:independent}
Let $N$ be an even-dimensional oriented $G$-manifold 
and $Z$ a closed subset of $N$.
Suppose that there exists a smooth equivariant weak deformation retraction 
from $N$ to $Z$.
Let $v \colon N \to \g$ be a bounded taming map
with corresponding localizing set $\{ \bfv = 0 \}$ equal to $Z$. 
Let $\omega_0$ and $\omega_1$ be closed $G$-invariant $2$-forms on $N$ 
and $\Phi_0$ and $\Phi_1$ corresponding momentum maps. 
Suppose that $\Phi_0$ and $\Phi_1$ are $v$-polarized and 
suppose that $\omega_0 + \Phi_0$ and $\omega_1 + \Phi_1$ 
agree on~$Z$ in the sense of Definition~\ref{agree on}.
Let $W = [0,1] \times N$, equipped with the $G$-action 
that is trivial on the first factor 
and is the given action on $N$ on the
second factor. Let $i_0,i_1 \colon N \to W$
be the inclusions at levels $0$ and $1$ respectively.  Then there exists on $W$
a closed $2$-form $\tomega$ and proper momentum map
$\tPhi \colon W \to \g^*$ such that
$i_0^* (\tomega + \tPhi) = \omega_0 + \Phi_0$
and
$i_1^* (\tomega + \tPhi) = \omega_1 + \Phi_1$.
In particular, there exists a proper Hamiltonian cobordism between
$(N, \omega_0, \Phi_0)$ and $(N, \omega_1, \Phi_1)$. 
\end{Lemma}

Applying Proposition~\ref{cob then same DH}, 
we immediately get the following important consequence of the lemma:

\begin{Corollary} \Label{cor:stokes}
Under the assumptions and notation of Lemma~\ref{lemma:independent},
\[ \tDH_{(N, \omega_0, \Phi_0)} = \tDH_{(N, \omega_1, \Phi_1 )}.\]
\end{Corollary}

\begin{proof}[Proof of Lemma~\ref{lemma:independent}] 
Let $p_t \colon N \to N$, for $0 \leq t \leq 1$,
be an equivariant smooth weak deformation retraction
from $N$ to $Z$ according to Definition~\ref{weak}.
Because the image of $p_1 \colon N \to N$ is contained in $Z$,
and because $\omega_0+\Phi_0$ and $\omega_1+\Phi_1$ agree on $Z$
according to Definition~\ref{agree on}, 
the pullback 
$p_1^* \left( (\omega_1+\Phi_1) - (\omega_0+\Phi_0) \right)$ is zero
on $N$ and 
in particular is equivariantly exact.
Because $p_1$ is smoothly equivariantly homotopic to the identity
map on $N$, and equivariantly homotopic maps induce the same
pullback map on equivariant cohomology,
$(\omega_1+\Phi_1) - (\omega_0+\Phi_0)$ 
is also equivariantly exact on $N$. 
Thus there exists a $G$-invariant $1$-form $\alpha$ on $N$ 
such that 
\begin{equation} \Label{eq primitive}
  d\alpha = \omega_1 - \omega_0
  \quad \text{ and } \quad
  \alpha(X^\sharp) =   \Phi_0^X - \Phi_1^X  \text{ for every } X \in \g.
\end{equation}
Let $t$ denote the first coordinate on the product
$ W = [0,1] \times N $,
and let $\pi \colon W \to N$ denote the projection map to the second
factor. 
Define the closed $G$-invariant $2$-form 
$$ \tomega := \pi^*\omega_0 + d(t\pi^*\alpha) $$
on $W$. This has an associated 
momentum map $\tPhi \colon W \to \g^*$ given by 
$$ \tPhi^X := \pi^*\Phi_0^X - t\pi^*\alpha(X^\sharp).$$
By~\eqref{eq primitive},
the function $\tPhi$ is equal to the convex combination
$ (1-t) \pi^* \Phi_0 + t \pi^* \Phi_1 $.
The functions $\pi^* \Phi_0$ and $\pi^* \Phi_1$
are $\pi^*v$-polarized, because $\Phi_0$ and $\Phi_1$ 
are $v$-polarized and $\pi$ is proper.
By Lemma~\ref{lemma:convex}, 
the function $\tPhi$ is $\pi^*v$-polarized.
Since $v$ is bounded by assumption, from
Lemma~\ref{lemma:polarized-proper} we conclude that
$\tPhi$ is proper.

We have shown that $(W, \tomega, \tPhi)$ is a Hamiltonian $G$-manifold
with boundary whose momentum map $\tPhi$ is proper.
Identifying the boundary $\partial W$ with
$- N \sqcup N$ (where the minus sign denotes reversed orientation),
the restriction of $\tomega$ to $\partial W$ is 
$\omega_0 \sqcup \omega_1$ 
and the restriction of $\tPhi$ is $\Phi_0 \sqcup \Phi_1$.
This completes the proof. 
\end{proof}

In preparation for our arguments in Section~\ref{sec:twistedDH} and
because the generalization requires no substantial additional argument,
we also consider
Hamiltonian $G$-manifolds equipped with equivariant
cohomology classes, namely, quadruples $(M,\omega,\Phi,A)$, 
where $(M,\omega,\Phi)$ is a Hamiltonian $G$-manifold 
and $A \in H_G^*(M)$ is an equivariant cohomology class on $M$.
We define a \textbf{proper Hamiltonian cobordism}
between two such quadruples, $(M_0,\omega_0,\Phi_0,A_0)$
and $(M_1,\omega_1,\Phi_1,A_1)$, to be 
a proper Hamiltonian $G$-manifold with boundary
equipped with an equivariant cohomology class
$(\tM,\tomega,\tPhi,\tA)$,
and a diffeomorphism $i \colon -M_0 \sqcup M_1 \to \del \tM$, 
such that 
\begin{equation}\Label{requirements for quadruples}
i^* \tomega = \omega_0 \sqcup \omega_1, \quad
   i^* \tPhi = \Phi_0 \sqcup \Phi_1, \quad
   i^* \tA = A_0 \sqcup A_1. 
\end{equation}
If there exists a proper Hamiltonian cobordism in the above sense
between $(M_0,\omega_0,\Phi_0,A_0)$ and $(M_1,\omega_1,\Phi_1,A_1)$, 
then in particular there exists a
proper Hamiltonian cobordism between $(M_0,\omega_0,\Phi_0)$
and $(M_1,\omega_1,\Phi_1)$ in the original sense of
Definition~\ref{proper cob}. Thus, by slight abuse of language, 
henceforth we use the term ``proper Hamiltonian cobordism"
to refer to the relation defined above between Hamiltonian $G$-spaces, 
i.e., triples
$(M, \omega, \Phi)$, and also between Hamiltonian $G$-spaces equipped
with an equivariant cohomology class, 
i.e., quadruples $(M, \omega, \Phi, A)$ as above.

\begin{Remark} \Label{compose with A} As in Remark~\ref{compose}, the
  existence of a proper Hamiltonian cobordism is an equivalence
  relation on quadruples $(M,\omega,\Phi,A)$. Indeed, we can compose
  two cobordisms as in Remark~\ref{compose}, and by a Mayer-Vietoris
  argument there exists an equivariant cohomology class on the
  composed cobordism that restricts to the given ones on the two
  pieces. \end{Remark}

The next lemma shows that the \DH distribution of a polarized
completion on a neighbourhood of $Z$ with respect to $Z$ is not only
independent of the choice of polarized completion, as we saw in
Corollary~\ref{cor:stokes}, but is also independent of the choice of the
neighborhood.

\begin{Lemma} \Label{lemma:independent2}
Let $(M,\omega,\Phi)$ be a Hamiltonian $G$-manifold, $Z$ an invariant closed subset,
and $A \in H_G^*(M)$ an equivariant cohomology class.
Let $U_Z^0$ and $U_Z^1$ be invariant open neighbourhoods of $Z$ in $M$, and 
suppose that there exist smooth equivariant weak deformation retractions
from $U_Z^0$ and $U_Z^1$ to $Z$. 
Let $v \colon M \to \g$ be a bounded taming map
such that $\{ v^\sharp = 0 \} \cap U_Z^0 = \{ v^\sharp = 0 \} \cap U_Z^1 = Z$.
For $j=0,1$, let $\omega_j$ be a closed $G$-invariant $2$-form on $U_Z^j$
and let $\Phi_j$ be a corresponding momentum map.
For $j=0,1$, suppose that $\Phi_j$ is $v|_{U_Z^j}$-polarized
and that $\omega_j + \Phi_j$ agrees with $\omega+\Phi$ on $Z$ 
in the sense of Definition~\ref{agree on}.
Then there exists a proper Hamiltonian cobordism
between $(U_Z^0,\omega_0,\Phi_0,A|_{U_Z^0})$ and $(U_Z^1,\omega_1,\Phi_1,A|_{U_Z^1})$.
\end{Lemma}

\begin{proof}
The union $U := U_Z^0 \cup U_Z^1$ is a $G$-invariant open neighbourhood
of $Z$ in $M$ whose intersection with the localizing set
$\{ v^\sharp =0 \}$ is equal to $Z$.
Consider the product $U \times [0,1]$ as a $G$-manifold,
with the $G$-action given by the action on the first factor.

Consider the $G$-invariant open subset $W$ of $U \times [0,1]$ 
defined by
$$ W := U \times [0,1] \ \ssminus \  \left( 
     ((U \ssminus U_Z^0) \times \{ 0\}) \, \sqcup \,
     ((U \ssminus U_Z^1) \times \{ 1\})         \right) .$$
Let $\omega_W$ and $\Phi_W$ denote the pullbacks of $\omega$ and $\Phi$
under the map from $W$ to $U$ given by projecting to the first
factor. 
Then $(W,\omega_W,\Phi_W)$ is a Hamiltonian $G$-manifold.

By slight abuse of notation 
we denote also by 
  $v \colon W \to \g$ the pullback
  of $v \colon U \to \g$ by the projection $W \to U$
  to the first factor.  
The localizing set of $(W , v)$ is then $Z \times [0,1]$.
Moreover, the restriction of $\Phi_W$ to $Z \times [0,1]$
is $v$-polarized.

Now let $(W,\tomega,\tPhi)$ be a $v$-polarized completion
of $(W,\omega_W,\Phi_W)$ relative to $Z \times [0,1]$,
which exists by Proposition~\ref{proposition:completion}.
Identifying the boundary $\del W$ of $W$ with $U_Z^0 \sqcup U_Z^1$,
the restriction of $\tomega + \tPhi$ to $\del W$
has the form
$$ (\omega_0' + \Phi_0') \ \sqcup \ (\omega_1' + \Phi_1') ,$$
where $\omega_j'+\Phi_j'$ is a polarized completion
of $(U_Z^j,\omega|_{U_Z^j},\Phi|_{U_Z^j})$ with respect to~$Z$.
Taking $\tA \in H_G^*(W)$ to be the pullback of $A \in H_G^*(M)$ 
through the map $(x,t) \mapsto x$,
we obtain that $(W,\tomega,\tPhi,\tA)$
is a proper Hamiltonian cobordism
between $(U_Z^0,\omega_0',\Phi_0',A|_{U_Z^0})$
and $(U_Z^1,\omega_1',\Phi_1',A|_{U_Z^1})$.

By Lemma~\ref{lemma:independent},
there also exist proper Hamiltonian cobordisms
between $(U_Z^0,\omega_0',\Phi_0',A|_{U_Z^0})$ 
and $(U_Z^0,\omega_0,\Phi_0,A|_{U_Z^0})$
and between $(U_Z^1,\omega_1',\Phi_1',A|_{U_Z^1})$ 
and $(U_Z^1,\omega_1,\Phi_1,A|_{U_Z^1})$.
Composing these cobordisms (see\ Remark~\ref{compose with A})
gives the desired result.
\end{proof}

Applying Proposition~\ref{cob then same DH} yields the following 
consequence.

\begin{Corollary}\Label{corollary:DH indep of nbhd}
Under the assumptions and notation of Lemma~\ref{lemma:independent2}, 
\[
\tDH_{(U_Z^0, \omega_0, \Phi_0)} = \tDH_{(U_Z^1, \omega_1, \Phi_1)}.
\]
\end{Corollary}

Our next step is to find a proper
Hamiltonian cobordism between two Hamiltonian $G$-manifolds
whose \DH distributions yield the left and right hand sides
of our localization formula~\eqref{main} below. 
For this we again invoke the existence of $v$-polarized completions.
The proof uses the same ideas 
as that of Lemma~\ref{lemma:independent2}: we 
start with a trivial cobordism, remove irrelevant pieces of the boundary, 
and take a polarized completion.

\begin{Proposition} \Label{main cobordism}
Let $(M,\omega,\Phi)$ be an even-dimensional Hamiltonian $G$-manifold 
without boundary.  
Let $v \colon M \to \g$ be a bounded taming map,
let $Z = \{ \bfv = 0 \}$ be the corresponding localizing set,
and let $U_Z$ be a $G$-invariant neighbourhood of $Z$ in $M$.
Suppose that $\Phi$ is $v$-polarized on $M$.
Let $A \in H^*_G(M)$ be an equivariant cohomology class on $M$.
Then there exist a $v|_{U_Z}$-polarized completion $(U_Z,\omega_Z,\Phi_Z)$
of $(U_Z,\omega|_{U_Z},\Phi|_{U_Z})$ relative to $Z$ 
(in the sense of Definition~\ref{definition:v-polarized-completion})
and a proper Hamiltonian cobordism between $(M,\omega,\Phi,A)$
and $(U_Z,\omega_Z,\Phi_Z,A|_{U_Z})$. 
\end{Proposition}

\begin{proof}
  Consider $M \times [0,1]$ as a $G$-manifold, with the $G$-action given
  by the action on the first factor.  Let $\pi: M \times [0,1] \to M$
  denote the projection to the first factor. By a slight abuse of notation 
we denote by
  $v \colon M \times [0,1] \to \g$ the pullback
  of $v \colon M \to \g$ by $\pi$. 
 The localizing set of $(M \times [0,1],v)$ is then $Z \times [0,1]$.  
Consider the open subset $W$ of $M \times [0,1]$ defined by
$$ W := M \times [0,1] \ssminus 
   \left((M \ssminus U_Z) \times \{0\}\right),$$
and consider the closed subset $Y$ of $W$ defined by 
$$  Y := \left(Z \times [0,1] \right) \cup \left( M \times \{1\} \right).$$
Both $W$ and $Y$ are $G$-invariant since $U_Z$ and $Z$ are
$G$-invariant. 
Let $\omega_W$ and $\Phi_W$ denote the pullbacks of $\omega$ and $\Phi$ 
under the projection map $\pi \vert_W \colon W \to M$. 
Then $(W,\omega_W,\Phi_W)$ is a Hamiltonian $G$-manifold. 
Moreover, 
since $\Phi$ is assumed to be $v$-polarized on $M$ and $\pi$ is proper, 
the map $\pi^*\Phi$ is $v$-polarized on $M \times [0,1]$, and 
since $Y$ is closed as a subset of $M \times [0,1]$,
the restriction of $\Phi_W$ to $Y$ is also $v$-polarized.

Now let $(W,\tomega,\tPhi)$ be a choice of $v$-polarized completion 
of $(W,\omega_W,\Phi_W)$ relative to $Y$, which exists 
by Proposition~\ref{proposition:completion}. 
Identifying the boundary $\partial W$ of $W$ with $-U_Z \sqcup M$, 
the restriction of $\tomega+\tPhi$ to $\partial W$ becomes
\begin{equation}\Label{eq:boundary} 
(\omega_Z+\Phi_Z) \ \sqcup \ (\omega+\Phi),
\end{equation}
where $\omega_Z+\Phi_Z$ denotes the restriction of $\tomega+\tPhi$
to the boundary component~$U_Z$.
Because $\tPhi$ is $v$-polarized by construction and this boundary component 
is a closed subset of $W$, the map $\Phi_Z$ is $v$-polarized.
Moreover, by the construction of $\tomega$ and $\tPhi$ and by the definition 
of $Y$, 
we also have that $\omega_Z+\Phi_Z$ agrees with $\omega+\Phi$ on $Z$.
Thus, $(U_Z,\omega_Z,\Phi_Z)$ is a $v$-polarized completion 
of $(U_Z,\omega|_{U_Z},\Phi|_{U_Z})$ relative to $Z$. 
Because $\tPhi$ is $v$-polarized 
and $v$ is bounded, $\tPhi$ is proper. Finally, note that the restriction 
to $\del W$ of the equivariant cohomology class $\tA := \pi|_W^*A$
is $A|_{U_Z} \sqcup A$.
Thus, $(W, \tomega, \tPhi,\tA)$ is a proper Hamiltonian cobordism 
between $(M,\omega,\Phi, A)$ and $(U_Z,\omega_Z,\Phi_Z, A \vert_{U_Z})$.
\end{proof}

For the next theorem, we introduce the following notation. 

\begin{Definition}\Label{def:DHgerm}
Let $(M,\omega,\Phi)$ be a Hamiltonian $G$ manifold, 
let $v \colon M \to \g$ be a bounded taming map,  
and let $Z_i$ be a connected component of the localizing set.
Suppose that there exist arbitrarily small neighbourhoods of $Z_i$
that admit smooth equivariant weak deformation retractions to $Z_i$.
(This means that every neighbourhood of $Z_i$ contains 
a neighbourhood with this property.)
Let $U_i$ be such a neighbourhood, 
and let $(U_i,\omega_i,\Phi_i)$ be a $v|_{U_i}$-polarized completion
of $(U_i,\omega|_{U_i},\Phi|_{U_i})$
relative to $Z_i$.  In this situation we define the notation
\begin{equation} \Label{DH on germ}
 \tDH_{\germ_{Z_i}(M,\omega,\Phi)}^v := 
 \tDH_{(U_i, \omega_i, \Phi_i)}. 
\end{equation}
\end{Definition} 

To justify the notation, we note that the distribution
on the right hand side of~\eqref{DH on germ}
is independent of the choice of $U_i$ and polarized completion
$(\omega_i, \Phi_i)$ 
by Corollary~\ref{corollary:DH indep of nbhd}. Moreover, 
it is determined by the restriction of $\omega$ 
and $\Phi$ to arbitrarily small neighbourhoods of $Z_i$
because $U_i$ can be chosen to be arbitrarily small.

\begin{Remark} \Label{technical hypothesis}
In the above discussion we made the technical hypothesis,
that there exist arbitrarily small neighbourhoods of $Z_i$ 
that admit a smooth equivariant weak deformation retraction
to $Z_i$.  This hypothesis is automatically satisfied when $Z_i$ 
is a manifold by choosing an invariant tubular neighbourhood of $Z_i$.
In many examples,
such as those considered in Section~\ref{sec:BrianchonGram},
the components $Z_i$ of the localizing set $Z$ are indeed smooth and
hence manifolds.
However, there are important situations in which the $Z_i$ are not
necessarily smooth. 
Specifically, the critical set for the norm-square of the momentum map
for a nondegenerate Hamiltonian $G$-manifold can be singular,
as we saw in Remark~\ref{Z not manifold}.
Nevertheless, we expect this critical set
to always satisfy our technical hypothesis. 
A proof would construct such weak deformation
retractions locally using local normal forms for Hamiltonian $G$-manifolds,
and would then patch them in an appropriate sense using a partition of
unity. 
\end{Remark}

We now state and prove the main theorem.

\begin{Theorem}\Label{theorem:GLS-cobordism}
Let $(M,\omega,\Phi)$ be an even-dimensional
Hamiltonian $G$-manifold without boundary. 
Let $v \colon M \to \g$ be a bounded taming map and
let $Z = \{ \bfv = 0 \}$ be the corresponding localizing set.
Suppose that $\Phi$ is $v$-polarized on $M$, hence on $Z$.
Let
\[
Z = \bigsqcup_{i \in \mathcal{I}} Z_i
\]
be the decomposition of the localizing set $Z$ into its connected components.
Suppose that, for every $i \in \mathcal{I}$, 
there exist arbitrarily small neighbourhoods
of $Z_i$ that admit smooth equivariant weak deformation retractions to $Z_i$.
Then
\begin{equation} \Label{main}
 \tDH_{(M,\omega,\Phi)} = \sum_i \tDH_{\germ_{Z_i}(M,\omega,\Phi)}^v.
\end{equation}
\end{Theorem}

\begin{proof}[Proof of Theorem~\ref{theorem:GLS-cobordism}] 
For each $i \in \calI$, choose an invariant neighborhood $U_i$ 
of $Z_i$ that admits an equivariant smooth weak deformation retraction
to $Z_i$.  Moreover, choose these neighbourhoods $U_i$
sufficiently small so that their closures are disjoint.
Let $U_Z$ be the union of the neighbourhoods $U_i$.
Let $(U_Z,\omega_Z,\Phi_Z)$ be a $v|_{U_Z}$-polarized
completion of $(U_Z,\omega|_{U_Z},\Phi|_{U_Z})$ relative to $Z$, as 
obtained from Proposition~\ref{main cobordism}.
In particular, there exists a proper Hamiltonian cobordism
between $(M,\omega,\Phi)$ and $(U_Z,\omega_Z,\Phi_Z)$.
By Proposition~\ref{cob then same DH}, 
\((M, \omega, \Phi)\) and \( (U_Z,\omega_Z,\Phi_Z) \) 
have the same \DH distribution:
\begin{equation} \Label{eq1}
 \tDH_{(M,\omega,\Phi)} = \tDH_{(U_Z,\omega_Z,\Phi_Z)} .
\end{equation}

Let $\omega_i$ and $\Phi_i$ be the restrictions 
of $\omega_Z$ and $\Phi_Z$ to the component $U_i$.
Then the \DH distribution of $(U_Z,\omega_Z,\Phi_Z)$
is the sum of the \DH distributions of $(U_i,\omega_i,\Phi_i)$:
\begin{equation} \Label{eq2}
  \tDH_{(U_Z,\omega_Z,\Phi_Z)} = \sum_i \tDH_{(U_i,\omega_i,\Phi_i)}.
\end{equation}

Because
$(U_i,\omega_i,\Phi_i)$ is a $v|_{U_i}$-polarized completion
of $(U_i,\omega|_{U_i},\Phi|_{U_i})$ relative to $Z_i$, by definition
\begin{equation} \Label{eq3}
 \tDH_{\germ_{Z_i}(M, \omega, \Phi)}^v = \tDH_{(U_i,\omega_i,\Phi_i)}.
\end{equation}

Equation~\eqref{main} follows from~\eqref{eq1}, \eqref{eq2}, and~\eqref{eq3}.
\end{proof}

\begin{Remark} \Label{equivalent maps give same loc}
In Definition~\ref{def:DHgerm},
if $(U_i,\omega_i,\Phi_i)$ is a polarized completion
of $(U_i,\omega|_{U_i},\Phi|_{U_i})$,
and if $v'$ is a taming map that is equivalent to $v$
in the sense of Remark~\ref{rk:equivalent},
then $(U_i,\omega_i,\Phi_i)$ is also a polarized completion
of $(U_i,\omega|_{U_i},\Phi|_{U_i})$ relative to $v'$.
Thus, the \DH distribution
$\tDH^v_{\germ_{Z_i}(M,\omega,\Phi)}$
is independent of the choice of taming map $v$
within an equivalence class in the sense of Remark~\ref{rk:equivalent}.
\end{Remark}

\section{Localization formulas for twisted Duistermaat-Heckman distributions}
\Label{sec:twistedDH}

The main result of this section is a localization theorem for 
\emph{twisted} \DH distributions
(Theorem~\ref{theorem:twisted-GLS-cobordism}) that is analogous to
Theorem~\ref{theorem:GLS-cobordism}.

\subsection*{Definitions and notation.}

For any real vector space $V$, 
there is a natural embedding \(V \to \Vect(V)\) of $V$ into
the space of smooth vector fields on $V$ by \(v \mapsto \tilde{v},\)
where $\tilde{v}$ denotes the constant coefficient vector field
\(\tilde{v}(x) = v \in T_xV \cong V\). 
Furthermore, a smooth vector field $X$ on $V$
may be interpreted as an element of the space $\LDO(V)$ 
of linear differential operators on $V$
via the Lie derivative ${\mathcal L}_X$,
so we also have an embedding $\Vect(V) \to 
  \LDO(V)$ by the association $X \mapsto {\mathcal L}_X$. 
Because partial derivatives commute,
the composition $V \to \Vect(V) \to \LDO(V)$
extends to an algebra embedding $S(V) \to \LDO(V)$, 
denoted $Q \mapsto D_Q$, of the symmetric algebra $S(V)$
into $\LDO(V)$.
Given an element \(Q \in S(V),\) we denote by
\begin{equation}
\begin{split}
  S(V) \otimes_{\R} C^{\infty}(V) & \mapsto C^{\infty}(V) \\
  (Q, \varphi) & \mapsto D_Q \varphi
\end{split}
\end{equation}
the pairing obtained by applying the differential
operator $D_Q$ to the function~$\varphi$. 

Consider now the case \(V = \g^*,\) the dual of the Lie algebra $\g$
of a compact Lie group $G$. Also, let $M$ be a $G$-manifold
and \(\Phi \colon M \to \g^*\) a smooth map. 
Composing with the pullback
\(\Phi^* \colon C^{\infty}(\g^*) \to C^{\infty}(M)\) yields the linear
map 
\begin{equation}\Label{eq:def-DQ}
\begin{split}
S(\g^*) \otimes_\R C^{\infty}(\g^*) & \to C^{\infty}(M) \\
(Q, \varphi) & \mapsto \Phi^*(D_Q \varphi) = D_Q \varphi \circ \Phi.
\end{split}
\end{equation}
Finally, tensoring~\eqref{eq:def-DQ} with the identity map on
$\Omega^*(M)$ and composing with the pointwise multiplication map
\(\Omega^*(M) \otimes_\R C^{\infty}(M) \to \Omega^*(M),\) we obtain
a linear map \(\left(\Omega^*(M) \otimes_\R S(\g^*)\right) \otimes_\R
C^{\infty}(\g^*) \to \Omega^*(M).\)   An element
\(\eta \in \Omega^*(M) \otimes_\R S(\g^*)\) determines via this map a
linear transformation 
\begin{equation}
  \Label{eq:def-DetaPhi}
 D_{\eta,\Phi}  \colon C^{\infty}(\g^*) \to \Omega^*(M), 
\end{equation}
which may be expressed in explicit coordinates as follows. 

We fix for the rest of this discussion a choice of basis $\{X_1, X_2,
\ldots, X_r\}$ for $\g$ and
corresponding dual basis $\{\beta_1, \beta_2, \ldots, \beta_r\}$ of
$\g^*$. For a multi-index \(\mathbf{a} = (a_1, a_2, \ldots, a_r) \in
\Z^r_{\geq 0},\) denote by $\beta^{\mathbf{a}}$ the monomial
\(\beta_1^{a_1} \beta_2^{a_2} \cdots \beta_r^{a_r} \in S(\g^*).\) An
element \(\eta \in \Omega^*(M) \otimes_\R S(\g^*)\) may be
expressed in these coordinates as 
\begin{equation}\Label{eq:eta}
\eta = \sum_{\mathbf{a}} \eta_{\mathbf{a}} \beta^{\mathbf{a}},
\end{equation}
where the coefficients \(\eta_{\mathbf{a}} \in \Omega^*(M)\) are
differential forms. Tracing through the definition of the map
\(D_{\eta,\Phi}\) of~\eqref{eq:def-DetaPhi}, an explicit
computation shows that, for \(\varphi \in C^{\infty}(\g^*)\) 
and $\eta$ as above, 
the function~\eqref{eq:def-DetaPhi} is given by
\begin{equation}
  \Label{eq:DetaPhivarphi}
  D_{\eta,\Phi}(\varphi) = \sum_{\mathbf{a}} \left(
    \Phi^*(D_{\beta^{\mathbf{a}}} \varphi) \right) \eta_{\mathbf{a}}.
\end{equation}

We may now define the twisted \DH distribution on $\g^*$.
We first place the additional assumption that $\Phi$ is \emph{proper}. 
In this case, for any compactly supported function $\varphi$ on $\g^*$
and any $\mathbf{a} \in \Z^r_{\geq 0}$,
the functions $D_{\beta^{\mathbf{a}}}\varphi$ 
and $\Phi^*(D_{\beta^{\mathbf{a}}}\varphi)$ are also
compactly supported (on $\g^*$ and $M$ respectively), and hence
$D_{\eta, \Phi}(\varphi)$ is a compactly supported differential form
on $M$. In particular, its integration over $M$ is well-defined. 
Now let
$(M, \omega, \Phi)$ be a proper Hamiltonian $G$-manifold
(see\ Definition~\ref{proper nondegenerate}).
Let \(\eta \in
(\Omega^*(M) \otimes_\R S(\g^*))^G\) be an equivariantly closed
equivariant form on $M$. We define the \textbf{twisted
  Duistermaat-Heckman distribution with respect to $(M, \omega, \Phi)$
  and $\eta$} on $\g^*$ as follows: 
\begin{equation}
  \Label{eq:def-twistedDH}
  \tDH_{(M,\omega,\Phi)}(\eta) \colon 
   \varphi \mapsto \int_M e^{\omega} \wedge D_{\eta,\Phi}(\varphi).
\end{equation}
When there is no danger of ambiguity, we will occasionally abuse
notation and denote by $\tDH(\eta)$ the distribution
$\tDH_{(M,\omega,\Phi)}(\eta)$.

The explicit formula~\eqref{eq:def-twistedDH} 
implies that $\varphi \mapsto \tDH(\eta)(\varphi)$ 
is linear and continuous
as a map from the space $\Cinf_c(\g^*)$
of compactly supported functions, with its $\Cinf$ topology, to $\R$.  
Hence $\tDH(\eta)$ is a distribution. Moreover, when $\eta \equiv 1$,
the twisted
Duistermaat-Heckman distribution reduces to the classical
(``untwisted'') Duistermaat-Heckman distribution discussed in the
previous section.

\begin{Remark}
When $M$ is compact, we can integrate $\exp(\omega+i\Phi) \wedge \eta$ 
over $M$ to obtain an analytic function on $\g$.
The twisted \DH distribution $\tDH_{(M,\omega,\Phi)}(\eta)$
is essentially the Fourier transform of this function.
See \cite[section 3.1]{Woo05}.
\end{Remark}

We now add some extra data to that of a Hamiltonian $G$-manifold:
we call \textbf{Hamiltonian $G$-manifold equipped with a closed 
equivariant form} a quadruple $(M, \omega, \Phi, \eta)$ 
where $(M, \omega, \Phi)$ is a Hamiltonian $G$-manifold and
\(\eta \in \Omega^*_G(M)\) is an equivariant differential form on $M$
that is equivariantly closed. 
In analogy with the definitions in Section~\ref{sec:DH} (and analogous
slight abuse of language), 
a \textbf{proper Hamiltonian cobordism} between two such quadruples
\((M_0, \omega_0, \Phi_0, \eta_0)\) and \((M_1, \omega_1, \Phi_1, \eta_1)\)
is a quadruple $(\tM, \tomega, \tPhi, \teta)$
and a diffeomorphism \(i \colon -M_0 \sqcup M_1 \to \del \tM\) 
 such that
 \begin{equation}\Label{eq:def proper Hamiltonian cobordism with form}
 i^* \tomega = \omega_0 \sqcup \omega_1, \quad i^*\tPhi = \Phi_0 \sqcup
 \Phi_1, \quad i^* \teta = \eta_0 \sqcup \eta_1, 
 \end{equation}
 and such that \((\tM, \tomega, \tPhi, \teta)\) is itself a proper Hamiltonian
 $G$-manifold (with boundary) equipped with a closed equivariant
 form. Note that if such a cobordism exists, 
then $\Phi_0$ and $\Phi_1$ are necessarily proper.

\begin{Remark}
As in Remark~\ref{compose with A},
after possibly modifying the equivariant differential forms
on tubular neighbourhoods of the boundary components,
we can compose such cobordisms.  Thus, being cobordant
in the sense defined above is an equivalence relation on 
proper Hamiltonian $G$-manifolds equipped with closed equivariant
forms.
\end{Remark}

We begin with the following ``twisted analogue" of
Lemma~\ref{lemma:Stokes} and Proposition~\ref{cob then same DH}:

\begin{Lemma} \Label{cobordism twisted DH}
Let $(M_0,\omega_0,\Phi_0,\eta_0)$
and $(M_1,\omega_1,\Phi_1,\eta_1)$
be proper Hamiltonian $G$-manifolds equipped with closed equivariant
forms.
Suppose that there exists a proper Hamiltonian cobordism
between $(M_0,\omega_0,\Phi_0,\eta_0)$
and $(M_1,\omega_1,\Phi_1,\eta_1)$. Then 
$$ \tDH_{(M_0,\omega_0,\Phi_0)} (\eta_0) 
   = \tDH_{(M_1,\omega_1,\Phi_1)} (\eta_1) .$$
\end{Lemma}

\begin{proof}
Let $(\tM,\tomega,\tPhi,\teta)$, with diffeomorphism
$i \colon -M_0 \sqcup M_1 \to \del \tM$,
be a proper Hamiltonian cobordism 
between $(M_0,\omega_0,\Phi_0,\eta_0)$ and $(M_1,\omega_1,\Phi_1,\eta_1)$.
To show that the twisted \DH distributions are equal, it suffices to
show that for any \(\varphi \in C^{\infty}_c(\g^*)\) we have
\begin{equation}
  \tDH_{(M_0,\omega_0,\Phi_0)}(\eta_0) (\varphi)
= \tDH_{(M_1,\omega_1,\Phi_1)}(\eta_1) (\varphi).
\end{equation}
We compute the difference as
\begin{align*}
 \tDH_{(M_1,\omega_1,\Phi_1)}(\eta_1) (\varphi)
 - \tDH_{(M_0,\omega_0,\Phi_0)}(\eta_0) (\varphi) 
 & = \int_{M_1} e^{\omega_1} \wedge D_{\eta_1, \Phi_1}(\varphi)
 - \int_{M_0} e^{\omega_0} \wedge D_{\eta_0, \Phi_0}(\varphi) \\ 
 & = \int_{\del \tM} e^{\tomega} \wedge D_{\teta, \tPhi}(\varphi)
\qquad \text{ by~\eqref{eq:def proper Hamiltonian cobordism with form} } \\
 & = \int_{\tM} d \left( e^{\tomega} \wedge D_{\teta, \tPhi}(\varphi) \right)
\qquad \textup{ by Stokes' theorem}.
\end{align*}
Thus it suffices to prove that 
\begin{equation} \Label{should be zero}
 \int_{\tM} d \left(e^\tomega \wedge D_{\teta,\tPhi}(\varphi)\right) = 0. 
\end{equation}
We first write
\[
\teta = \sum_{\mathbf{a}} \teta_{\mathbf{a}} \beta^{\mathbf{a}}
\]
with respect to the basis $\{\beta_i\}$ of $\g^*$ fixed above,
where the $\teta_{\mathbf{a}}$ are differential forms on $M$
(of mixed degree).
Then 
\(D_{\teta, \tPhi}(\varphi) = \sum_{\mathbf{a}} \left(
  \tPhi^*(D_{\beta^{\mathbf{a}}} \varphi) \right)
\teta_{\mathbf{a}}\) by~\eqref{eq:DetaPhivarphi}.
Since $\tomega$ is closed, \(d(e^{\tomega}) = 0\), so
\[
d \left( e^\tomega \wedge D_{\teta,\tPhi}(\varphi) \right) = e^\tomega
\wedge d(D_{\teta, \tPhi}(\varphi)).
\]
We compute 
\begin{equation}
\Label{eq:dDtetatPhivarphi}
\begin{split}
  d\left( D_{\teta,\tPhi}(\varphi) \right) & = d \left(
    \sum_{\mathbf{a}} \left(
      \tPhi^*(D_{\beta^{\mathbf{a}}}\varphi)\right) \teta_{\mathbf{a}}\right) \\
 & = \sum_{\mathbf{a}} d \left( \tPhi^*(D_{\beta^{\mathbf{a}}}\varphi)
 \right) \wedge \teta_{\mathbf{a}} + \sum_{\mathbf{a}} \left(
   \tPhi^*(D_{\beta^{\mathbf{a}}}\varphi) \right) d\teta_{\mathbf{a}}
 \\
 & = \sum_{\mathbf{a}} \tPhi^* \left( d
   (D_{\beta^{\mathbf{a}}}\varphi) \right) \wedge \teta_{\mathbf{a}} +
 \sum_{\mathbf{a}} \left(
   \tPhi^*(D_{\beta^{\mathbf{a}}}\varphi) \right) d\teta_{\mathbf{a}}
 \\
 & = \sum_{\mathbf{a}} \sum_i \tPhi^* \left( D_{\beta_i}
   (D_{\beta^{\mathbf{a}}}\varphi) \right) d \langle \tPhi, X_i \rangle
 \wedge \teta_{\mathbf{a}} + \sum_{\mathbf{a}} \left( \tPhi^*
   (D_{\beta^{\mathbf{a}}}\varphi) \right) d\teta_{\mathbf{a}} \\
 & = \sum_{\mathbf{a}} \sum_i \tPhi^*\left( D_{\beta_i}
   (D_{\beta^{\mathbf{a}}}\varphi) \right)  \iota(X_i^\sharp) \tomega
 \wedge \teta_{\mathbf{a}} +  \sum_{\mathbf{a}} \left( \tPhi^*
   (D_{\beta^{\mathbf{a}}}\varphi) \right) d\teta_{\mathbf{a}}. 
\end{split}
\end{equation}

Now recall that $\teta$ is equivariantly closed, so for any \(X \in
\g,\) we have 
\[
d \left( \teta(X) \right) - \iota(X^\sharp) \teta(X) = 0.
\]
This implies 
\[
\sum_{\mathbf{a}} d \teta_{\mathbf{a}} \beta^{\mathbf{a}} = 
\sum_{\mathbf{a}} \sum_i \iota(X_i^\sharp) \teta_{\mathbf{a}} \beta_i
\beta^{\mathbf{a}},
\]
where $\{X_i\}$ is the basis for $\g$ as fixed above. Thus for any
\(\varphi \in C^{\infty}_c(\g^*)\) we have 
\[
\sum_{\mathbf{a}} \left( \tPhi^*(D_{\beta^{\mathbf{a}}}\varphi) \right)
d\teta_{\mathbf{a}} =  \sum_{\mathbf{a}} \sum_i \left( \tPhi^* (
  D_{\beta_i \beta^{\mathbf{a}}} \varphi) \right) \iota(X_i^\sharp)
\teta_{\mathbf{a}}.
\]
Noticing that \(D_{\beta_i \beta^{\mathbf{a}}} \varphi = D_{\beta_i} D_{\beta^{\mathbf{a}}} \varphi\) by definition of $D_Q$, we have 
\[
\sum_{\mathbf{a}} \left( \tPhi^*(D_{\beta^{\mathbf{a}}}\varphi) \right)
d\teta_{\mathbf{a}} = \sum_{\mathbf{a}} \sum_i \left( \tPhi^* (
  D_{\beta_i} D_{\beta^{\mathbf{a}}} \varphi) \right) \iota(X_i^\sharp)
\teta_{\mathbf{a}}.
\]
Substituting into the last expression in~\eqref{eq:dDtetatPhivarphi}, we obtain 
\[
d \left( D_{\teta, \tPhi}(\varphi) \right) = 
\sum_i \sum_{\mathbf{a}} \tPhi^*\left( D_{\beta_i}
  D_{\beta^{\mathbf{a}}} \varphi \right) \left( \iota(X_i^\sharp)
  \tomega \wedge \teta_{\mathbf{a}} + \iota(X_i^\sharp)
  \teta_{\mathbf{a}} \right).
\]
Hence 
\begin{equation}
  \Label{eq:etomega-dD}
  e^{\tomega} \wedge d\left( D_{\teta, \tPhi}(\varphi) \right) = 
\sum_i \sum_{\mathbf{a}} \tPhi^*\left( D_{\beta_i}
  D_{\beta^{\mathbf{a}}} \varphi \right) \left( e^\tomega \wedge \iota(X_i^\sharp)
  \tomega \wedge \teta_{\mathbf{a}} + e^\tomega \wedge \iota(X_i^\sharp)
  \teta_{\mathbf{a}} \right).
\end{equation}
From the definition of the exponential 
\[
e^\tomega = 1 + \tomega + \frac{\tomega^2}{2!} + \frac{\tomega^3}{3!}
+ \cdots 
\]
it can be seen
that \(\iota(X_i^\sharp)\tomega \wedge e^\tomega =
\iota(X_i^\sharp)e^\tomega = e^\tomega \wedge \iota(X_i^\sharp)
\tomega.\) Hence we may further simplify~\eqref{eq:etomega-dD} as
\begin{equation*}
\begin{split}
   e^{\tomega} \wedge d\left( D_{\teta, \tPhi}(\varphi) \right) & = 
  \sum_i \sum_{\mathbf{a}} \tPhi^*\left( D_{\beta_i}
  D_{\beta^{\mathbf{a}}} \varphi \right) \left( \left( \iota(X_i^\sharp)
  e^\tomega \right) \wedge \teta_{\mathbf{a}} + e^\tomega \wedge
\left( \iota(X_i^\sharp)
  \teta_{\mathbf{a}}\right) \right) \\
 & = \sum_i \sum_{\mathbf{a}} \tPhi^*\left( D_{\beta_i}
  D_{\beta^{\mathbf{a}}} \varphi \right) \iota(X_i^\sharp) \left(
  e^\tomega \wedge \teta_{\mathbf{a}} \right) \\
 & = \sum_i \sum_{\mathbf{a}} \iota(X_i^\sharp) \left( \tPhi^* \left(
     D_{\beta_i} D_{\beta^{\mathbf{a}}} \varphi \right) \left(
     e^\tomega \wedge \teta_{\mathbf{a}} \right) \right).
\end{split}
\end{equation*}

The integral over $M$ of the right hand side of this last equation is
$0$, because the expression is the contraction by a vector field of a
differential form, and hence its top degree part (which is the only
part which contributes to the integral) must be $0$. 
Hence~\eqref{should be zero} vanishes, as desired. 
\end{proof}

The purpose of the next lemma is to show that the twisted \DH
distribution is independent of the choice of a closed
equivariant form $\eta$ within an equivariant cohomology class.

\begin{Lemma} \Label{cohomologous are cobordant}
Let $(M,\omega,\Phi)$ be a proper Hamiltonian $G$-manifold.
Suppose that $\eta_0$ and $\eta_1$ are closed equivariant differential forms
on $M$ such that \([\eta_0] = [\eta_1] \in H^*_G(M;\R)\).
Then there exists a proper Hamiltonian cobordism 
between $(M,\omega,\Phi,\eta_0)$ and $(M,\omega,\Phi,\eta_1)$. 
\end{Lemma}

\begin{proof}
Since $[\eta_0] = [\eta_1]$ there exists 
an equivariant differential form $\gamma$ on $M$ 
such that $\eta_1 - \eta_0 = d_G \gamma$.
Equip $\tM := [0,1] \times M$
with the $2$-form $\tomega = \pi^* \omega$
and momentum map $\tPhi = \pi^* \Phi$,
where $\pi \colon \tM \to M$ is the projection map to the
second factor.  
Consider $\tM$ to be a $G$-manifold equipped with the given action on
$M$ and the trivial action on $[0,1]$. This makes $(\tM, \tomega,
\tPhi)$ a Hamiltonian $G$-space, and $\pi$ an equivariant map. 
Moreover, since $\Phi \colon M \to \g^*$ is proper,
$\tPhi \colon \tM \to \g^*$ is also proper.

Let 
$t$ denote the coordinate on the interval $[0,1]$. 
Define on $\tM$ the equivariant different form
$$ \tilde{\eta} = \pi^* \eta_0 + d_G (t\pi^*\gamma). $$
Since $d_G \eta_0=0$ by assumption, 
$\teta$ is also equivariantly closed. 
Define $i_0 \colon \tM \times M$ and $i_1 \colon \tM \times M$
by $i_0(m) = (0,m)$ and $i_1(m) = (1,m)$.
Then $i_0^* \tomega = i_1^* \tomega = \omega$,
$i_0^* \tPhi = i_1^* \tPhi = \Phi$,
$i_0^* \teta = \eta_0$, and $i_1^* \teta = \eta_0 + d_G \gamma =
\eta_1$. 
This shows that $(\tM,\tomega,\tPhi,\teta)$
is a proper Hamiltonian cobordism 
between $(M,\omega,\Phi,\eta_0)$ and $(M,\omega,\Phi,\eta_1)$, 
as was to be shown. 
\end{proof}

By Lemmas~\ref{cobordism twisted DH}
and~\ref{cohomologous are cobordant} we have just shown that 
the following notion is well-defined:

\begin{Definition} \Label{twisted DH with coh}
Let $(M,\omega,\Phi)$ be a proper Hamiltonian $G$-manifold, and 
let $A$ be an equivariant cohomology class in $H^*_G(M)$.
Let $\eta$ be an equivariant differential form on $M$
representing $A$, i.e., \(A = [\eta].\)  
We define the \textbf{twisted \DH distribution
with respect to $\mathbf{(M,\omega,\Phi)}$ 
and the equivariant cohomology class $A$} as 
$$ \tDH_{(M,\omega,\Phi)}(A) := \tDH_{(M,\omega,\Phi)}(\eta).$$
\end{Definition}

In the previous section we introduced
the notion of a proper Hamiltonian cobordism
between Hamiltonian $G$-manifolds equipped with equivariant cohomology classes.
Using this notion we have the following analogue
of Lemma~\ref{cobordism twisted DH}:

\begin{Lemma} \Label{coh cob twisted DH}
Let $(M_0,\omega_0,\Phi_0)$ and $(M_1,\omega_1,\Phi_1)$
be proper Hamiltonian $G$-manifolds
with equivariant cohomology classes $A_0 \in H^*_G(M_0)$
and $A_1 \in H^*_G(M_1)$.  Suppose that there exists
a proper Hamiltonian cobordism between the quadruples
$(M_0,\omega_0,\Phi_0,A_0)$ and $(M_1,\omega_1,\Phi_1,A_1)$.
Then the corresponding twisted \DH distributions are equal, i.e.,
$$ \tDH_{(M_0,\omega_0,\Phi_0)}(A_0) = \tDH_{(M_1,\omega_1,\Phi_1)}(A_1) .$$
\end{Lemma}

\begin{proof}
Let $([0,1] \times M,\tomega,\tPhi,\tA)$ be a cobording quadruple,
with diffeomorphism $-M_0 \sqcup M_1 \to \del ([0,1] \times M)$.
Let $\teta$ be an equivariant differential form on $[0,1] \times M$
that represents the class $\tA$,
and let $\eta_0$ and $\eta_1$ be its pullbacks to $M_0$ and $M_1$.
Then $\eta_0$ and $\eta_1$ represent the classes $A_0$ and $A_1$.
The result now follows from Lemma~\ref{cobordism twisted DH}
and Definition~\ref{twisted DH with coh}.
\end{proof}

In analogy with the previous section, we introduce the following
notation. 

\begin{Definition} 
With notation as in Definition~\ref{def:DHgerm}, we define
\begin{equation}\Label{twisted DH on germ} 
\tDH_{\germ_{Z_i}(M,\omega,\Phi)}^v(A \vert_{Z_i})  
 := \tDH_{(U_i, \omega_i, \Phi_i)}(A \vert_{U_i}).
\end{equation}
\end{Definition}

The justification that this notation is well-defined 
follows that given in Section~\ref{sec:DH} for
Definition~\ref{def:DHgerm} except that we use 
Lemmas~\ref{lemma:independent2} and \ref{coh cob twisted DH}.
We may now state and prove the main theorem:

\begin{Theorem}\Label{theorem:twisted-GLS-cobordism}
Let $(M,\omega,\Phi)$ be an even-dimensional Hamiltonian $G$-manifold 
without boundary.
Let $A$ be an equivariant cohomology class in $H^*_G(M)$. 
Let $v \colon M \to \g$ be a bounded taming map
and let $Z = \{ \bfv = 0 \}$ be the corresponding localizing set.
Suppose that $\Phi$ is $v$-polarized on $M$, hence on $Z$.
Let
\[
Z = \bigsqcup_{i \in \mathcal{I}} Z_i
\]
be the decomposition of the localizing set $Z$ into its connected components.
Suppose that, for every $i \in \mathcal{I}$, 
there exist arbitrarily small neighbourhoods
of $Z_i$ that admit smooth equivariant weak deformation retractions to $Z_i$.
Then
\begin{equation} \Label{main twisted}
 \tDH_{(M,\omega,\Phi)}(A) = 
\sum_i \tDH_{\germ_{Z_i}(M,\omega,\Phi)}^v(A \vert_{Z_i}). 
\end{equation}
\end{Theorem}

\begin{proof}
For each $i \in \calI$, choose an invariant neighbourhood $U_i$ of $Z_i$
that admits an equivariant smooth weak deformation retraction to $Z_i$.
Moreover, choose these neighbourhoods $U_i$ to be sufficiently small
so that their closures are disjoint.
Let $U_Z$ be the union of the neighbourhoods $U_i$.  Let $(U_Z,\omega_Z,\Phi_Z)$
be a $v|_{U_Z}$-polarized completion relative to $Z$ that is obtained
from Proposition~\ref{main cobordism}.  In particular, there exists
a proper Hamiltonian cobordism between $(U_Z,\omega_Z,\Phi_Z,A|_{U_Z})$ 
and $(M,\omega,\Phi,A)$.  By Lemma~\ref{coh cob twisted DH}, 
$(M,\omega,\Phi,A)$ and $(U_Z,\omega_Z,\Phi_Z,A|_{U_Z})$
have the same twisted \DH distribution:
\begin{equation} \Label{eq4}
   \tDH_{(M,\omega,\Phi)}(A) = \tDH_{(U_Z,\omega_Z,\Phi_Z)}(A|_{U_Z}).
\end{equation}

Let $\omega_i$ and $\Phi_i$ be the restrictions of $\omega_Z$ and $\Phi_Z$
to the component $U_i$.  Then the twisted \DH distribution 
of $(U_Z,\omega_Z,\Phi_Z,A|_{U_Z})$ 
is the sum of the twisted \DH distributions of $(U_i,\omega_i,\Phi_i,A|_{U_i})$:
\begin{equation} \Label{eq5}
   \tDH_{(U_Z,\omega_Z,\Phi_Z)}(A|_{U_Z})
 = \sum_i \tDH_{(U_i,\omega_i,\Phi_i)}(A|_{U_i}).
\end{equation}

Because $(U_i,\omega_i,\Phi_i)$ is a $v|_{U_i}$-polarized completion
of $(U_i,\omega|_{U_i},\Phi|_{U_i})$ relative to $Z_i$,
\begin{equation} \Label{eq6}
   \tDH^v_{\germ_{Z_i}(M,\omega,\Phi)}(A) 
 = \tDH_{(U_i,\omega_i,\Phi_i)}(A|_{U_i}).
\end{equation}

Equation~\ref{main twisted}
follows from~\eqref{eq4}, \eqref{eq5}, and~\eqref{eq6}.
\end{proof}

\begin{Remark}\Label{equivalent maps give same loc twisted}
As in the untwisted case (see\ Remark~\ref{equivalent maps give same loc}),
taming maps that are equivalent in the sense of Remark~\ref{rk:equivalent} 
give rise to the same localization formula. 
\end{Remark}

\section{The Brianchon-Gram polytope decomposition 
and symplectic toric manifolds}
\Label{sec:BrianchonGram}

This paper was originally motivated by a question that Shlomo Sternberg 
posed some years ago.  We first recall the context of his question in some
detail. 

As was mentioned in the introduction, it is known 
that the Atiyah-Bott-Berline-Verge localization
theorem in equivariant cohomology \cite{BerVer:1982, BerVer:1983,
AtiBot:1984}, when applied to the exponent of the equivariant
symplectic form of a compact symplectic toric manifold, yields
the measure-theoretic version of the Lawrence-Varchenko polytope
decomposition \cite{Law91, Var87}, 
when applied to the corresponding momentum polytope.  As an
example, Figure~\ref{LV example} illustrates a Lawrence-Varchenko
polytope decomposition that corresponds to localization on a $\CP^2$.

\begin{figure}[h]
\psfrag{+}{$+$}
\psfrag{=}{\Large$=$}
\psfrag{-}{$-$} 
\begin{center}
\includegraphics[height=6cm]{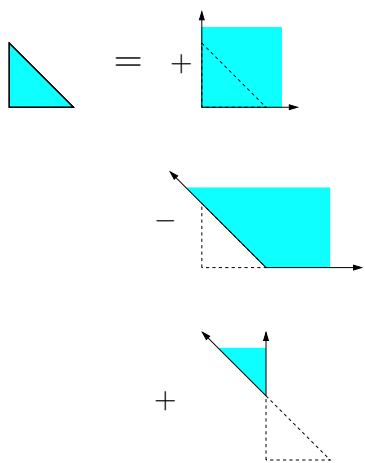}
\end{center}
\begin{center} 
\caption{A Lawrence-Varchenko decomposition of a triangle.
The summands on the right hand side correspond
to the vertices of the triangle.}
\Label{LV example}
\end{center}
\end{figure}

Motivated by this correspondence, Shlomo Sternberg pointed at a
different (and classical) 
polytope decomposition
that goes back to Brianchon and Gram \cite{Bri1837, Gram1874} (see also
\cite{She67}) and asked the following question. 

\begin{Question}\Label{question:shlomo} {\bf (Shlomo Sternberg)}
Is there a localization formula on manifolds that corresponds 
to the Brianchon-Gram polytope decomposition in the same way 
that the Atiyah-Bott-Berline-Vergne localization formula corresponds 
to the Lawrence-Varchenko polytope decomposition?
\end{Question}

As an example, Figure~\ref{BG triangle} illustrates the Brianchon-Gram
decomposition of the same polytope as in Figure~\ref{LV example}.

\begin{figure}[h]
\psfrag{+}{$+$}
\psfrag{-}{$-$}
\psfrag{=}{\Large$=$}
\begin{center}
\includegraphics[height=6cm]{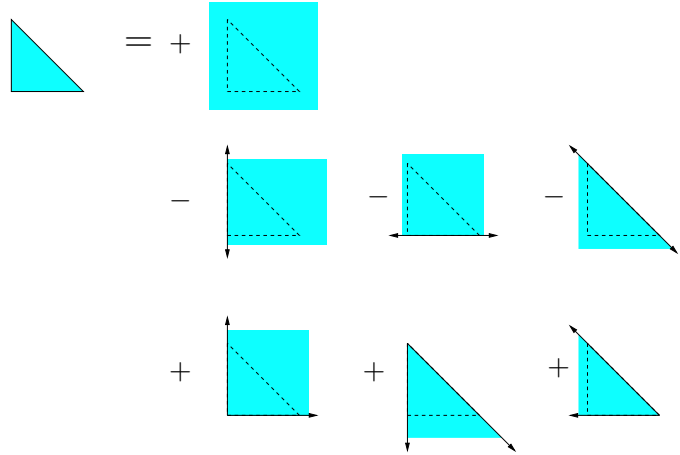}
\end{center}
\begin{center}
\caption{The Brianchon-Gram decomposition of a triangle.
The summands on the right hand side correspond to the faces 
of various dimensions of the triangle.}
\Label{BG triangle}
\end{center}
\end{figure}

\bigskip

These decompositions can be described as follows. 
Let $\Delta$ be an $n$-dimensional simple polytope in $\R^n$
(see Appendix~\ref{sec:appendix} for definitions).
The \textbf{tangent cone} of $\Delta$ at a face $F$ is defined to be
\[
C_F = \{x+\lambda(y-x) \ \vert \ y \in \Delta,\ x \in F,\ 
        \lambda \in \R_{\geq 0}\}.
\]
One may think of this as the polyhedral cone which ``a near-sighted person
would see'' if she stands at a point in the relative interior
of the face $F$. Clearly, $C_F$ is determined only by the local
structure of $\Delta$ near this point.
The (measure-theoretic version of the) Lawrence-Varchenko decomposition
of $\Delta$ can then be expressed in the equality 
\begin{equation} \Label{measure LV}
 \mu_{\Delta} = \sum_q (-1)^{\eps_q} \mu_{C_q^\sharp} 
\end{equation}
where the summation is over the vertices $q$ of $\Delta$, 
where $\mu_\Delta$
is Lebesgue measure on $\Delta$, where $\mu_{C_q^\sharp}$ is Lebesgue
measure on the cone obtained from the tangent cone to $\Delta$ at $q$
by flipping some of its edge vectors so that they all pair positively
with a pre-chosen ``polarizing vector" in the dual space, and
where $\eps_q$ is the number of edge vectors that are flipped. 
The formula~\eqref{measure LV} has a symplectic-geometric interpretation
as follows. Suppose that $\Delta$ is Delzant; this is equivalent to
the condition that $\Delta$ is the momentum polytope of a
symplectic toric manifold $M$. (See \cite{CdS01} for the definition and basic facts of symplectic toric manifolds
and Delzant polytopes.)
The Duistermaat-Heckman measure of $M$ is precisely 
$\mu_\Delta$. The fixed points for the torus action on $M$ exactly
correspond to the vertices of the polytope $\Delta$ under the momentum
map. For each fixed point $p$, the tangent space $T_pM$, with the
symplectic form, torus action, and orientation induced from those of
$M$, is the symplectic toric manifold corresponding to
the tangent cone $C_q$ of $\Delta$ at the vertex $q$ corresponding to $p$. 
This tangent space is isomorphic to $\C^n$ with
its standard symplectic form and with the torus acting by rotations of
the coordinates. Flipping the symplectic form on some of the
coordinates in $\C^n$ and flipping the corresponding summands in the
formula for the momentum map yields a symplectic vector space
$(T_pM)^\sharp$ with a torus action and with momentum image
$C_q^\sharp$. Taking its \DH measure with respect to its original
orientation, which differs from the symplectic orientation if $\eps_q$
is odd (the \DH measure is then negative), the measure-theoretic
Lawrence-Varchenko decomposition~\eqref{measure LV} becomes the
assertion that the \DH measure of $M$ is equal to that of
$$\bigsqcup_p (T_pM)^\sharp.$$
In the spirit of this manuscript, this equality of \DH measures can be
deduced from the fact that 
$M$ is cobordant to $\bigsqcup_p (T_pM)^\sharp$
as Hamiltonian $T$-manifolds (with $T=(S^1)^n$)
equipped with proper momentum maps~\cite[Chap.~4, Sec.~6]{GGK}.

We next recall the Brianchon-Gram polytope decomposition formula
\cite{Bri1837, Gram1874}. 
Let $\Delta$ be a polytope. 
The Brianchon-Gram formula is the following relation between the characteristic
functions of the polytope and of the tangent cones of its faces:
\begin{equation}\Label{eq:Brianchon-Gram}
{\bf 1}_{\Delta}(x) = \sum_F (-1)^{\dim(F)} {\bf 1}_{C_F}(x).
\end{equation}
Note that the summation is now over all faces $F$ of $\Delta$ of all
dimensions, in contrast to the Lawrence-Varchenko decomposition. 
The measure-theoretic version of this decomposition is 
\begin{equation} \Label{measure BG}
 \mu_\Delta = \sum_F (-1)^{\dim F} \mu_{C_F}
\end{equation}
where $\mu_\Delta$ is again Lebesgue measure on $\Delta$, and
$\mu_{C_F}$ is Lebesgue measure on the tangent cone $C_F$. 
Again, the formula~\eqref{measure BG} has a
symplectic-geometric interpretation as follows. Suppose again that
$\Delta$ is Delzant and let $M$ be the corresponding symplectic toric
manifold. 
The
measure $\mu_{C_F}$ is then the \DH measure of the symplectic toric
manifold $M_{C_F}$ that corresponds to the tangent cone $C_F$. When
$\dim F = \ell$, this symplectic toric manifold $M_{C_F}$ is
isomorphic to $(S^1 \times \R)^\ell \times \C^{n-\ell}$ with the
standard symplectic form and with the torus acting by rotations of the
$S^1$ factors and of the $\C$ factors. By flipping the symplectic form
on the first $\ell$ components of this product, we get an oriented
symplectic toric manifold which we denote by $M_{C_F}^\sharp$ whose momentum
image is still $C_F$. The orientation arising from the symplectic form 
is consistent with the original orientation only if $\ell$ is even.
By taking the \DH measure of
$M_{C_F}^\sharp$ with respect to the original orientation, 
\eqref{measure BG} becomes the assertion that the \DH measure of
$M_\Delta$ is equal to that of $\bigsqcup_F M_{C_F}^\sharp$. 
We show below that this assertion coincides with our localization 
formula~\eqref{main} when applied to the toric manifold $M$
with an appropriate taming map.

\bigskip

Throughout this section we work with an identification 
\(\t \cong \R^n \cong \t^* \).  
Suppose as above that $\Delta$ is a Delzant polytope and let $M$ be the 
corresponding symplectic toric manifold. 
It turns out that the measure-theoretic formula~\eqref{measure BG} 
is what we obtain 
from Theorem~\ref{theorem:GLS-cobordism} when applied to $M$ with a
taming map that comes from a function that satisfies the conditions
described in the following lemma.

\begin{Lemma}\Label{lemma:bump-function}
Let $\Delta \subset \R^n$ be an $n$-dimensional simple polytope. Then
there exists an open neighbourhood $U$ of $\Delta$ in $\R^n$ and a
smooth function $\rho \colon U \to \R$ with the following properties: 
\begin{enumerate} 
\item 
For each
face $F$ of $\Delta$, the restriction $\rho|_{\relint(F)}$ of $\rho$ to the
relative interior of $F$ has a unique critical point $x_F$.
\item 
Let $F$ be an $\ell$-dimensional face of $\Delta$,
and let $x_F$ be the critical point of $\rho|_{\relint(F)}$.
Then there exist $\veps > 0$
and affine coordinates $(x_1,\ldots,x_\ell,y_1,\ldots,y_{n-\ell})$
on $\R^n$ with respect to which
\begin{enumerate} 
\item 
   the point $x_F$ becomes the origin, and a neighbourhood $x_F$ in $F$ 
   becomes the set 
   $U_{x_F} := (-\veps, \veps)^\ell \times (-\veps,0]^{n-\ell}$ 
   for some $\veps > 0$.
\item
The function $\rho|_{U_{x_F}}$ becomes
\[
\rho(x_1, \ldots, x_\ell, y_1, \ldots, y_{n-\ell}) 
   = \sum_{j=1}^\ell x_j^2 + \sum_{j=1}^{n-\ell} y_j
\]
after composing it with an affine map of $\R$
(i.e., multiplying by a constant and adding a constant).
\end{enumerate} 
\end{enumerate}
\end{Lemma}

The proof of this lemma is technical and unrelated to the arguments
in this section so we relegate it to an appendix.
We prove (an equivalent version of) the lemma in Appendix~\ref{sec:appendix}
as Parts (A) and (C) of Proposition~\ref{prop:bump-function new}.

Now suppose that $(M,\omega,\Phi)$ is a symplectic toric $T$-manifold
with momentum polytope $\Delta = \Phi(M)$. 
Let $\rho$ be a function as specified in Lemma~\ref{lemma:bump-function},
and let $v = d(-\rho) \circ \Phi \colon M \to \t$ be the taming map
corresponding to $-\rho$.
Let $Z = \{ v^\sharp=0 \}$ be the corresponding localizing set.
We begin with the following observation.

\begin{Lemma}\Label{lemma:Z face} 
$$ Z = \bigsqcup_F Z_F ,$$
where the union is over all the faces $F$ of $\Delta$,
and where $Z_F = \Phi\Inv( \{ x_F \} )$.
Moreover, every $Z_F$ consists of exactly one $T$-orbit. 
\end{Lemma}

\begin{proof}
We have
\begin{equation} \Label{Z is}
Z = \Crit(\rho \circ \Phi) 
  = \bigsqcup_F \{ x \in M : \Phi(x)  \textup{ is a critical
  point of } \rho|_{\relint(F)} \}.
\end{equation}
Indeed, the first equality 
is the content of~\eqref{Z=Crit}, 
and the second equality follows from Lemma~\ref{lemma:crit-rho},
since the orbit type strata of a symplectic toric manifold
are exactly the preimages of the relative interiors 
of the faces of its momentum polytope.
By the construction of $\rho$, 
(specifically property (1) of Lemma~\ref{lemma:bump-function}),
the term in the union~\eqref{Z is} that corresponds to the face $F$
is exactly $\Phi\Inv(\{ x_F \})$.
Because in a toric manifold the momentum level sets
are exactly the $T$-orbits, $\Phi\Inv(\{ x_F \})$ is a $T$-orbit.
\end{proof}

Our next task is to explicitly construct a neighbourhood $U_Z$ of $Z$
and a $v$-polarized completion $(U_Z, \omega_Z, \Phi_Z)$ 
of $(U_Z, \omega \vert_{U_Z}, \Phi \vert_{U_Z})$ 
relative to $Z$ for which an
application of Theorem~\ref{theorem:GLS-cobordism} and a concrete
computation of the right hand side of~\eqref{main} 
for our choice of $(U_Z, \omega_Z, \Phi_Z)$ 
yield the measure-theoretic Brianchon-Gram formula. 
By Lemma~\ref{lemma:Z face}, we can construct $U_Z$
as a disjoint union, over faces $F$, of neighbourhoods $U_F$
of $Z_F$, and we can construct the polarized completion
separately on each $U_F$.
The following result is the main technical tool that we need:

\begin{Proposition}\Label{proposition:brianchon-gram}
Let $(M,\omega,\Phi)$ be a compact connected symplectic toric manifold
with momentum polytope $\Delta = \Phi(M)$.
Let $\rho \colon \Delta \to \R$ be a smooth function
as in Lemma~\ref{lemma:bump-function}, and 
let $v \colon M \to \t$ be the taming map corresponding to $-\rho$. 
Let $F$ be a face of $\Delta$ and
let 
\[
C_F = \{ x + \lambda(y-x) \ | \ y \in \Delta, x \in F, \lambda \in
\R_{\geq 0}\}
\]
be the tangent cone of $\Delta$ at $F$. 
Let $Z_F$ be the component of the localizing set
that corresponds to the face $F$ as in Lemma~\ref{lemma:Z face}.
Then there exist
\begin{itemize}
\item an arbitrarily small $T$-invariant tubular neighbourhood $U_F$ of $Z_F$;
\item a $v$-polarized completion $(U_F,\omega_F,\Phi_F)$ of
  $(U_F,\omega|_{U_F}, \Phi|_{U_F})$; and
\item an isomorphism of (oriented) Hamiltonian $T$-manifolds 
  between $(U_F,\omega_{F},\Phi_{F})$ and
  the symplectic toric manifold $(M_{C_F}, \omega_{C_F}, \Phi_{C_F})$ 
corresponding to $C_F$, which carries the orientation on $U_F$
to the symplectic orientation on $M_{C_F}$ if $\dim F$ is even and
to the opposite of the symplectic orientation on $M_{C_F}$ if $\dim F$ is odd.
\end{itemize}
\end{Proposition}

\begin{proof}
Let $\ell$ denote the dimension of $F$.
By Lemma~\ref{lemma:bump-function}
we assume that the affine span of $F$ 
is $\R^\ell \times \{0\} \subseteq \R^n$, the critical point
$x_F$ is the origin $0$, the polytope $\Delta$ 
coincides near $x_F$ with the sector
$\R^\ell \times \R^{n-\ell}_{\leq 0}$,
and the function $\rho$ near $x_F$ is of the form 
\begin{equation}\Label{eq:formula for rho}
\rho(x_1, \ldots, x_\ell, y_1, \ldots, y_{n-\ell}) = \sum_{j=1}^\ell
x_j^2 + \sum_{j=1}^{n-\ell} y_j.
\end{equation}
The tangent cone $C_F$ is the sector
$$ \{ (x_1,\ldots,x_\ell,y_1,\ldots,y_{n-\ell}) \ | \ 
        y_1,\ldots,y_{n-\ell} \leq 0 \} .$$
The corresponding symplectic toric manifold is
$$ M_{C_F} = (S^1 \times \R)^\ell \times \C^{n-\ell} $$
where the torus $T \cong T^\ell \times T^{n-\ell}$ acts by rotating
the $S^1$ coordinates and the $\C$ coordinates.  
The symplectic form on $M_{C_F}$, which we denote $\omega_{\std}$,
is the split form which is equal to $d\theta \wedge dt$ on every
cylinder 
(parametrized as $\{ (e^{i\theta},t) \}$)
and is standard on the $\C^{n-\ell}$ factor.
The momentum map is
$$ \Phi_{C_F} 
 \left( (e^{i\theta_1},t_1) , \ldots , (e^{i\theta_\ell},t_\ell) , 
                  z_1, \ldots , z_{n-\ell} \right) 
 = \left( t_1 , \ldots , t_\ell , -\frac{|z_1|^2}{2} , \ldots ,
            -\frac{|z_{n-\ell}|^2}{2} \right) .$$

The local normal form theorem identifies a neighbourhood $U_F$ of $Z_F$ 
in $M$ with the open subset
\begin{equation}\Label{eq:circle and disk} 
(S^1 \times (-\varepsilon, \varepsilon))^\ell \times
(D^2_\varepsilon)^{n-\ell} 
\end{equation} 
of $M_{C_F}$, for some $\veps > 0$.
Here $D^2_\varepsilon$ is a disc with momentum image $(-\varepsilon,0]$.
Thus its radius is $\sqrt{2\varepsilon}$.

By the explicit formula~\eqref{eq:formula for rho}
for $\rho$, and identifying $\t$ and $\t^*$ with $\R^n$,
if $\veps$ is sufficiently small,
the identification of $U_F$ with the open subset~\eqref{eq:circle and disk} 
of $M_{C_F}$ carries the taming map $v$ to the pullback via $\Phi_{C_F}$
of
\begin{equation}\Label{eq:drho in general}
(x_1, \ldots, x_\ell, y_1, \ldots, y_{n-\ell}) \mapsto (-2x_1, \ldots,
-2x_\ell, -1, \ldots, -1). 
\end{equation}
Denote this pullback $v_{\std}$.
The pairing of the momentum map $\Phi_{C_F}$ with the taming map 
$v_{\std}$ is the function
$$ - \sum_{j=1}^\ell 2 t_j^2 + \sum_{j=1}^{n-\ell} \frac{|z_j|^2}{2},$$
which is neither proper nor bounded from below.
Its restriction to the open subset~\eqref{eq:circle and disk}
is bounded but is not proper.

We now equip $M_{C_F}$ with the split symplectic form, 
which we denote $\omega_{C_F}^\sharp$,
which is the \emph{negative} of $d\theta \wedge dt$
on every cylinder component and remains standard on the $\C^{n-\ell}$
component.
This symplectic form $\omega_{C_F}^\sharp$ is consistent with the
original symplectic orientation if $\ell$ is even and inconsistent 
if $\ell$ is odd, and it has the momentum map
$$ \Phi_{C_F}^\sharp = \left( -t_1, \ldots , -t_\ell , 
    - \frac{|z_1|^2}{2} , \ldots , -\frac{|z_{n-\ell}|^2}{2} \right).$$

We will now describe an equivariant diffeomorphism from $U_F$ to $M_{C_F}$
under which the pullbacks of $\omega_{C_F}^\sharp$ and $\Phi_{C_F}^\sharp$
coincide with $\omega$ and $\Phi$ on $Z_F$
(in fact, their further pullbacks to $Z_F$ are zero)
and under which the pullback of $\Phi_{C_F}^\sharp$ is $v$-polarized.

Let \(g \colon (-\varepsilon, \varepsilon) \to \R\) be a diffeomorphism 
such that \(g(-x) = -g(x)\) for all $x$ and such that \(g(x) = x\) on
a neighbourhood of $x=0$. Consider the diagram
  \begin{equation} \Label{eq:reparam}
  \xymatrix @C=0.9in{ 
  (S^1 \times (-\varepsilon, \varepsilon))^\ell 
       \times (D^2_\veps)^{n-\ell} \ar[r]^{\Phi_{C_F}}  \ar[d]_\psi &
  (-\varepsilon, \varepsilon)^\ell 
       \times (-\varepsilon, 0]^{n - \ell} \ar[d] \\
  (S^1 \times \R)^\ell \times (\C)^{n-\ell}
  \ar[r]^{\Phi_{C_F}^\sharp} & \R^\ell \times (-\infty, 0]^{n-\ell}
 }    
  \end{equation}
in which the right vertical map is
\[
(x_1, \ldots, x_\ell, y_1, \ldots, y_{n-\ell}) \mapsto (g(-x_1),
\ldots, g(-x_\ell), g(y_1), \ldots, g(y_{n-\ell}))
\]
and the left vertical map $\psi$ is the map that acts 
on the first $\ell$ factors as 
$(e^{i\theta_k}, x_k) \mapsto (e^{i\theta_k}, g(x_k))$ 
and on the last $n-\ell$ coordinates by
\(z_j = r_j e^{i \theta_j} \mapsto 
         \sqrt{ 2 g (r_j^2/2) } e^{i\theta_j}.\) 
From this explicit description of $\psi$ it follows
that $\psi$ is a $T$-equivariant diffeomorphism
and that the diagram~\eqref{eq:reparam} commutes.

Let $\omega_F := \psi^* \omega_{C_F}^\sharp$
and $\Phi_F := \psi^* \Phi_{C_F}^\sharp $.
The Hamiltonian $T$-manifold
$$ \left( U_F , \omega_F, \Phi_F \right) $$
is isomorphic to $(M_{C_F},\omega_{C_F}^\sharp,\Phi_{C_F}^\sharp)$
since 
the equivariant diffeomorphism $\psi$ provides such an isomorphism.
To finish the proof we must show that $\Phi_F$ is $v$-polarized.
Since the diagram~\eqref{eq:reparam} commutes, this is equivalent
to showing that the composition of the top horizontal arrow
with the right vertical arrow in~\eqref{eq:reparam} is $v$-polarized.
Recall that 
the taming map $v$ is the pullback via $\Phi_{C_F}$ 
of~\eqref{eq:drho in general}.
Since the momentum map $\Phi_{C_F}$, taken with the domain
and codomain as in the top horizontal arrow of~\eqref{eq:reparam},
is proper, it is enough to show that the pairing
of the map~\eqref{eq:drho in general} with the 
right vertical arrow in~\eqref{eq:reparam} is proper 
and bounded from below.
This pairing is the map
$$ (-\veps,\veps)^\ell \times (-\veps,0]^{n-\ell} \to \R $$
that is given by the formula
\begin{align}
 (x_1,\ldots,x_\ell,y_1,\ldots,y_{n-\ell})
    & \mapsto \left( -2x_1,\ldots,-2x_\ell,-1,\ldots,-1 \right)
      \cdot \left( g(-x_1), \ldots, g(-x_\ell), 
                  g(y_1), \ldots, g(y_{n-\ell}) \right) \nonumber \\
   & = 2 \sum_{j=1}^\ell x_j g(x_j) - \sum_{j=1}^{n-\ell} g(y_j).
                           \Label{dot product}
\end{align}
Notice that $y_j$ takes values in $(-\veps,0]$,
and $-g(y_j)$ is nonnegative and approaches $\infty$ as $y_j$
approaches $-\veps$.
Also, $x_j$ takes values in $(-\veps,\veps)$,
and $x_j g(x_j)$ is nonnegative and approaches $\infty$
as $|x_j| \to \pm \veps$.
So the function~\eqref{dot product} is nonnegative,
and for every $L$
there exists $\delta$ such that $0 < \delta < \veps$ 
and such that both $tg(t)$ and $g(t)$ are $\geq L$ 
whenever $\delta \leq |t| < \eps$,
and so the preimage of $[0,L]$ under the function~\eqref{dot product} 
is contained in the compact subset
$[-\delta,\delta]^\ell \times [-\delta,0]^{n-\ell}$
of the domain $(-\veps,\veps)^\ell \times (-\veps,0]^{n-\ell}$.
This shows that the function is proper and bounded from below.
\end{proof} 

This proposition allows us to identify the left and right hand sides
of our localization formula~\eqref{main}, applied to the symplectic
toric manifold $M$ and the taming map obtained from the function $\rho$,
with the left and right hand sides of the 
Brianchon-Gram decomposition~\eqref{measure BG},
applied to the momentum polytope $\Delta$ of $M$.
The right hand side of our localization formula~\eqref{main}
is a summation over the components of the localizing set.
These components exactly correspond to the faces $F$ of $\Delta$.
By Proposition~\eqref{proposition:brianchon-gram},
a neighbourhood of the component that corresponds to the face $F$
has a polarized completion that is isomorphic to the symplectic toric manifold
$M_{C_F}$ with an orientation that is consistent with its symplectic form
if and only if $\dim F$ is even. Thus, the localization formula~\eqref{main}
in this case becomes the equality
$$ \tDH_M = \sum_F (-1)^{\dim F} \tDH_{M_{C_F}}.$$
Since the \DH measure of a symplectic toric manifold is precisely
Lebesgue measure on its momentum polytope,
this equality is precisely the Brianchon-Gram equality~\eqref{measure BG}.

\begin{Remark}
Our results are not the first to relate the Brianchon-Gram polytope 
decomposition to localization. 
A partial answer to Question~\ref{question:shlomo} is given by 
localization theory using the norm-square $\|\Phi\|^2$ of a momentum
map for a
Hamiltonian $G$-space, as developed by Paradan \cite{Par00} 
and Woodward~\cite{Woo05}, 
following Witten~\cite{Witten:1992}.
Indeed, the localization formula with respect to $\|\Phi\|^2$,
applied to the exponent of the 
equivariant symplectic form of a symplectic toric manifold,
yields the Brianchon-Gram decomposition under the following assumption
on the momentum polytope $\Delta$: 
\begin{equation}\Label{assumption}
\begin{minipage}{0.7\linewidth}
for every face $F$ of $\Delta$, of any dimension, 
the point of  $F$ that is closest to the origin
lies in the relative interior of $F$.
\end{minipage}
\end{equation}
This correspondence between the localization formula for $\|\Phi\|^2$
and the Brianchon-Gram decomposition was also observed by Jonathan
Weitsman and was worked out by Agapito and Godinho in \cite{AgaGod05}.
Moreover, when the assumption~\eqref{assumption} on $\Delta$ fails,
Agapito and Godinho show that the localization formula for the
norm-square of the momentum map corresponds to a \emph{new} polytope
decomposition that is different from Brianchon-Gram's. 
\end{Remark}

We close by addressing the issue of the difference between the
Brianchon-Gram formula and its
measure-theoretic version. 

\begin{Remark}
Although the Brianchon-Gram formula~\eqref{eq:Brianchon-Gram} 
can be proved directly, it can also be derived from 
its measure-theoretic version, \eqref{measure BG}.
Because the measures that appear in~\eqref{measure BG}
are constant multiples of Lebesgue measure
outside the union of the affine spans of the facets of $\Delta$,
this measure-theoretic formula
implies the formula~\eqref{eq:Brianchon-Gram} 
whenever $x$ is outside the union of these affine spans.
To prove~\eqref{eq:Brianchon-Gram} for an arbitrary~$x$,
we apply the measure-theoretic formula~\eqref{measure BG} 
to the polytope that is obtained from $\Delta$
by shifting its facets outward
by an amount (depending on $x$) that is small enough
to not affect the values at~$x$ of the left and right hand sides
of~\eqref{eq:Brianchon-Gram}.
\end{Remark}

\section{Example: a circle action on the $2$-sphere}\label{sec:sphere}

As an illustration of our methods, we now work out in detail the case of $S^1$ acting on 
the unit sphere $S^2 \subset \R^3$ with the standard rotation
action. We begin by setting some notation. 
The area form on $S^2$ can be written in cylindrical coordinates 
as $\omega = d\theta \wedge dh$ where $h \colon S^2 \to \R$
is the height function; this equips $S^2$ with the standard orientation. The $S^1$ action is generated by the
vector field $\del/\del\theta$. 
We identify the Lie algebra $\Lie(S^1)$ and its dual $\Lie(S^1)^*$
with $\R$ so that the exponential map becomes 
$\theta \mapsto e^{i\theta}$.
Then Hamilton's equation~\eqref{eq:Hamiltons} becomes
$d \Phi = \iota(\del/\del\theta) \omega$, and the momentum map 
can be given by the height function 
\(\Phi = h,\) as indicated in Figure~\ref{fig:S2}.

\begin{figure}[h]
\psfrag{h}{$h$}
\psfrag{R}{$\R$}
\begin{center}
\includegraphics[height=4cm]{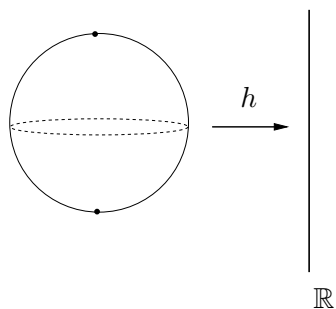}
\end{center}
\begin{center}
\parbox{4.25in}{
\caption{The standard action of $S^1$ on $S^2$ has momentum map the
  height function \(\Phi = h.\) The equator has value \(h=0,\) the north
  pole has \(h=1,\) and the south pole has \(h=-1.\)}\label{fig:S2}
}
\end{center}
\end{figure}

Below, we apply our 
localization theorem to three 
different choices of taming map \(v \colon S^2 \to \Lie(S^1)^* \cong \R,\)
obtaining as a consequence three different polytope
decompositions. The first example recovers the classical
Duistermaat-Heckman theorem and hence the measure-theoretic Lawrence-Varchenko
decomposition. The second is the decomposition given by Woodward's
localization with respect to $\|\Phi\|^2$, and finally, the third is
the measure-theoretic Brianchon-Gram decomposition.

\subsection{Example: a constant taming map}

We first consider the case corresponding to
Example~\ref{example:Prato-Wu}, i.e., where the taming map $v$ 
is equal to a constant $\eta \in \Lie(S^1) \cong \R$. 
For concreteness we take \(\eta =
1.\) In this case \(Z = \{\eta^{\sharp} =
0\} = (S^2)^{S^1} = \{N, S\},\) so the localizing set is the classical 
localizing set consisting of the fixed points of the action. 
An equivariant tubular neighbourhood $U_Z$
of $Z$ consists of two components $U_N$ and $U_S$, equivariant
neighbourhoods of the north and south poles respectively. In order to apply
Theorem~\ref{theorem:GLS-cobordism} we must choose $v$-polarized completions of
$(\Phi = h, \omega = d\theta \wedge dh)$ on both $U_N$ and $U_S$. We
first consider the north pole $N$. By definition, a $v$-polarized completion 
$(U_N,\omega_N,\Phi_N)$ of $(\N, \omega \vert_{U_N}, \Phi \vert_{U_N})$ 
must satisfy $\Phi_N(N) = \Phi(N)=h(N) = 1$ 
and $\Phi_N^v = \Phi_N$ must be proper and bounded below. 
(There is no condition on $\omega_N$ because $\{N\}$ is
$0$-dimensional so the restriction of any $2$-form to that component of
$Z$ is $0$.)

In order to make explicit computations, we choose an
orientation-preserving $S^1$-equivariant diffeomorphism 
(not symplectomorphism)
from an open neighbourhood $U_N$ of $N$ to (all of) $\C$, 
equipped with its standard
orientation and $S^1$-action.  The momentum map for the standard
symplectic form on $\C$ is $- \half \| z \| ^2$ (up to a constant),
which is not bounded below.  To correct this, we therefore equip $\C$
with the \emph{negative} of the standard symplectic form, 
$- \omega_{\std} = -dx \wedge dy$, and we take the momentum map
\begin{equation}\label{eq:PhiN}
 \Phi_N(z) = 1+ \half \| z \|^2, 
\end{equation}
which is both bounded below and proper. Hence we can take 
$(U_N, \omega_N, \Phi_N)$ to be given by~\eqref{eq:PhiN} 
and $\omega_N = - \omega_{\std}$. 
Because integration of $- \omega_{\std} = -dx \wedge dy$ 
with respect to the standard
orientation takes negative values, the corresponding Duistermaat-Heckman
measure is \emph{negative} Lebesgue measure on the ray $[1,\infty)$
  and zero outside the ray.

  Similarly, a neighbourhood $U_S$ of the south pole can be identified via an orientation-preserving $S^1$-equivariant diffeomorphism
  with $\C$ with its standard orientation and the \emph{opposite}
  $S^1$-action:
  $$ \lambda \colon z \mapsto \lambda^{-1} z .$$
  The momentum map for this action is, up to a constant, 
  $ + \frac{1}{2} \|z\|^2$, which is already proper and 
  bounded below. To obtain the condition $\Phi_S(S) = \Phi(S)=h(S) = -1$ 
  we define 
  \[
  \Phi_S(z) = -1 + \half \|z\|^2.
  \]
  Here we take the standard symplectic form (and not its negative),
  so the contribution from the south pole is {\em positive} Lebesgue
  measure on the ray $[-1,\infty)$ and zero outside.

Hence we get the decomposition
of the Duistermaat-Heckman measure of $S^2$ as illustrated in the
following Figure~\ref{fig:DHforS2-v-const}. This corresponds to the 
Lawrence-Varchenko polytope decomposition of the interval $[-1,1]$.

\begin{figure}[h]
\psfrag{=}{$=$} \psfrag{-}{$-$} \psfrag{+}{$+$}
\begin{center}
\includegraphics[height=4cm]{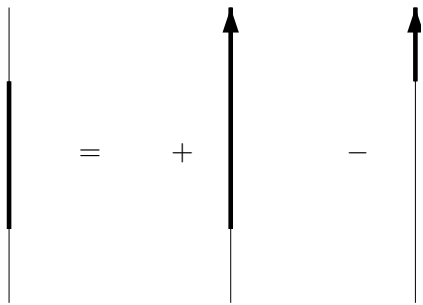}
\end{center}
\begin{center}
  \parbox{4.25in}{ \caption{By choosing $v \equiv 1$ constant and
      applying the localization formula, we obtain the Lebesgue measure
      on the interval $[-1,1]$ as the difference of Lebesgue measures
      on two rays.
      This corresponds to a Lawrence-Varchenko decomposition.
    }\label{fig:DHforS2-v-const}
}
\end{center}
\end{figure}

\subsection{Example: the norm-square of the momentum map}

We now consider the case corresponding to
Example~\ref{example:norm-square}, i.e., 
where \(v = \widehat{\Phi} = h.\) 
In this case, the zero set $Z := \{x \in S^2: v^{\sharp}_x = 0\}$
of $v^\sharp$ is \(\{N\} \cup \{S\} \cup \{h=0\},\) so we have an
additional component of $Z$ corresponding to the equator in~$S^2$.

We begin our computations with the north pole.  As in the previous
example, we must construct a $v$-polarized completion $(\Phi_N,
\omega_N)$ of $(\Phi|_{U_N}, \omega|_{U_N})$ on $U_N$.  We may assume
that $U_N$ is contained in the upper quarter of the sphere,
$\{ 1/2 < h \leq 1 \}$, so $\Phi_N^v$ is
between $\half \Phi_N$ and $\Phi_N$. 
So $\Phi_N^v$ is proper and bounded from below 
if and only if $\Phi_N$ is proper and bounded from below. 
The same
analysis as in the previous example applies and we take
  \[
  \Phi_N(z) = 1+ \half \|z\|^2
  \]
  with $U_N \cong \C$ equipped with the negative of the standard
  symplectic form. Hence the 
  contribution from the north pole is the negative Lebesgue measure on
  $[1,\infty)$, as in the previous example.

  In the case of the south pole, however, the analysis is different
  from the previous example since we now have $v \approx -1 < 0$ 
  near the south pole. We may assume that $U_S$ is contained in the
  bottom quarter of the sphere, $\{ -1 \leq h \leq -1/2 \}$.
  Thus, in order to satisfy the $v$-polarization condition,
  we must construct $\Phi_S$ such that its {\em negative} $- \Phi_S$
  is proper and bounded from below. A similar analysis as in the previous
  case then shows that we may take the negative of the standard
  symplectic form on $ \C$ and momentum map
  \[
  \Phi_S(z) = -1 - \frac{1}{2} \|z\|^2,
  \]
  with contribution negative Lebesgue measure on the ray $(-\infty, -1]$. 
 
  We now consider the contribution from the equator.  
  The neighbourhood $\{ -\frac14 < h < \frac14 \}$ of the equator in $S^2$
  is (non-symplectically)
  equivariantly diffeomorphic to the
  cylinder \(S^1 \times \R,\) with coordinates \((\theta, s).\) The
  action of $S^1$ is by standard multiplication on the left component
  of \(S^1 \times \R.\) By Hamilton's equation~\eqref{eq:Hamiltons},
  and using the standard orientation of the cylinder given by the
  symplectic form \(\omega_E = d\theta \wedge ds,\) the momentum map
  $\Phi_E$ is given by \(\Phi_E(\theta,s) = s,\) i.e., projection onto
  the second factor. This momentum map satisfies $\Phi_E |_{s=0} = 0$,
  so it agrees with the height function at the equator, as
  required. 
  Moreover, 
  when $s >\!\!> 0 $ we have $v \approx \frac14$, 
  and when $s <\!\!< 0 $ we have $v \approx -\frac14$.
  So we have
  \(\Phi_E^v(\theta,s)
  \approx \frac14 s \) for \(s >\!\!> 0\) and \(\Phi_E^v(\theta,s) \approx
  (-\frac14) s \) for \(s <\!\!< 0.\) 
  So $\Phi_E^v(\theta,s) \approx \frac14 |s|$ for $|s| >\!\!>0$,
  and hence $\Phi_E$ is also $v$-polarized.  The
  orientation of $\omega_E$ is the same as the orientation induced
  from the standard orientation of $S^2$ restricted to $U_E$, so this
  term will appear with {\em no} sign change. Hence the contribution
  from the equator is positive Lebesgue measure on all of $\R$.

  In summary, we get that the Duistermaat-Heckman measure for the
  $S^1$-action on $S^2$ may be written in terms of these three
  contributions, as given in Figure~\ref{fig:DHforS2-v-Phi}. This is
  the decomposition corresponding to the localization via the
  norm-square of the momentum map.

\begin{figure}[h]
\psfrag{=}{$=$} \psfrag{-}{$-$} \psfrag{+}{$+$}
\begin{center}
\includegraphics[height=4cm]{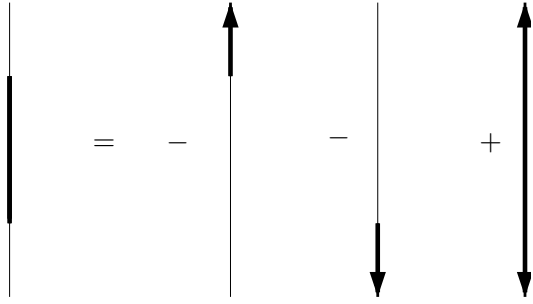}
\end{center}
\begin{center}
\parbox{4.25in}{
\caption{By choosing $v = \hat{\Phi} = h$, 
  we obtain the Lebesgue measure on the interval $[-1,1]$ 
  as the sum of three contributions as indicated. 
  This corresponds to localization via the norm-square of the momentum map.
}\label{fig:DHforS2-v-Phi}
}
\end{center}
\end{figure}

\subsection{Example: the negative of the norm-square of the momentum map} 

Finally, we consider the example where we pick \(v = -\widehat{\Phi} =
-h.\) Since $S^2$ is compact, $\Phi^v = -h^2$ is bounded below so
$\Phi=h$ is $v$-polarized on 
$S^2$. We have \(Z = \{N\} \cup \{S\} \cup \{h=0\}\) as in the
previous case.  The analysis of the components of an equivariant
tubular neighbourhood $U_Z = U_N \cup U_S \cup U_E$ is exactly
analogous to the previous case and we do not go through the details
here, and only note that the choices of direction will differ
because of the sign change in $v$.

This choice of $v$ yields the decomposition of the
Duistermaat-Heckman measure
as illustrated in Figure~\ref{fig:DHforS2-v-negPhi}. 
This is the Brianchon-Gram decomposition.

\begin{figure}[h]
\psfrag{=}{$=$} \psfrag{-}{$-$} \psfrag{+}{$+$}
\begin{center}
\includegraphics[height=4cm]{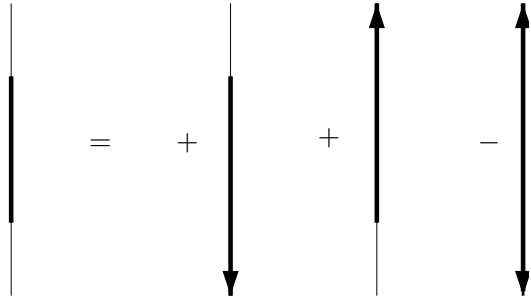}
\end{center}
\begin{center}
\parbox{4.25in}{
\caption{By choosing $v =  - \hat{\Phi} = -h$, we obtain
  the Lebesgue measure on the interval as the sum of three
  contributions as indicated. This corresponds to the Brianchon-Gram
  decomposition. 
}\label{fig:DHforS2-v-negPhi}
}
\end{center}
\end{figure}

\appendix

\section{Construction on a simple polytope of a smooth function 
with a unique critical point on the relative interior of each face.}
\label{sec:appendix}

In this appendix we prove the technical
Proposition~\ref{prop:bump-function new}, which asserts the existence
of a smooth function on a simple compact polytope $\Delta$ with
certain prescribed properties, the most important of which is that it
has a unique critical point on the relative interior of every face.
The existence of a function with such critical points is intuitively
quite clear but, firstly, we could not find a reference, and secondly,
it is surprisingly difficult to prove rigorously. Our approach is to
use brute-force differential topology on $\R^n$.  In fact, our
explicit construction yields a function which, near a critical point
in a face $F$, is linear in coordinates transverse to the face, and
quadratic in coordinates along the face. These specific properties of
our construction are used in Section~\ref{sec:BrianchonGram}.
Moreover, although we do not explicitly use these properties in this
manuscript, we can also specify in advance the location and the
function value of each of the critical points in the
faces of $\Delta$.

We begin with some terminology and notation. 
By a \textbf{polytope} $\Delta$ we mean the
convex hull of a finite set of points in a vector space (or in an affine space).
In particular, our polytopes are always compact.
A \textbf{face} $F$ of $\Delta$ is its intersection with a supporting
hyperplane: $F = \Delta \cap \{ L = \lambda \}$
where $L$ is a linear functional and $L|_\Delta \geq \lambda$.
The \textbf{dimension} of a polytope is the dimension of its affine span.
Faces of a polytope are themselves polytopes. \textbf{Facets}
are faces of codimension one.  Every face is an intersection of facets.
A polytope is \textbf{simple} if every face of codimension $k$
is contained in no more than $k$ facets.
Given a convex subset $X$ of $\R^n$, 
we denote by $\relint(X)$ the relative interior of $X$, i.e., 
the interior of $X$ in the affine span of $X$. 
Finally, we
recall that a function $f$ defined on an arbitrary subset of $X$ of
$\R^n$ (e.g.\ on a polytope) is said to
be \textbf{smooth} if, near each point \(x \in X,\) there exists a
smooth extension $f_U$ of $f$ to an open neighbourhood $U$ of $x$ such
that $f_U \vert_{U \cap X} = f \vert_{U \cap X}$. In the case that $X$
is closed as a subset of $\R^n$, it then follows that $f$ extends to a
global smooth function on all of $\R^n$. 

The following technical proposition records the results of our
explicit construction:

\begin{Proposition}\label{prop:bump-function new}
Let $\Delta \subset \R^n$ be an $n$-dimensional simple polytope. 
Then there exists a smooth function $f \colon \Delta \to \R$ 
with the following properties.
\begin{enumerate}
\item[(A)] For every face $F$ of $\Delta$,
the minimal value of $f$ on $F$ is attained
at exactly one point in the relative interior of $F$,
and the restriction of $f$ to the relative interior of $F$
has no other critical points.
\item[(B)] Suppose that we are given 
the data $\{(x_F, \alpha_F): F \textup{ a face in } \Delta\}$, 
where for each face $F$ we have $x_F \in \relint(F)$
and $\alpha_F \in \R_{\geq 0}$, and where $\alpha_F < \alpha_{F'}$ whenever
$F'$ is a proper subface of $F$. Then the function $f$ may be
chosen such that, for each face $F$ of $\Delta$, the restriction $f
\vert_{\relint(F)}$ attains its minimum at the chosen point $x_F$, and
the minimum value is the chosen $\alpha_F$, i.e. \(f(x_F) = \alpha_F.\) 
\item[(C)] Let $F$ be an $\ell$-dimensional face of $\Delta$. 
Near the above given point $x_F \in F$, there exist $\veps > 0$, 
a neighbourhood $U_{x_F}$ of $x_F$ in $F$, and affine local coordinates
$(x_1,\ldots, x_\ell, y_1, \ldots, y_{n-\ell}) \in (-\veps, \veps)^n$ 
on $U_{x_F}$ such that 
\begin{enumerate} 
\item 
the neighbourhood $U_F$ is given by
$\{(x_1, \ldots, x_\ell, y_1,\ldots, y_{n-\ell}) 
      \in (-\veps, \veps)^\ell \times [0,\veps)^{n-\ell}\}$ 
and the point $x_F$ is given by the origin 
$(0,0) \in (-\veps, \veps)^\ell \times [0,\veps)^{n-\ell}$;
\item 
the function $f$ can be chosen such that with respect to these coordinates, 
$f \vert_{U_{x_F}}$ is of the form 
\[
f(x_1, \ldots, x_\ell, y_1, \ldots, y_{n-\ell}) = \sum_{j=1}^\ell
x_j^2 - \sum_{j=1}^{n-\ell} y_j
\]
up to an affine translation in $\R$ (i.e. up to a multiplication and
translation by constants).
\end{enumerate}
\end{enumerate}
\end{Proposition}

In order to prove Proposition~\ref{prop:bump-function new}, we 
construct the required function $f$ by a recursive procedure. Before
delving into technicalities, we first sketch our method. Recall that
the \textbf{$\ell$-skeleton} of a polytope $\Delta$ is the union of
its $\ell$-dimensional faces. We begin the inductive argument by
constructing an appropriate function $f_0$ on a neighbourhood in $\R^n$
of the $0$-skeleton, i.e., of the vertices of $\Delta$. Such a function will
automatically have a unique critical point
on the relative interior of each vertex, because 
these relative interiors are just single points. Continuing by
induction, suppose that we are given a function $f_{\ell-1}$, satisfying
appropriate technical conditions (to be specified below) on a
neighbourhood in $\R^n$ of the $(\ell-1)$-skeleton. After possibly
shrinking this neighbourhood, we show that there exists a function
$f_\ell$ on a neighbourhood of the $\ell$-skeleton extending
$f_{\ell-1}$ and satisfying similar technical conditions. Continuing in
this manner, at the final step we then obtain a function $f := f_n$,
defined on all of $\Delta$ and that has the desired properties. The
concrete implementation of this plan occupies the rest of this
appendix.

At each 
inductive step the functions $f_\ell$ that we construct are required
to satisfy conditions that are stated in terms of certain vector fields 
$\xi_i$ that are defined along the
facets of $\Delta$, point ``into'' the interior of $\Delta$, and
are tangent to lower-dimensional faces. We therefore begin with the
construction of these vector fields, for which we need some notation. 
Let $\Delta$ be a simple polytope in $\R^n$ with $N$ facets. We may
express the polytope as an intersection of half-spaces, i.e. 
\(\Delta = \bigcap_{i=1}^N H_i \) where 
\begin{equation}\label{eq:Hi}
H_i = \{x \in \R^n \mid  \phi_i(x) \leq \lambda_i \},
\end{equation}
the $\phi_i$ are linear functionals on $\R^n$,
and the $\lambda_i$ are real numbers.
We always assume that $\Delta$ has non-empty interior in $\R^n$.
Let \(\{\sigma_1, \sigma_2, \ldots, \sigma_N\}\) 
be the facets of $\Delta$, 
i.e., $\sigma_i = \Delta \cap \partial H_i$.
For a subset $I \subseteq \{1, 2, \ldots, N\}$, let $F_I$ denote the (possibly empty) face of $\Delta$ 
obtained by intersecting the facets $\sigma_i$ for \(i \in I:\) 
\[
F_I := \bigcap_{i \in I} \sigma_i \subseteq \Delta.
\]
We define $F_{\emptyset} := \Delta$. 
If $F_I$ is nonempty, then since 
$\Delta$ is simple, $F_I$ is a face of codimension $|I|$. Moreover, if
\(J \subset I\) then \(F_I \subset F_J\). 

With this terminology in place we construct the vector fields $\xi_j$ 
along the facets which we use throughout our construction. 
Identifying $T\R^n|_{\sigma_j}$ with $\sigma_j \times \R^n$ in the
standard way, we think of 
these vector fields as functions 
$\xi_j \colon \sigma_j \to \R^n$.

\begin{Lemma}\label{lemma:xi}
Let $\Delta = \bigcap_{i=1}^N \{x \in \R^n \mid \phi_i(x) \leq \lambda_i
\}$ be an $n$-dimensional simple polytope in $\R^n$ and let
\(\{\sigma_i\}_{i=1}^N\) denote the facets of $\Delta$. 
Then there exist smooth vector fields
$\xi_j \colon \sigma_j \to \R^n$ along the facets such that
\begin{enumerate}
\item \(d \phi_j(\xi_j|_x) < 0 \) for all \(x \in \sigma_j\), and 
\item \(d \phi_i(\xi_j|_x) = 0\)  for all $x \in \sigma_j \cap \sigma_i$
      and $i \neq j$.
\end{enumerate}
In particular, let \(F = \sigma_{j_{1}} \cap \cdots \cap
\sigma_{j_{n-\ell}}\) and let 
\[
\pi = (\lambda_{j_1} - \phi_{j_1}, \ldots,
\lambda_{j_{n-\ell}} - \phi_{j_{n-\ell}}): \R^{n} \to \R^{n-\ell}.
\]
Then for all \(x \in F\) the vectors
\(\{\pi_{*}\left(\xi_{j_{1}}\vert_{x} \right), \ldots,
\pi_{*}\left(\xi_{j_{n-\ell}} \vert_{x} \right)\}\) are 
positive multiples of the standard basis elements of $\R^{n-\ell}$.

Moreover, having a priori chosen for each face $F$ a point $x_F$
in the relative interior of $F$, the $\xi_j$ can be chosen to be constant
on a neighbourhood of $x_F$ for each $F$.
\end{Lemma}

Condition (1) in Lemma~\ref{lemma:xi} means that $\xi_j|_x$ is transverse
to $\sigma_j$ and points into $\Delta$.  Condition (2) implies that,
for any face $F_I$ and for the indices $j$ that correspond 
to relative facets $F_I \cap \sigma_j$, 
the restrictions of the vector fields $\xi_j$
to $F_I \cap \sigma_j$
are tangent to $F_I$
and point into $F_I$.

\begin{proof} 
  Let \(x \in \Delta \ssminus \relint(\Delta).\) Let \(j_{1}, j_{2}, \ldots, j_{n-\ell} \in
  \{1, 2, \ldots, N\}\) be the indices of all the facets that pass
  through $x$. The linear functionals \(\phi_{j_{1}}, \ldots,
  \phi_{j_{n-\ell}}\) are linearly independent in $(\R^n)^*$ since the
  polytope is simple. Let \(t_{1}, \ldots, t_{\ell}\) be linear
  functionals such that \(t_{1}, \ldots, t_{\ell}, \phi_{j_{1}},
  \ldots, \phi_{j_{n-\ell}}\) is a basis of $(\R^n)^*$. Let \(t_{\ell+1} :=
  \lambda_{j_{1}} - \phi_{j_{1}}, \ldots, t_{n} := \lambda_{j_{n-\ell}}
  - \phi_{j_{n-\ell}}.\) Then there exists a neighbourhood $U_x$ of $x$
  in $\Delta$ such that \(t_{1}, \ldots, t_{n} \colon \R^{n} \to
  \R^{n}\) carries $U_{x}$ onto a
  neighbourhood of $0$ in $\R^{\ell} \times \R^{n-\ell}_{\geq 0}$.
  Note in particular that for every point $y$ of $U_x$, the indices 
\(j_{1}, j_{2}, \ldots, j_{n-\ell} \in
  \{1, 2, \ldots, N\}\) are precisely the indices of all the facets that pass
  through $y$.

The vectors \(\xi_{i}^{(x)} := \frac{\partial}{\partial t_{\ell+i}}\)
satisfy the required properties at all the points of $U_{x}$.

Now fix \(j \in \{1, 2, \ldots, N\}\) and consider $\sigma_j$. The
open sets $\{U_x \cap \sigma_j \}_{x \in \sigma_j}$ form an open cover of
$\sigma_j$. Since $\Delta$ is compact, so is $\sigma_j$, so we may
choose a finite subcover $\{U_{x_s} \cap \sigma_j\}_{s=1}^{N_j}$ for some 
$N_j \in \N$. Without loss of generality we may assume that for each
$F$ contained in $\sigma_j$, the point $x_F$ appears among the
$\{x_s\}_{s=1}^{N_j}$. We may also assume that for each $x_F$, there is a
sufficiently small neighborhood $V_F \subseteq U_{x_F}$ such that
$U_{x_s} \cap V_F = \emptyset$ for all $x_s \neq x_F$. 
Let \(\rho_{s} \colon \sigma_j \to \R, 1 \leq s \leq N_j,\) be 
a smooth partition of unity 
with $\supp \rho_{s} \subseteq U_{x_s} \cap \sigma_{j}$. 
Define 
\[
\xi_{j} := \sum_{s} \rho_{s} \xi_{j}^{(x_{s})}.
\]
Then the vector fields \(\xi_{1}, \ldots, \xi_{N}\) satisfy the
required properties. 
\end{proof}

Using the above vector fields $\xi_j$, we may now state
the recursive conditions on the functions $f_{\ell}$
in our construction:

\begin{enumerate}
\item[(f1)]
$f_{\ell}$ is a smooth function defined on an open neighbourhood $U_{\ell}$ 
in $\Delta$ of the $\ell$-skeleton;
\item[(f2)]
for each face $F$ of dimension \( \leq \ell\),
the restriction of $f_{\ell}$ to $\relint(F)$ 
attains its minimum at the point $x_F$, 
has no critical points other than $x_F$,
and $f_{\ell}(x_F) = \alpha_F$; and 
\item[(f3)] \label{p3} for each face $F$ of dimension $\leq \ell$, we have
  \(df_{\ell}(\xi_j\vert_x) < 0\) for all $j$ with $\sigma_j \cap F
  \neq \emptyset$ and for all \(x \in \sigma_j \cap F \).
\end{enumerate}
(Note that $\sigma_j \cap F$ can be either $F$ itself
or a relative facet of $F$.)

We now begin the recursive construction of the functions $f_\ell$. 
The base case requires us to construct a function 
$f_0$ near the $0$-skeleton satisfying (f1)-(f3) above.  
Let $x$ be a vertex of $\Delta$. 
Since $\Delta$ is
simple, there exists an open neighbourhood $U_x \subset \R^n$ of $x$
and an element $A$ of $\AGL(n,\R)$ (=affine automorphisms of $\R^n$)
such that the map $A$ takes $x$ to $0$ and takes the intersection 
$U_x \cap \Delta$ to a neighbourhood of $0$ in the positive orthant 
\(\R^n_{\geq 0} = \{v=(v_1, \ldots, v_n) \in \R^n: 
   v_i \geq 0 \textup{ for all } 1 \leq i \leq n\}.\) 
We may take $f_0|_{U_x}$ to be the composition of the affine
transformation $A$ with the function \(v \mapsto \alpha_F - \sum_i v_i.\)
The last condition in Lemma~\ref{lemma:xi} implies that for each 
facet $\sigma_j$ that contains the vertex $x$ of $\Delta$ 
the vector field $\xi_i$, near $x$, is a positive multiple of
$\frac{\partial}{\partial v_i}$ in the
above coordinates for $\R^n_{\geq 0}$. Hence the function $f_0$ near
$x$ satisfies the condition (f3) above. By our formula for $f_0|_{U_x}$,
the condition (f2) also holds at this vertex. 
Repeating this for all vertices (and possibly shrinking
the open neighbourhoods so that their closures are disjoint), 
we obtain a function that 
satisfies the above conditions (f1)-(f3) with $\ell = 0$. 
This completes the base case of the induction.

We now proceed with the recursive step. Let $\ell \geq 1$, and 
assume that we have already defined a function $f_{\ell-1}$ on a
neighbourhood of the $(\ell-1)$-skeleton that satisfies 
the conditions (f1)-(f3). 
We now construct a function $f_{\ell}$ which (after possibly shrinking
the 
neighbourhood on which $f_{\ell-1}$ is defined) extends $f_{\ell-1}$ near each $\ell$-dimensional face $F$
separately. In fact, we will first construct $f_\ell$ in the
$\ell$-dimensional affine
span of $F$ and then extend to a neighbourhood of $F$ in $\R^n$. 
Fix an $\ell$-dimensional face $F$ of $\Delta$. By using an affine
change of coordinates, we may assume 
without loss of generality that 
\begin{itemize} 
\item the affine span of $F$ is $\R^{\ell}$, identified with the subspace 
$\R^{\ell} \times \{ 0 \}^{n-\ell}$ of $\R^n$;
\item the chosen point $x_F \in \relint F$ is the origin $0$; and 
\item in a neighbourhood of $0$ the polytope $\Delta$ coincides 
with $\R^{\ell} \times \R^{n-\ell}_{\geq 0}$. 
\end{itemize} 
Furthermore, after permuting the indices if necessary,
we may without loss of generality assume that 
$F = \sigma_1 \cap \ldots \cap \sigma_{n-\ell}$,
so that the vector fields \( \xi_1, \ldots, \xi_{n-\ell} \) 
are defined along $F$. 

The set of indices
\[
{\mathcal K} := \{k \, \vert \, \sigma_k \cap F \neq \emptyset, \,
n-\ell < k \leq N\}
\]
parametrizes the set of relative facets of $F$. 
We have on each relative facet $\sigma_k \cap F$
a vector field $\xi_k$ along which $f_{\ell-1}$ decreases.
In Lemma~\ref{lemma:def-etaF} below we construct a single vector field
$\eta^F$ on a neighbourhood $V_F \ssminus \{0\}$
of $F \ssminus \{0\}$ in $\R^\ell$ which interpolates between $-\ddr$ and
the vector fields $\xi_j$ constructed in Lemma~\ref{lemma:xi}, and
along which  $f_{\ell-1}$ decreases where it is defined. 
The flow along $\eta^F$ 
allows us to reparametrize the relative interior of $F$, which will
in turn allow us to explicitly construct the extension $f_{\ell}$ of
$f_{\ell-1}$ near $F$.

Let $r$ denote the radial coordinate in $\R^{\ell}$ 
and let \(-\ddr\) denote the
corresponding inward-pointing radial vector field, defined and smooth
on $\R^{\ell} \ssminus \{0\}$. 
Let \(k \in {\mathcal K}\). 
Denote by $L_k$ the restriction to the affine subspace $\R^{\ell}$
of the linear functional $\frac{1}{\lambda_k}\phi_k$ 
in the notation of~\eqref{eq:Hi}. 
Then the face $F$ as a subset of $\R^\ell$ is given by the intersection
\[
F = \bigcap\limits_{k \in {\mathcal K}} \left\{ x \in \R^{\ell} \, \mid \,
L_k(x) \leq 1 \right\} \subseteq \R^{\ell}.
\]
Given \(\delta > 0,\) we may also define a smaller polytope
$F_{\delta} \subseteq F$ by 
\[
F_{\delta} := \bigcap\limits_{k \in {\mathcal K}} \left\{ x \in \R^{\ell} \,
\mid \, L_k(x) \leq 1 - \delta \right\} \subseteq \R^{\ell}.
\]
We always assume that $\delta$ is sufficiently small such that
$F_{\delta}$ has the same combinatorial type as $F$ 
(this is possible because $\Delta$ is simple)
and such that $ 0 \in \relint F_{\delta}. $

\begin{Lemma}\label{lemma:def-etaF}
There exists an open subset $V_F$ of $\R^\ell$ containing $F$,
a smooth vector field $\eta^F$ on 
\(V_F \ssminus \{0\}\),
and a constant $\delta$, $0 < \delta < 1$, such that
\begin{enumerate} 
\item for all \(k \in {\mathcal K}, \) 
\begin{equation}\label{eq:etaF-delta-ineq}
dL_k(\eta^F)  < - \delta < 0 \, \textup{ \ on \ }  \, 
         L_k^{-1}([1-\delta,1]) \cap V_F ;
\end{equation}
\item \(\eta^F = - \ddr\) on \(F_\delta \ssminus \{ 0 \} \); and 
\item $df_{\ell-1}(\eta^F) < 0$ on $\del F$.
\end{enumerate}
\end{Lemma}

\begin{proof} 
The proof is by explicit construction. We first construct vector fields
locally which satisfy the conditions of the lemma, and then 
patch them together using a partition of unity. 

We begin with the interior of $F$. 
Consider the open set
\(U_0 := \relint(F) \ssminus \{0\}. \) The radial vector field
\(\eta_0 := - \ddr\) is certainly smooth on $U_0$. Moreover, since \(0
\in \relint(F)\) and $F$ is convex, $-\ddr$ is transverse to all
relative facets $F_k = \sigma_k \cap F$ of $F$. In particular, 
for each \(k \in {\mathcal K}\) we have \(dL_k(\eta_0) < 0\) on $F_k$. 
Since each $F_k$
is compact, there exists a neighbourhood $V_k$ of $F_k$ and \(\delta_k
> 0\) such that \(dL_k(\eta_0) < -\delta_k < 0\) on $V_k$. Let
\(\delta'_k  > 0\) be such that 
\(L_k^{-1}([1-\delta'_k, 1]) \cap U_{0}\subseteq V_k\) for each $k$, 
and set \(\delta_0 := \min \left( \cup_{k \in {\mathcal K}} 
     \{\delta_k, \delta'_k\} \right) \). 
Then, for all \(k \in {\mathcal K}\), we have
\begin{equation}\label{eq:eta0-Lk}
dL_k(\eta_0) < -\delta_0 \, \textup{ \ on \ } \, 
    L_k^{-1}([1-\delta_0,1]) \cap U_0.
\end{equation}
We conclude that the vector field $\eta_0$ on $U_0$ satisfies the conditions (1)
and (2) of the lemma. (Since $U_0$ does not intersect $\del F$, the
condition (3) is not relevant for this case.)

We now proceed to the local construction at the boundary. Let $x$ be
in the relative boundary $\del F$ of $F$ in $\R^{\ell}$. Let \({\mathcal K}_x
:= \{ k \in {\mathcal K} \, \mid \, x \in F_k\}.\) 
By Lemma~\ref{lemma:xi}(2), the vectors $\xi_j
\vert_x$ for $j \in {\mathcal K}_x$ lie in $\R^\ell$. Let 
\[
\eta_x := \sum_{j \in {\mathcal K}_x} \xi_j \vert_x.
\]
Observe that by the property (f3) of $f_{\ell-1}$ and the definition of the
$\eta_x$, 
\[
df_{\ell-1}(\eta_x) < 0.
\]
Moreover, by Lemma~\ref{lemma:xi}(1),
the derivative $dL_k(\eta_x)$ is negative for every $k \in \calK_x$,
since $dL_k(\xi_k) <0$ and $dL_k(\xi_j) = 0$ 
for all $j \in \calK_x \ssminus \{ k \}$.
Choose $\delta'_x$ such that \(0 < \delta'_x < \min\{ - dL_k(\eta_x)
\}_{k \in {\mathcal K}_x}.\) We now deal with the indices $s$ not in
${\mathcal K}_x$. By the definition of ${\mathcal K}_x$,
\(L_s(x) < 1\) for all \(s \in {\mathcal K} \ssminus {\mathcal
  K}_x.\) Choose $\delta''_x$ such that \( 0 < \delta''_x < \min \{1 -
L_s(x)\}_{s \in {\mathcal K} \ssminus {\mathcal K}_x}.\) Then \(x
\not \in L_s^{-1}([ 1-\delta''_x, 1])\) for all \(s \in {\mathcal K}
\ssminus {\mathcal K}_s.\) 
Let \(U'_x\) be a neighbourhood of $x$ in $\R^{\ell}$ such that 
$0 \not \in U'_x$, \ 
$U'_x \cap L_s^{-1}([1-\delta''_x, 1]) = \emptyset$ 
for all \(s \in {\mathcal K} \ssminus {\mathcal K}_x\), 
and \(d(f_{\ell-1})_y(\eta_x)(y) < 0\) at all points $y$ of $U'_x$. 
(Here $\eta_x$ is viewed as a vector field on $U_x$ with constant
coefficients.) 
Since the relative
boundary $\partial F$ of $F$ is compact, there exists a finite set $\{x_1, \ldots,
x_m\} \subseteq \partial F$ such that \(\partial F \subseteq U'_{x_1} \cup
U'_{x_2} \cup \cdots \cup U'_{x_m}\). 
Let 
\(\delta = \min\{ \delta'_{x_1}, \delta''_{x_1}, 
   \ldots, \delta'_{x_m}, \delta''_{x_m}\} \).
Define $V_F := U_0 \cup U'_{x_1} \cup \ldots \cup U'_{x_m}$.
Then, by construction, $V_F$ contains $F$, and the sets
$U_0$, $U'_{x_1} \ssminus (U'_{x_1} \cap F_\delta)$, $\ldots$, $U'_{x_m}
\ssminus (U'_{x_m} \cap F_\delta)$ 
form an open covering of $V_F$.  Let $\rho_0, \rho_1,\ldots, \rho_m$
be a partition of unity subordinate to this covering, and let
\begin{equation}\label{eq:def-rho}
\eta^F := \rho_0 \left(-\ddr\right) + \sum_{i=1}^m \rho_i \eta_{x_i}. 
\end{equation}
Then $\eta^F$, $V_F$, and \(\delta > 0\), as chosen above, 
satisfy the conditions of the lemma.
\end{proof}

We now wish to show that we may use the vector field $\eta^F$
constructed in Lemma~\ref{lemma:def-etaF} to reparametrize the
relative interior of $F$. Denote by $\psi(t,x)$ the flow along $-
\eta^F$, where $t$ is the time parameter of the flow and $x$ is the
initial condition.  For fixed initial condition $x$, by a
\emph{maximal trajectory} through $x$ we mean the flow $\psi(t,x)$ on
the maximal interval $(a,b) \subseteq \R$ on which the flow is
defined. 
As a first step, we wish to show that for any $x$
in $\del F_\delta$, this flow $\psi(t,x)$ takes $x$ 
to $\del F$ in finite time:

\begin{Lemma}\label{lemma:flow-into-Fdelta}
Let $\eta^F$ and $\delta$ be as in Lemma~\ref{lemma:def-etaF} and 
let $x$ be a point in $\del F_\delta$.
Let $X(t) := \psi(t,x)$ be the maximal trajectory 
of $-\eta^F$
with initial condition $\psi(0,x) = x$.
Then there exists $t_x>0$ such that 
$X(t) \in \relint F$ for all $0 \leq t < t_x$
and $X(t_x) \in \del F$. 
\end{Lemma}

\begin{proof}
Since $x \in \del F_\delta$, there exists
$k \in \calK$ such that $L_k(x) = 1 - \delta$. 
Fix one such $k$. Let $t > 0$. First we claim that 
\begin{quotation}
if $X(\tau)$ is defined 
and takes values in $F$ for all $\tau$ in $[0,t]$ then
\begin{equation}  \label{claim}
 L_k(X(t)) > 1 - \delta + \delta t .
\end{equation}
\end{quotation}
To prove the claim, recall that by the construction of $\eta^F$ and $\delta$,  if $X(t)$ is defined
and belongs to $F$, then
\begin{equation} \label{bound}
 \quad \textup{ if } \quad
  L_k(X(t)) \geq 1-\delta
 \quad \textup{ then } \quad
 \frac{d}{dt} L_k(X(t))  = (dL_k)_{X(t)}(-\eta^F) > \delta .
\end{equation}
Since $X(0)=x$ and $L_k(x)=1-\delta$, then by~\eqref{bound} we may
conclude that the derivative $\frac{d}{dt} L_k(X(t))$ is greater than
$\delta$ when $t=0$.  By continuity, this derivative is greater than
$\delta$ for $t$ in a neighbourhood of $0$.  Integrating, we conclude
that~\eqref{claim} is true if $t$ is positive and sufficiently
small. Observe also that by continuity of solutions of ODEs,
  for any $t>0$ for which~\eqref{claim} holds, there exists an
  open interval containing $t$ on which the trajectory is defined and
  for which~\eqref{claim} still holds; in particular, the set of $t$
  for which the claim holds is open. 
Now suppose that there exists a positive $t_0$ for which the claim
does not hold; 
then there exists a minimal such $t_0$ since the set of $t_0$ for which the
claim does not hold is closed. 
By minimality, if $0 < t < t_0$
then $L_k(X(t)) > 1 - \delta + \delta t$.  By~\eqref{bound},
$L_k(X(t)) > 1 - \delta + \delta t$ and $t>0$ imply $\frac{d}{dt}
L_k(X(t)) > \delta$.  Integrating,
\begin{equation*}
\begin{split}
 L_k(X(t_0)) & = L_k(X(0))  + \int_{0}^{t_0} \frac{d}{dt} L_k(X(t)) \, dt 
 \ > \ L_k(X(0)) + \int_0^{t_0} \delta \, dt \\
 & = 1 - \delta + \delta t_0, \\
\end{split}
\end{equation*}
which contradicts the assumption on $t_0$.
This completes the proof that~\eqref{claim} holds for all positive $t$.

Now suppose that the assertion of the lemma is not true,
that is, $X(t) \in \relint F$ for all $t > 0$ where $X(t)$ is defined.
Then~\eqref{claim} implies that $X(t) \not\in F_\delta$ for all $t>0$ 
where $X(t)$ is defined.
This implies that $X(t)$ is defined for all $t \in 
[0,\infty)$,
because it coincides with the trajectory of a compactly
supported vector field that equals $\eta_F$
on $F \ssminus F_\delta$. 
Taking $t$ big enough so that $1 - \delta + \delta t$ is greater than one, 
\eqref{claim} contradicts the assumption that $X(t)$ is in $F$. This
proves the lemma.
\end{proof}

We may now use \(\psi(t,x)\) to reparametrize $F$. In order to do so
smoothly, we first need to construct a smooth manifold in $F$ which
approximates $\del F$ in $F$.  We will construct this smooth
approximation by taking a regular level set of a smooth function,
which we now construct:

\begin{Lemma} \label{hF}
There exists a smooth function $h^F$ on a neighbourhood of $\partial F$
in $F$ such that
\begin{itemize}
\item \(dh^F(\eta^F) < 0\) at all points of \( \relint(F) \) 
      where $h^F$ is defined, and 
\item \(h^F \vert_{\partial F} \equiv 1.\) 
\end{itemize}
\end{Lemma}

\begin{proof} 
  We first construct a function satisfying the conditions of the lemma
  locally near any point \(x \in \partial F.\) 
  We then use a partition of unity constant along $\eta^F$ to form the global function
  $h^F$ required in the lemma. 

Let \(x \in \partial F\) and let 
$U_x$ be a neighbourhood of $x$ in $\R^\ell$ such that
$U_x$ only intersects facets of
$F$ that contain $x$ and such that $U_x$ is contained in the set $V_F$
where $\eta^F$ is defined.  By construction,
$(dL_k)_x(\eta^F) < 0$ for all $k \in \calK_x$, so
after possibly shrinking $U_x$, we may assume that $dL_k(\eta^F) < 0$
at all points in $U_x$.
On this neighbourhood $U_x$, we define the function 
\begin{equation}\label{eq:hx}
h^x := 1 - \prod_{k \in \calK_x} (1 - L_k).
\end{equation}
At points $y \in U_x \cap \del F$, at least one of the $L_k$ in the
right hand side of~\eqref{eq:hx} is equal to one, so $h^x(y) =1$.  At
points $y \in U_x \cap \relint F$, all the $L_k$ in~\eqref{eq:hx} are
less than $1$, so $h^x(y) < 1$.  Since $dL_k(\eta^F) < 0$ for all $k \in
\calK_x$, each $L_k$ is decreasing along the trajectories of
$\eta^F$. From~\eqref{eq:hx} we then see that $h^x$ is also decreasing
along $\eta^F$ at any point $y$ where $L_k(y) < 1$ for all $k \in
\calK$. In particular, $dh^x(\eta^F) < 0$ on $U_x \cap \relint F$.

Secondly, we patch together these functions $h^x$ by means of
a partition of unity. This will require some extra care since we wish
to guarantee that the resulting function still satisfies the first
condition of the lemma. To accomplish this, 
we now construct a partition of unity $\{ \rho^x \}$ 
such that for each $\rho^x$ we have 
$d\rho^x(\eta^F) = 0$, i.e. the functions $\rho^x$ are constant along
the flow of $\eta^F$. 
Let \(x \in \partial F.\) Since \(\eta^F |_x \neq 0\) and $\eta^F$ is
smooth, there exists a
neighbourhood \(V_x \subset U_x\) of $x$ in $\R^{\ell}$ with a smooth
parametrization \(\phi \colon \Omega_x \to V_x\), where 
\[
\Omega_x := \{ (t_1, \ldots, t_{\ell}) \in \R^{\ell} \hs \mid \hs
   |t_1|<4\varepsilon_1 \text{ and } 
   t^2_2 + \cdots + t_{\ell}^2 < 4 \varepsilon_2 \}
\]
for some $\varepsilon_1, \varepsilon_2 > 0$, 
such that \(\phi(0) = x\) and \(\phi_* \left(\frac{\partial}{\partial
t_1}\right) = \eta^F.\) 
Since $\eta^F$ is transverse to $\partial F$, 
after possibly shrinking $\varepsilon_1$ and $\varepsilon_2$, we may
assume without loss of generality that 
for every value of \(t_2, \ldots, t_{\ell}\) 
that occurs in $\Omega_x$, 
the function $ t_1 \mapsto L_k(\phi((t_1, t_2, \ldots, t_{\ell})) $
is defined for $t_1$ in some interval,
is strictly monotone, and assumes both positive and negative values. 
It then follows that for every $(t_2, \ldots, t_\ell)$ occurring in
$\Omega_x$, 
there exists exactly one $t_1$ such
that \(\phi(t_1, t_2, \ldots, t_{\ell}) \in \partial F.\) 
Moreover, after possibly further shrinking $\varepsilon_2$, we may also assume that 
\[
\{ (t_1, t_2, \ldots, t_{\ell}) \mid \phi(t_1, \ldots, t_{\ell}) \in
\partial F \}
\]
is contained in \(\{|t_1|< 2 \varepsilon_1 \}.\) 

Now observe that the function
\[
\rho^x(q) := 
\begin{cases}
e^{-\frac{1}{\varepsilon_2^2 - t_2^2 - \cdots - t_{\ell}^2}}  &
\begin{minipage}{3.3in}
if $q = \phi(t_1, t_2, \ldots, t_\ell)$ for some
$(t_1, t_2, \ldots, t_\ell) \in \Omega_x$,
$t_2^2 + \cdots + t_{\ell}^2 < \varepsilon_2$, and 
$|t_1| < \frac{5}{2} \varepsilon_1$
\end{minipage} \\
0 & \text{otherwise} \\
\end{cases}
\]
is smooth on the neighbourhood
\[
\widetilde{V}_x := \R^{\ell} \ssminus \{\phi(t_1, t_2, \ldots, t_{\ell})
\, \mid \, 2\varepsilon_1 \leq |t_1| \leq 3 \varepsilon_1, \quad t_2^2
+ \cdots + t_{\ell}^2 \leq 2 \varepsilon_2 \}
\]
of $\del F$ in $\R^\ell$.
To see this, observe that
$\widetilde{V}_x$ is the union of the open set \(\{\phi(t_1, \ldots,
t_{\ell}) \mid |t_1|<2\varepsilon_1 \}\) and the open set \(\R^{\ell} \ssminus \{\phi(t_1,
\ldots, t_{\ell}) \mid |t_1|\leq 3 \varepsilon_1, t_2^2 + \ldots +
t_{\ell}^2 \leq 2 \varepsilon_2 \}.\) On the second set, \(\rho^x \equiv
0\) by definition. On the first set, $\rho^x$ is smooth. Hence $\rho^x$ 
is smooth also on the union.

By slight abuse of notation, we let $\{\rho^x > 0\}$ denote the open
subset in $\widetilde{V}_x$ where $\rho^x$ is positive. By definition of
$\rho^x$, the point \(x \in \del F\) is contained in $\{\rho^x >
0\}$. 
Since $\partial F$ is compact, there exists a finite set \(\{x_1,
\ldots, x_M\} \subseteq \partial F\) such that \(\cup_{j=1}^M
\{\rho^{x_j}>0\}\) contains $\partial F$. 
In particular $\sum_{j=1}^M \rho^{x_j}$ is defined and positive on a neighbourhood
of $\del F$ in $\R^\ell$. Let $\widetilde{V}$ be such a neighbourhood.
We now define 
\begin{equation}\label{eq:def-rhok}
\rho^k := \frac{1}{\sum_{j=1}^M \rho^{x_j}} \rho^{x_k}
\end{equation}
on $\widetilde{V} \cap \widetilde{V}_{x_k}$. Since 
\(d\rho^{x_k}(\eta^F) \equiv 0 \) by construction, it follows that 
$d\rho^k(\eta^F) \equiv 0$ also. 

Since $h^{x_k}$ is defined on \(U_{x_k}\) which entirely contains the
support of $\rho^{x_k}$, the product $\rho^k h^{x_k}$ may be extended
to a smooth function on all of $\widetilde{V}$. Hence the sum 
\begin{equation}\label{eq:def-hF}
h^F := \sum_{i=1}^M \rho^i h^{x_i}
\end{equation}
is well-defined and smooth on $\widetilde{V}$. 

We now claim that \(h^F \equiv 1\) on \(\partial F\) and
\(dh^F(\eta^F) < 0\) on $\widetilde{V} \cap \relint F$.
To see this, suppose \(x \in \partial F\). 
Let \({\mathcal I}_x = \{k \mid x \in U_{x_k}\}.\) Consider 
\begin{equation}\label{eq:Uxi-tildeV}
\left(\bigcap\limits_{k \in {\mathcal I}_x} U_{x_k}\right) \cap
\widetilde{V}. 
\end{equation}
On this set, 
\[
h^F = \sum_{k \in {\mathcal I}_x} \rho^k h^{x_k},
\]
the $h^{x_k}$ for \(k \in {\mathcal I}_x\) are well-defined, 
satisfy \(dh^{x_k}(\eta^F) < 0\) on $\relint(F)$, and moreover, 
\(\sum_{k \in {\mathcal I}_x} \rho^k \equiv 1.\)
Hence for every $y$ in $\partial F \cap 
   \left(\cap_{k \in {\mathcal I}_x} U_{x_k}\right) \cap \widetilde{V}$ 
we have 
\[
h^F(y) = \sum_{k\in {\mathcal I}_x} \rho^k(y) h^{x_k}(y) 
       = \sum_{k \in {\mathcal I}_x} \rho^k(y) = 1.
\]
Moreover, on $\relint(F) \cap 
      \left(\cap_{k \in {\mathcal I}_x} U_{x_k}\right) \cap
\widetilde{V}$, we have 
\[
dh^F(\eta^F) = \sum_{k \in {\mathcal I}_x} d\left(\rho^k h^{x_k}\right)(\eta^F) = 
\sum_{k \in {\mathcal I}_x} \rho^k dh^{x_k}(\eta^F) < 0,
\]
where the second equality uses that \(d\rho^k(\eta^F)
= 0\) by our construction of the $\rho^k$, 
and the last inequality uses that $dh^{x_k}(\eta^F)  < 0$, 
that $\rho^k \geq 0$, and that 
not all the $\rho^k$ vanish. Hence $h^F$ satisfies the
required conditions of the lemma.
\end{proof}

By the construction of $h^F$ in Lemma~\ref{hF}, the differential $dh^F$
never vanishes in $\relint F$.
Thus, for any $\veps > 0$, the level set
\begin{equation}\label{eq:def Zepsilon}
Z_{\varepsilon} := 
    \left(h^F\right)^{-1}(1 - \varepsilon) \ \subset \ F
\end{equation}
is a smooth manifold. 
This is our smooth approximation to $\partial F$. The next lemma
proves smoothness properties of the flow $\psi(t,x)$ along $-\eta^F$
with respect to this level set. 
Using (similar) notation as in Lemma~\ref{lemma:flow-into-Fdelta}, for 
\(x \in \relint F_\delta \ssminus \{0\}\) we let $X(t)$ denote the
trajectory for $-\eta_F$ that is defined for $0 \leq t \leq t_x$ and
such that $X(0) = x$, $X(t_x) \in \del F$, and $X(t) \in \relint F
\ssminus \{0\}$ for all $0 < t < t_x$.  Such a trajectory and $t_x$
exist because $\eta^F$ coincides with $-\ddr$ within $F_\delta$ and by
Lemma~\ref{lemma:flow-into-Fdelta}.

\begin{Lemma} \label{finite time}
Let $\delta>0$ be as in Lemma~\ref{lemma:def-etaF}. 
Then there exists $\varepsilon>0$ sufficiently small such that for all \(x \in
\relint F_\delta\) there exists a unique $t'_x > 0$ with $h^F(X(t'_x))
= 1 - \varepsilon$ and such that the function \(x \mapsto t'_x\) is
smooth on $\relint F_\delta \ssminus \{0\}$. 
\end{Lemma}

\begin{proof}
Let $V$ be an open neighbourhood of $\del F$ in $F$ whose closure 
is contained in the open neighbourhood of $\del F$ where $h^F$ is defined.
Its boundary, $\ol{V} \ssminus V$, is a compact subset 
of $\relint F$ on which $h^F < 1$. 
Choose $\veps$ such that 
$1-\veps > \max h^F|_{\ol{V} \ssminus V}$.
The existence of $t_x'$ (specified by this choice of $\veps$) follows from the continuity of the function
$t \mapsto h^F(X(t))$. The uniqueness of $t_x'$ follows from the fact
that 
$dh^F(-\eta^F)>0$ at all points of $\relint F$ where $h^F$ is defined,
which implies that
$t \mapsto h^F(X(t))$ is monotone increasing.

The derivative $\deldelt h^F \circ \psi |_{(t,x)}$ is non-zero
because it is equal to $dh^F_{\psi(t,x)}(-\eta^F)$, 
which is positive by the construction of $h^F$. 
The smoothness of the function $x \mapsto t_x'$ 
now follows by applying the implicit function theorem to 
the condition $h^F(\psi(t,x)) = 1 - \veps$.
\end{proof}

We now wish to use the flow along 
$-\eta^F$ to reparametrize the compact region in
$F$ with boundary $Z_{\varepsilon}$ so that the region is diffeomorphically identified
with the standard closed ball in $\R^{\ell}$ of some radius. Moreover,
we want to arrange that under this identification, 
the vector field ${\eta}^F$ is identified with the
standard inward-pointing radial vector field $-\ddr$. Such a
parametrization allows us to construct the function
$f_\ell$ using explicit coordinates on the standard closed ball.

We now review some properties of $f_{\ell-1}$ useful in the constructions to
follow. 
\begin{itemize}
\item
First, $ \min\limits_{\del F} f_{\ell-1} > \alpha_F $.
Indeed, the induction hypotheses on $f_{\ell-1}$ imply that
$\min\limits_{\del F} f_{\ell - 1}
 = \min\limits_{E \subsetneq F} \{ \alpha_E \} $,
which is greater than $\alpha_F$ by assumption. 
\item  
Second, $df_{\ell-1}(\eta^F) < 0 \text{ on } \del F $
by the construction of $\eta_F$ in Lemma~\ref{lemma:def-etaF}.
\end{itemize}
After possibly shrinking the neighbourhood $U_{\ell - 1}$
of $\del F$ on which $f_{\ell - 1}$ is defined,
we may also assume without loss of generality that
\begin{itemize}
\item
$ \inf\limits_{U_{\ell - 1} \cap \R^{\ell}} f_{\ell - 1} > \alpha_F $;
\item
$\eta_F$ is well defined everywhere on $U_{\ell-1} \cap \R^\ell$;
and 
\item $df_{\ell - 1}(\eta^F) < 0$ on $U_{\ell-1} \cap \R^{\ell}$. 
\end{itemize}

Let $\varepsilon>0$ be sufficiently small
such that the conclusions of Lemma~\ref{finite time} apply,
such that the level set $Z_\veps$ of~\eqref{eq:def Zepsilon} 
is contained in $U_{\ell - 1}$,
and such that $\alpha_F < \min\limits_{Z_\veps} f_{\ell-1}$.
In particular, by the properties of $f_{\ell - 1}$
listed above,
$$ df_{\ell - 1}(\eta^F) < 0 \text{ near } Z_\veps. $$

To achieve the reparametrization mentioned above, we will need 
to further rescale $\eta^F$.
Specifically, choose \(r_2 > r_1 > 0\) such that the set $F_\delta$ 
of Lemma~\ref{lemma:def-etaF}
contains the closed ball of radius $r_2$ about the origin,
and consider the spheres $S^{\ell}_{r_1}$ and $S^{\ell}_{r_2}$ with
center $0$  and radii $r_1$ and $r_2$ respectively.
In the lemma below and in the arguments that follow, 
we rescale $\eta^F$ within the region of $F$ contained
between the concentric spheres $S^{\ell-1}_{r_1}$ and
$S^{\ell-1}_{r_2}$. This makes the
time required to flow out to $Z_\veps$ uniform along $S^{\ell-1}_{r_1}$.

\begin{Lemma} \label{normalize}
There exists a smooth function \(\sigma \colon \R^{\ell} \to \R \)
that takes positive values, is equal to $1$ outside the set
\( \{x \in \R^{\ell} \, \vert \, r_1 < |x| < r_2 \} \), 
and such that for the flow $\psi(t,x)^{new}$ corresponding to the
rescaled vector field 
  \(- \eta^F_{\new} := - \sigma \eta^F, \) there exists a constant
  $R>0$ such that 
$\psi(R,x)^{\new}$ is defined and belongs to $Z_\veps$
for all $x \in S_{r_1}^{\ell-1}$.
\end{Lemma}

\begin{proof}
By construction, $\eta^F$ agrees with $-\ddr$ in the region between
$S^{\ell-1}_{r_1}$ and $S^{\ell-1}_{r_2}$. 
Given a smooth function $h_T(r)$ of one variable and 
vector field $h_T(r) \ddr$ on $\R^\ell \ssminus \{0\}$, it takes time 
\begin{equation}\label{eq:travel}
\int_{r_1}^{r_2} \frac{dr}{h_T(r)}
\end{equation} to flow along this vector field
from a point in $S^{\ell-1}_{r_1}$ to 
$S^{\ell-1}_{r_2}$. Let $ \beta \colon [r_1,r_2] \to \R_{\geq 0} $
be a smooth function such that $\beta(r) \equiv 0$ 
for $r$ near both $r_1$ and $r_2$ and such that 
$ \int_{r_1}^{r_2} \beta(r) dr = 1$. 
For $T>0$, define $h_T: [r_1, r_2] \to \R$ by 
$$ h_T(r) := \frac{1}{1 + T \beta(r)}. $$
Then by construction, the travel time~\eqref{eq:travel} is equal to
$r_2 - r_1 + T$, $h_T(r) \equiv 1$ 
for $r$ near both $r_1$ and $r_2$,
and $(T,r) \mapsto h_T(r)$ is smooth. 

Now, let $T \colon S^{\ell-1}_{r_1} \to \R_{>0}$ be a smooth function.
For $x \in \R^\ell$, define 
\[
\sigma(x) = 
\begin{cases}
 h_{T(r_1 \frac{x}{\|x\|})} (\|x\|) & \mbox{if} \hs r_1 \leq \|x\| \leq r_2 \\
 1 & 0 \leq \|x\| < r_1 \text{ or } \|x\|> r_2 .
\end{cases}
\]
This defines a smooth function $\sigma \colon \R^\ell \to \R$.
Define a new vector field $\eta^F_{new}$ by 
$$ \eta^F_\new = \sigma \eta^F.$$
This vector field has the same trajectories as $\eta^F$, but with
different time parametrizations.
By the construction of the
function $h_T$, the travel time along the trajectory from a point $x$ in
$S^{\ell-1}_{r_1}$ to the sphere $S^{\ell-1}_{r_2}$
for the rescaled vector field
$-\eta^F_\new$ is equal to 
\[
T(x) + r_2 - r_1   =
T(x) + \left( \textup{ the travel time for the vector field }
-\eta^F\right).
\]
It follows that the travel time along the
trajectory of $-\eta^F_\new$ from a point $x \in S^{\ell-1}_{r_1}$ to $Z_\veps$
is equal to $t'_x + T(x)$ where $t'_x$ is as in Lemma~\ref{finite time}.
To finish the proof, it therefore remains to choose the function 
$T \colon S^{\ell-1}_{r_1} \to \R_{>0}$
so that $R := t_x' + T(x) $ is independent of $x \in S^{\ell-1}_{r_1}$.
Pick any $R > \max\limits_{x \in S^{\ell - 1}_{r_1}} t_x'$.
For $x \in S^{\ell - 1}_{r_1}$, define $ T(x) := R - t_x'$.
Then this function $T$ is smooth, it takes positive values, and 
$T(x) + t_x'$ is evidently independent of 
$x \in S^{\ell - 1}_{r_1}$,
as required.
\end{proof}

We can use the modified $\eta^F_{new}$ to create a diffeomorphism
between a standard closed ball in $\R^\ell$ and a subset of $F$ as
follows. Let $\psi^{new}(t,x)$ denote the flow along $-\eta^F_{new}$. 
For a real number $r>0$, let $B^\ell(r,0)$ denote the standard closed ball in $\R^\ell$ of
radius $r$ centered at $0 \in \R^\ell$. 
We define
$$
\Psi \colon B^{\ell}(r_1+R,0) \to 
  F \ssminus (h^F)^{-1}\left( (1-\varepsilon, 1] \right) 
$$ 
by 
\begin{equation}\label{eq:def-Psi}
\Psi(x) = \begin{cases}
x & \text{ if } \| x \| < r_1;                        \\
\psi^{new}(t,x) 
 & \text{ if } x = y + t \frac{y}{\|y\|} \hs \text{for some} \hs
y \in S_{r_1}^{\ell-1} \\
 & \quad \text{      and } 0 \leq t \leq R.
\end{cases}
\end{equation}

\begin{Lemma} \label{intertwines}
\begin{enumerate}
\item $\Psi$ intertwines the vector field $-\ddr$ 
      with the vector field $\eta^F_{new}$.
\item $\Psi$ is a diffeomorphism. 
\end{enumerate}
\end{Lemma}

\begin{proof} 
Since $\eta^F_{new}$ is defined in a neighbourhood $V_F$ of $F$, the
flow along the vector field $-\eta^F_{new}$ defines a smooth map 
$(-r_1, R+\veps_0) \times S^{\ell-1}_{r_1} \to \R^\ell$, 
$(x,t) \mapsto \psi^{new}(t,x)$, for sufficiently small $\veps_0 > 0$. 
We use the diffeomorphism \(x+t \frac{x}{\|x\|} \mapsto (t,x)\) 
to identify the interior of $B(r_1+R+\veps_0,0)\ssminus \{0\}$
with $(-r_1, R+\veps_0) \times S^{\ell-1}_{r_1}$. Since 
$-\eta^F_{new} \equiv \ddr$ is the standard radial vector field within
the ball of radius $r_1$, 
the restriction of this flow to $B^{\ell}(r_1+R,0) \ssminus \{0\}$ is precisely the map $\Psi$. In particular, $\Psi$ is smooth 
on $B^{\ell}(r_1+R,0) \ssminus\{0\}$.
Since $\Psi$ is the identity map near $0$, $\Psi$ is smooth everywhere.

It follows from the above that $\Psi$ carries the standard
radial vector field $\ddr$ on $B^\ell(r_1+R,0) \ssminus\{0\}$ to
$-\eta^F_{new}$, which implies (1). Moreover, 
Lemma~\ref{normalize} and the definition of $\eta^F_{new}$ 
imply that $\image(\Psi)$ is exactly $F \ssminus
(h^F)^{-1}((1-\varepsilon, 1])$. In particular, the boundary
$S^{\ell-1}_{r_1+R}$ is carried to $Z_\varepsilon$. 

We now show that $\Psi$ is a diffeomorphism. First, by the theory of
ODEs, trajectories of $-\eta^F_{new}$ are disjoint, so $\Psi$ is
injective. Next we claim that $\Psi$ is a local diffeomorphism for all
\(x \in B^{\ell}(r_1+R,0)\).  It suffices to prove that the
differential $d\Psi_x$ is always onto. From the definition of $\Psi$,
the claim is obvious for any point \(x\) in the interior
of $B^\ell(r_1,0)$.  Hence we may assume that \(x \in B^\ell(r_1+R,0)\) 
is of the form   $x = y + t_0 \frac{y}{\|y\|}\) 
for \(y \in S^{\ell-1}_{r_1}$ and $0 \leq t_0 < R$. 
For a fixed $t$, denote by $\psi^{new}(t,\cdot)$ 
the map \(x \mapsto \psi^{new}(t,x).\) By definition of $\Psi$, the image
\((\Psi)_*(T_xS^{\ell-1}_{r_1+t_0})\) is equal to \((\psi^{new}(t_0,
\cdot))_*(T_yS^{\ell-1}_{r_1})),\) and we have already seen above that
$\Psi_*(\ddr) = -\eta^F_{new} = \left(\psi^{new}(t_0,
  \cdot)\right)_*(\ddr)$. Since $\psi^{new}(t_0,\cdot)$ is a local
diffeomorphism, we conclude $\Psi_*$ is also onto at every \(x = y +
t_0 \frac{y}{\|y\|}\), as desired.

Since $\Psi$ is a local diffeomorphism, it is in particular an open
map. Surjectivity now follows from the general fact that a continuous
open map from a nonempty compact space to a connected Hausdorff space
is surjective. Hence $\Psi$ is injective, surjective, and locally a
diffeomorphism, hence a global diffeomorphism as desired.

\end{proof}

We are now in a position to explicitly construct an extension of
$f_{\ell-1}$ to the interior $\relint{F}$ of $F$. We do so by using the
reparametrization to $B(r_1+R,0)$ given by $\Psi$. 
Recall that the function $f_{\ell-1}$ is defined on a neighbourhood $U_{\ell-1}$ of
$\partial F$ in $F$ that contains, by our assumption on $\varepsilon$,
the hypersurface $Z_\veps = \Psi(\del B_{r_1+R})$.  Let
$\Psi^*f_{\ell-1}$ be the pullback of $f_{\ell-1}$ to a neighbourhood
$U_{\del B}$ of the boundary in $B(r_1+R,0)$. Since
\begin{itemize}
\item
$\Psi$ carries $-\ddr$ to $\eta^F_{new}$,
\item
\(\eta^F_{new}\) coincides with \(\eta^F\) 
near $Z_{\varepsilon}$, 
\item
$df_{\ell-1}(\eta^F) < 0$ near $Z_{\varepsilon}$, and
\item
$\alpha_F < \min\limits_{Z_\veps} f_{\ell-1}$,
\end{itemize}
after possibly shrinking the neighbourhood $U_{\partial B}$ 
of $\del B$ in $B(r_1+R,0)$ we may assume that 
\[
d(\Psi^* f_{\ell-1})\left(-\ddr\right) < 0
\quad \text{ on } \quad U_{\partial B},
\]
and that there exists $\gamma \in \R$ such that 
$\alpha_F < \gamma < \inf\limits_{U_{\del B}} \Psi^* f_{\ell-1}$.

We now define a function which, when patched together with $\Psi^*
f_{\ell-1}$ via a partition of unity, will yield the desired
extension. 
Define  
\begin{equation}\label{eq:def zeta}
 \zeta(x) := \alpha_F + (\gamma - \alpha_F) \left(\frac{\|x\|^2}{|r_1
    + R|^2}\right) \text{ \ \ for \ } x \in B(r_1+R,0) \subseteq \R^\ell. 
\end{equation}
Then by construction \(\zeta(x) < \gamma\) for all \(x \in B(r_1+R,0),\)
so 
$$ \Psi^* f_{\ell-1} > \zeta \quad \text{ on } U_{\del B}.$$
Since $\del B(r_1+R,0)$ is compact, 
there exists an $\tilde{R}$ with \(0 < \tilde{R} < r_1 + R\) such that the annulus
$$ \{ x \ | \ \tilde{R} \leq \| x \| \leq r_1+R \} $$
is contained in the neighbourhood $U_{\del B}$
of $\del B(r_1+R,0)$. Let $\rhobar: [0,r_1+R] \to \R$ be a smooth
function such that 
\begin{itemize}
\item
$\rhobar$ is weakly monotone increasing, and 
\item
there exist $R_1$ and $R_2$ such that
$0 < \tilde{R} < R_1 < R_2 < r_1+R$ and 
such that $\rhobar = 0$ 
on $[0,R_1]$ and $\rhobar =1$ on $[R_2,r_1+R]$.
\end{itemize}
Given such a $\rhobar$, define 
$ \rho \colon B(r_1+R,0) \to \R $
by
$ \rho(x) := \rhobar( \| x \| )$.
Note that 
\(\left(-\ddr\right)\rho \leq 0\) by assumption on
$\rhobar$. 
Consider the function 
\begin{equation}\label{eq:extension-on-ball}
\rho \Psi^* f_{\ell-1} + (1-\rho) \zeta \colon \ B^{\ell}(r_1+R,0) \ \to \ \R.
\end{equation}

\begin{Lemma} \label{properties of combination}
The function~\eqref{eq:extension-on-ball} 
\begin{itemize} 
\item is smooth;
\item has a unique critical point in the relative interior
  $\mathring{B}_{r_1+R}$, and this critical point is at the origin $0$;
\item at $0$ it takes the value $\alpha_F$.
\end{itemize}
Moreover, there exists a neighbourhood of \(\partial B(r_1+R,0)\) 
in $\R^\ell$
on which this function agrees with $\Psi^*f_{\ell-1}$. 
\end{Lemma}

\begin{proof}
  The smoothness of~\eqref{eq:extension-on-ball} follows immediately
  from the fact that it is a smooth convex sum of two smooth
  functions.  To show that the only critical point is at $0$, suppose
  that \(x \neq 0.\) We will show that the differential does not
  vanish at $x$.  We compute
\begin{equation*}
\begin{split}
-\ddr \left(\rho \Psi^*f_{\ell-1} + (1-\rho)\zeta\right)
 & = \rho \left( - \ddr \Psi^*f_{\ell-1} \right) 
   + (1-\rho) \left( - \ddr \zeta \right)
   - \frac{d\rho}{dr} \Psi^* f_{\ell-1}
   - \frac{d(1-\rho)}{dr} \zeta \\
 & = \rho \left( - \ddr \Psi^* f_{\ell-1} \right)
   + (1-\rho) \left( - \ddr \zeta \right)
   - \frac{d\rho}{dr} \left( \Psi^* f_{\ell-1} - \zeta \right).  
\end{split}
\end{equation*}
Since $-\ddr \Psi^* f_{\ell-1}  < 0$ where defined (i.e. near the boundary), 
$-\ddr \zeta <0$ where defined and for $x\neq 0$, \  
\(\Psi^*f_{\ell-1} - \zeta > 0\) by construction of $\zeta$, and
$-\ddr \rho \leq 0$ by assumption, the last quantity is always $\leq 0$. 
In fact, by the above, at least one of the first two terms must be
\emph{strictly} negative for any $x \neq 0$. 
Hence the quantity is non-zero and 
we conclude that the points \(x \neq 0\) are not critical points
of~\eqref{eq:extension-on-ball}. On the other hand, for $x=0$, since
$\zeta$ is defined in terms of the norm-square $\|x\|^2$, it is
immediate that $x=0$ is a critical point. 
Finally, in the neighbourhood of \(\partial B{r_1+R}\) where \(\rho \equiv 1\) 
the function~\eqref{eq:extension-on-ball} is equal to $\Psi^* f_{\ell-1}$.
\end{proof}

Now consider the pullback of the function~\eqref{eq:extension-on-ball} to 
$F \ssminus (h^F)\Inv ((1-\veps,1])$ via the diffeomorphism
$\Psi^{-1}$ inverse to $\Psi$. By Lemma~\ref{properties of
  combination}, this pullback agrees with $f_{\ell-1}$ on a
neighbourhood of the boundary $Z_\varepsilon$. Thus we may extend this
pullback to a smooth function $f_{\ell,F}$ on all of $F$ by setting
$f_{\ell,F}:=f_{\ell-1}$ on $(h^F)^{-1}((1-\varepsilon, 1])$. This
function $f_{\ell,F}$ has the properties that 
\begin{itemize} 
\item $f_{\ell,F}$ has a unique critical point at the origin $0 \in F$; 
\item the unique critical point $0$ is a global minimum; and 
\item $f_{\ell,F}(0) = \alpha_F$.
\end{itemize}

We have achieved our goal of extending the initial function $f_{\ell-1}$ 
(after possibly shrinking its domain of definition)
to a function $f_\ell$ that is defined on the entire face $F$. 
Repeating this for every $\ell$ dimensional face,
we obtain a function $f_{\ell,F}$ on each $\ell$ dimensional face $F$
such that $F_{\ell,F}$ agrees with $f_{\ell-1}$
on some neighbourhood $U_F$ of $\del F$ in $F$.
We would like to extend these functions further
to a whole ($n$-dimensional) neighbourhood of the $\ell$-skeleton of $\Delta$
whole ensuring that the required conditions (f1)--(f3) continue to hold.

Below, we extend each $f_{\ell,F}$ to a smooth function $\tilde{f}_{\ell,F}$ 
on a neighbourhood of $F$ in $\Delta$ 
such that each function $\tilde{f}_{\ell,F}$ agrees with $f_{\ell-1}$ 
on a neighbourhood of $\partial F$ in $\Delta$
and satisfies the derivative condition 
\[
d(\tilde{f}_{\ell,F})(\xi_j) < 0
\]
on $F$ for all $j$ such that $F \subset \sigma_j$.
Before proceeding, and supposing for a moment that such functions 
$\tilde{f}_{\ell,F}$ can be constructed, we first explain how this
completes the proof of the statements (A) and (B) 
of Proposition~\ref{prop:bump-function new}. 
Choose open subsets $V_F$ of $\Delta$ such that $F
\subseteq V_F$ and $V_F \subseteq U_F$. By shrinking the $V_F$
if necessary, we may without loss of generality assume that for two
distinct nontrivially intersecting faces $F \cap F' \neq \emptyset$,
the functions $\tilde{f}_{\ell,F}$ and $\tilde{f}_{\ell, F'}$ agree on
the overlap of the open sets, i.e. \(\tilde{f}_{\ell,F} \vert_{V_F
  \cap V_{F'}} = \tilde{f}_{\ell,F'} \vert_{V_F \cap V_{F'}},\) since
the $\{\tilde{f}_{\ell,F}\}$ are assumed to agree on a neighbourhood in
$\Delta$ of the $(\ell-1)$-skeleton (which contains $F \cap F'$). With
this understood, we may therefore define a smooth function $f_\ell$ on
the open neighbourhood $\bigcup_F V_F$ in $\Delta$ of the
$\ell$-skeleton by \(f_{\ell} \vert_{V_F} := \tilde{f}_{\ell,F}.\)
By construction, $f_{\ell}$ satisfies the properties (f1)-(f3)
listed above. This then completes the inductive step and hence the
proof.

Hence, to complete the proof of statements (A) and (B) of Proposition~\ref{prop:bump-function
  new}, it remains only to construct these extensions
$\tilde{f}_{\ell,F}$. We begin by choosing a convenient (non-linear)
coordinate chart. Recall that we are assuming that the affine span
of the face $F$ is $\R^{\ell}$, embedded in $\R^n$ as $\R^{\ell}
\times \{ 0 \}^{n-\ell}$.  Fix smooth extensions of the vector fields
$\xi_1,\ldots,\xi_{n-\ell}$ to $\R^{\ell}$.  By Lemma~\ref{lemma:xi},
the vectors \( \xi_1\vert_x , \ldots , \xi_{n-\ell}\vert_x \) are
linearly independent and span a complementary subspace to $\R^{\ell}
\subset \R^n$ at each point $x \in F$. Thus, the differential of the
map
\begin{equation}\label{eq:def varphi}
\begin{split}
\varphi \colon \R^{\ell} \times \R^{n-\ell} & \to \R^n  \\ 
 (x,y_1,\ldots,y_{n-\ell}) & \mapsto 
   x + y_1 \xi_1\vert_x + \ldots + y_{n-\ell} \xi_{n-\ell}\vert_x \\
\end{split}
\end{equation}
is a linear isomorphism at each point $x$ of $F$.
This implies that there exists $\veps > 0$ 
and a neighbourhood $W_F$ of $F$ in $\R^{\ell}$
such that the map $\varphi$ restricts to a diffeomorphism 
of $W_F \times (-\veps,\veps)^{n-\ell}$ with an open subset of $\R^n$
that carries a neighbourhood of $F \times \{ 0 \}^{n-\ell}$
in $F \times \R_{\geq 0}^{n-\ell}$ to a neighbourhood of 
$F$ in $\Delta$. In the argument below, we therefore use these coordinates $(x,y_1, \ldots,
y_{n-\ell}) \in F \times (-\varepsilon, \varepsilon)^{n-\ell} \subseteq F
\times \R^{n-\ell}$ to parametrize a neighbourhood of $F$ in $\R^n$.

By assumption on $f_{\ell-1}$, there exists some $0 < \varepsilon' <
\varepsilon$ and a neighbourhood $W_1$ of $\del F$ in $F$ such that $\varphi^* f_{\ell-1}$ is defined on 
\(W_{\partial F} := W_1 \times (-\varepsilon', \varepsilon')^{n-\ell}.\) 
Let $W_2$ be the relative interior of $F$ in $\R^{\ell}$.
Then \(\{W_1, W_2\}\) form an open cover of $F$.
Let \(\{\rho_1, \rho_2\}\) be a partition of
unity subordinate to this cover. 
For $(x,y) \in F \times (-\veps',\veps')^{n-\ell}$, we define 
\begin{equation}\label{eq:def-tildef-ellF}
\tilde{f}_{\ell,F}(\varphi(x,y)) 
  := \rho_1(x) \varphi^* f_{\ell-1}(x,y) + \rho_2(x) \left(
  \varphi^* f_{\ell,F}(x) - \sum_{i=1}^{n-\ell} y_i \right).
\end{equation}
Since $\rho_2$ is supported in $W_2$,
there exists a neighbourhood of $\partial F$ in
$\R^{\ell}$ on which $\rho_2 \equiv 0$. This means that on a
neighbourhood in $\R^n$ of the $(\ell-1)$-skeleton near $F$,
\ \(\tilde{f}_{\ell,F} \equiv f_{\ell-1},\)  as desired. 
The only remaining claim needing proof is that $\tilde{f}_{\ell,F}$ satisfies the
derivative condition \(d(\tilde{f}_{\ell,F})(\xi_j) < 0\) 
on $F$ for $1 \leq j \leq n-\ell$.
Since $\tilde{f}_{\ell,F}$ agrees with
$f_{\ell-1}$ on a neighbourhood of $\partial F$, it suffices to check
this condition on \(\relint(F).\) 
Since $\varphi_* \left(\frac{\del}{\del y_j}\right) = \xi_j$ by
construction of $\varphi$, and the three functions 
$\rho_1, \rho_2,$ and $f_{\ell,F}$ are independent of the $y_j$
variables, we have 
\begin{align*}
d(\tilde{f}_{\ell,F})(\xi_j|_{\varphi(x,0)})
 & = \frac{\partial}{\partial y_j} \vert_{(x,0)}
\left[ \rho_1(x) (\varphi^* f_{\ell-1})(x,y) + \rho_2(x) \left( f_{\ell,F}(x) -
    \sum_{i=1}^{n-\ell} y_i \right) \right] \\
 & = \rho_1(x) (df_{\ell-1})(\xi_j|_x) - \rho_2(x) \\
 & < 0,
\end{align*} 
as desired, since \((df_{\ell-1}(\xi_j \vert_x) < 0\) by assumption and at least
one of $\rho_1$ or $\rho_2$ must be positive at any \(x \in F.\) 
This completes the proof of the claim and hence of the statements (A)
and (B) in 
Proposition~\ref{prop:bump-function new}.

It remains to justify the statement (C) of
Proposition~\ref{prop:bump-function new}. In a small enough
neighbourhood of the prescribed critical point
$x_F$ of a face $F$, the function defined
in~\eqref{eq:extension-on-ball} has the property that $\rho \equiv 0$
and hence, along the face $F$, is equal to $\zeta$. The explicit
formula for $\zeta$ in~\eqref{eq:def zeta} shows that, in appropriate
coordinates along the face $F$, the function $\zeta$ 
is quadratic in
the coordinates up to an affine translation in $\R$, as
desired. Moreover, the explicit formula for $\tilde{f}_{\ell,F}$
in~\eqref{eq:def-tildef-ellF} shows that with respect to the
coordinates $(x,y)$ in~\eqref{eq:def varphi} is linear in the
coordinates $y_j$ and decreases in the directions pointing into the
polytope $F$, again as desired. Since the vector fields $\xi_j$ are
also arranged to be constant sufficiently near $x_F$, the coordinates
$(x,y)$ are in fact affine. This concludes the proof of part (C) 
of Proposition~\ref{prop:bump-function new}
and hence of the entire proposition.

\end{document}